\documentclass[10pt]{article}
\usepackage{amsmath, amsthm, setspace, graphicx, subfig}  
\usepackage{latexsym, amsfonts, amssymb}
\usepackage[all,cmtip]{xy} 
\usepackage{subfig}
\usepackage[titles,subfigure]{tocloft}
\usepackage{bbold} 

\usepackage{xcolor}

\newtheorem{thm}{Theorem}
\newtheorem{cor}[thm]{Corollary} 
\newtheorem{lem}[thm]{Lemma} 
\newtheorem{prop}[thm]{Proposition}
\newtheorem{defn}[thm]{Definition}
\newtheorem{remark}[thm]{Remark} 
\newtheorem{remarks}[thm]{Remarks}

\newtheorem{example}[thm]{Example}
\newtheorem{examples}[thm]{Examples}

\def\a{\alpha}
\def\b{\beta} 
 
\def\d{\delta}
\def\e{\epsilon}

\def\l{\lambda} 
\def\s{\sigma} 
\def\t{\tau} 
\def\o{\omega}
\def\ch{\it{Ch}}
\def\fr{\it{fr}}

\def\G{\Gamma}
\def\D{\Delta}
\def\L{\Lambda} 
\def\O{\Omega}    

\def\R{\mathbb{R}}

\def\p{\partial} 
\def\i{\infty}
\def\cal{\mathcal}
\def\B{\cal{B}}

\def\hB{\hat{\cal{B}}}

\def\P{\cal{P}}
\def\A{\cal{A}}
\def\per{\it{per}}
\def\det{\it{det}} 
\def\vec{\it{vec}}
\def\supp{\it{supp}}
\def\<{\langle}
\def\>{\rangle}

\def\rotateminus{\reflectbox{\rotatebox[origin=c]{155}{\hspace{.6pt}-}}}
\def\Xint#1{\mathchoice
{\XXint\displaystyle\textstyle{#1}}%
{\XXint\textstyle\scriptstyle{#1}}%
{\XXint\scriptstyle\scriptscriptstyle{#1}}%
{\XXint\scriptscriptstyle\scriptscriptstyle{#1}}%
\!\int}
\def\XXint#1#2#3{{\setbox0=\hbox{$#1{#2#3}{\int}$}
\vcenter{\hbox{$#2#3$}}\kern-.5\wd0}}

\def\cint{\Xint \rotateminus } 
\renewcommand{\labelenumi}{(\alph{enumi})} 

\numberwithin{thm}{subsection} 
\numberwithin{equation}{section} 

\begin{document}

	\title{Operator Calculus of Differential Chains and Differential Forms }
	  \author{J. Harrison 
	  \\Department of Mathematics 
	  \\University of California, Berkeley }
  \date{}
	
	\maketitle
 
	\begin{abstract} 	
		Differential chains are a proper subspace of de Rham currents given as an inductive limit of Banach spaces endowed with a geometrically defined strong topology.  Boundary is a continuous operator, as are operators that dualize to Hodge star, Lie derivative, pullback and interior product. Partitions of unity exist in this setting, as does Cartesian wedge product.  Subspaces of finitely supported Dirac chains and polyhedral chains are both dense, leading to a unification of the discrete with the smooth continuum.  We conclude with an application generalizing a simple version of the Reynolds' Transport Theorem to rough domains.
	\end{abstract}       
 	   
\section{Introduction} 
\label{sec:introduction} 
A new analytic and topological formulation of a \( k \)-dimensional domain of integration in an open subset \( U \) of \( \R^n \) is presented by way of a chain complex of topological vector space of domains with a number of useful properties.  In particular, the topological vector space \( \hB \) of domains called ``differential chains'' has good compactness properties, making it useful for solving existence problems in the calculus of variations (\cite{plateau}, \cite{plateau10}).   The space \( \hB \) is endowed with the strong (Mackey) topology (\ref{sub:history}) which has an explicit definition.  It includes analogues of classical \( k \)-dimensional objects in \( \R^n \), but excludes many pathologies seen in currents, making it easier, for example, to obtain physically realistic solutions for problems in the calculus of variations.   

Let \( \L^k(\R^n) \) be the \( k \)-th exterior power of \( \R^n \).  We call \( (p;\a) \) a \( k \)-\emph{\textbf{element}} if \( p\in U \) and \( \a \in \L_k(\R^n) \).  When \( \a \) is a simple \( k \)-vector, \( (p;\a) \) can be viewed as either a de Rham current or as a single-point geometric structure. Its ``boundary'' is a \( (k-1) \)-element \( (p; \b) \) where \( \b \) is a non-simple \( (k-1) \)-vector if \( \a \ne 0 \).    The collection of all \( k \)-elements generates a vector space \( \A_k(U) \) whose elements are called ``Dirac \( k \)-chains''.   A one-parameter family of norms on the space  \( \A_k(U) \) arises naturally. The inductive limit of the resulting Banach spaces endowed with the locally convex inductive limit topology \( \tau \) is the space \(  \hB_k(U) := ({\cal A_k}(U), \t) \) of differential chains.  One can show that \( \tau \) is the finest topology on Dirac chains satisfying three properties\footnote{It is this universal property which led to the name ``differential chain''.} \cite{topological}:  \( \t \) is coarser than the mass norm topology on \( \A_k(U) \),   bounded operators are continuous,  and a geometric operator ``prederivative'' (\S\ref{sub:prederivatve_wrt}), whose dual is Lie derivative, is bounded.  From these simple axioms follow many properties. For example, the boundary operator is continuous.    

Another useful property is that Dirac and polyhedral chains are dense in \( \tau \)  (Theorem \ref{thm:cellsmore}). As a pleasant consequence, we can build everything from a starting point of Dirac chains, as opposed to polyhedral chains, and without reference to differential forms.  This way, our definitions and proofs are often much simpler than their classical counterparts for smooth domains.   An example is the perpendicular complement operator \( \perp \) (\S\ref{sub:perp}) whose dual is Hodge star.  The definition of \( \perp \) using Dirac chains (Definition \ref{perp}) is trivial, while its definition using polyhedral chains \cite{hodge} is not. 
          
The topological dual of \( \hB_k(U) \) is topologically isomorphic to the Fr\'echet space \( \B_k(U) \) of differential \( k \)-forms satisfying that each derivative of each coefficient function is uniformly bounded over points in \( U \). (See Theorem \ref{thm:derham} below and \cite{harrison2} for the first proof which used polyhedral chains.)  The \( 0 \)-forms give differential chains a module structure, making possible partitions of unity for differential chains Theorem \ref{thm:pou}.    Other geometric operators on differential chains have explicit definitions using limits of Dirac chains, but are more familiar in their dual versions on differential forms, e.g.,  wedge and interior products of a \( k \)-vector field \( X \) with a differential form.  The ``predual'' versions provide similar multiplication\footnote{Differential chains do not form an algebra, but they are a coalgebra and coproduct dualizes to wedge product. A preprint is in preparation.} and division of a \( k \)-vector field with a differential chain.  These are ``extrusion'' \( E_X \) (\S\ref{ssub:extrusion}) and ``retraction'' \( E_X^\dagger \) (\S\ref{sub:retraction}). 
There operators are used to extend Stokes \ref{thm:stokes}, Gauss--Green \ref{thm:div}, and Kelvin--Stokes \ref{thm:curl} to differential chains.
Using a new geometrically defined Laplace operator on differential chains, paired with the classical Laplace operator on differential forms, we obtain a new higher order divergence theorem \ref{cor:lapl} which equates net flux of the ``orthogonal differentials'' of a \( k \)-vector field across an oriented boundary with net interior ``higher order divergence''.

In \S\ref{sec:fundamental_theorems_of_calculus_for_chains_in_a_flow} we prove two new fundamental theorems for differential chains in a flow:  Given a differential \( k \)-chain \( J \) and the flow \( \phi_t \) of a smooth vector field \( V \), we define a single differential \( k \)-chain \( [J_t]_a^b \) satisfying \( \p([J_t]_a^b) = [\p J_t]_a^b \) (see Figure  \ref{fig:MovingChains}).  Denoting \( J_t:={\phi_t}_* J \), we prove Theorem \ref{thm:Lieder}: \[ \cint_{J_b} \o - \cint_{J_a} \o = \cint_{[J_t]_a^b} L_V \o    \text{  (Fundamental theorem for flowing chains)}\] and Corollary \ref{cor:exactL}:  \[ \cint_{[J_t]_a^b} d L_V \o = \cint_{\p J_b} \o - \cint_{\p J_a} \o \text{  (Stokes' theorem for flowing chains,)} \] given a form \( \o \) of matching dimension. 

This is followed by a generalization of the Leibniz rule (Theorem \ref{thm:Leibniz}) from which we obtain an extension of the classical Differentiation of the Integral to differential chains (Theorem \ref{cor:differintegral}).  An immediate consequence of Corollary \ref{cor:differintegral} and Cartan's magic formula for differential chains (Theorem \ref{thm:Gpush}) is a generalization of Reynold's transport theorem Corollary \ref{cor:Reynolds}.  

\subsection{Dual pairs and the Mackey topology} \label{sub:history}

 A \emph{\textbf{dual pair}} \( (X,Y) = (X,Y, \langle , \rangle) \) is a pair of vector spaces \( (X,Y) \),  together with a bilinear form \( \<\cdot,\cdot\>: X \times Y \to \R\) called the \emph{\textbf{evaluation form}}.   A \emph{\textbf{dual topology}} on \( X \) is a locally convex topology \( \tau \) so that the continuous dual  \(  (X, \tau)' \) is isomorphic to \( Y \).  The Mackey–-Arens theorem characterizes all possible dual topologies on a locally convex space. At one extreme, the coarsest dual topology is the weak topology, and at the other is the finest dual topology called the \emph{\textbf{Mackey topology}}.   The Mackey topology on Dirac chains \( \A \) in the dual pair \( (\A,\B) \) contains representatives of surfaces with triple junctions and non-orientable surfaces.    Moreover,  both types of surfaces can be cycles. Compactness is an easy consequence of our explicit version of the Mackey topology on Dirac chains.   The tightness of the Mackey topology minimizes pathologies and makes it possible to prove soap film regularity.  
		
The Mackey topology is canonical and beautiful in its abstract conception, but for practical applications it is necessary to find some explicit formulation of it for a given dual pair. For example, the space of polyhedral \( k \)-chains can be paired with any space of continuous differential \( k \)-forms via integration, yielding a dual pair.   The Mackey topology on polyhedral \( k \)-chains paired with bounded differential \( k \)-forms turns out to be the mass norm. When paired with Lipschitz forms, the Mackey topology on polyhedral chains is Whitney's sharp norm, originally called the ``tight norm'' topology in his 1950 ICM lecture \cite{whitneyicm}.  However, the boundary operator is not continuous in the sharp norm topology, which is something Whitney needed for his study of sphere bundles\footnote{According to his own account in \cite{whitneytopology}, Whitney wanted to solve a foundational question about sphere bundles and this is why he developed Geometric Integration Theory.  Steenrod \cite{steenrod} solved the problem first, though, and Whitney stopped working on GIT (\cite{whitneytopology}). However, Whitney's then-student Eells was interested in other applications of GIT in analysis, and for a time he thought he had a workaround (\cite{whitney}, \cite{eells}), but the lack of a continuous boundary operator halted progress.}. He then defined the flat norm which does have a continuous boundary operator, and his student Wolfe \cite{wolfethesis} identified the topological dual to be that of flat chains. The flat norm has found important applications in geometric measure theory but the scope of applications is limited since the Hodge star of a flat form is not itself flat, and coordinate functions of flat forms are not flat.  There is no flat divergence theorem and there is no sharp Stokes' theorem.
  
Given these limitations, it became clear that a new dual pair \( (X,Y, \langle , \rangle) \) was needed where \( Y \) is a topological vector space of differential forms, and \( X \) is the space of Dirac chains \( \A \)( or polyhedral chains \( \P \)).  The first problem was to find a useful space of differential forms for which the Mackey topology \( \t \) on \( \A \)  could be made explicit.  Moreover the duals of all of the major operators of differential forms necessary for a full calculus, e.g, exterior derivative, Lie derivative, pullback, wedge product, and interior product,   should restrict to closed, geometrically defined operators on \( X \subset Y' \), not just boundary and Hodge star. Furthermore, \( X \), endowed with the Mackey topology, should be Hausdorff, separable, and complete.   Differential chains accomplish these goals.    
                                               
Acknowledgements: I am grateful to Morris Hirsch for his support of this work since its inception.  His helpful remarks and historical insights have been of great benefit as the theory evolved from a few geometrically appealing, but ad hoc, and difficult constructions on polyhedral chains to its current state with its  universal property \cite{topological},  its algebraically concise constructions on Dirac chains, and its applications to Plateau's problem \cite{plateau10}.  I also am very happy to thank James Yorke who has contributed to useful discussions throughout the development of this work, while Robert Kotiuga, Alain Bossavit,  Alan Weinstein, and Steven Krantz are thanked for their early interest and support.  Harrison Pugh's contributions have been significant starting with his Harvard thesis \cite{thesis}.  Our joint paper \cite{topological} is fundamental and necessary for results of this paper to hold in the inductive limit, and gives the entire theory a solid mathematical footing in the classical setting of topological vector spaces.    Finally, I am indebted to the anonymous reviewer for providing numerous helpful comments.  
             
\section{Differential chains of class \( \B^r \)} 
\label{sub:top}

Let \( U \) be an open subset\footnote{An extension of the theory to open subsets of a Riemannian  manifold \( M \) is in progress.  See \S\ref{sec:differential_chains_in_manifolds} for the main ideas.} of \( \R^n \). For \( 0\le k\le n \), define a ``Dirac k-chain''\footnote{Previously called  a ``pointed \( k \)-chain''. We prefer not to use the term ``Dirac current'' since it has various definitions in the literature, and the whole point of this work is to define calculus starting with chains, setting aside the larger space of currents.} in \( U \) to be a finitely supported section of the \( k \)-th exterior power of the tangent bundle of \(U\).       

\begin{defn}   
	\label{diracchains}
	For \( U \) open in \( \R^n \), let \( \A_k(U) \) denote the vector space of finitely supported functions \( U \to \L^k(\R^n) \) where \( \L_k(\R^n) \) is the \( k \)-power of the exterior algebra. We call \( \A_k(U) \)  the space of \emph{\textbf{Dirac \( k \)-chains in \( U \)}}.
\end{defn}
			We write an element \( A \) of \( \A_k(U) \) using formal sum notation \( A=\sum (p_i; \a_i) \) where \( p_i \in U \) and \( \a_i \in \L^k(\R^n) \). We call \( (p;\a) \) a \( k \)-\emph{\textbf{element in \( U \)}} if \( \a \in \L^k(\R^n) \) and \( p \in U \). If \( \a \) is simple, then \( (p;\a) \) is a \emph{\textbf{simple \( k \)-element}} in \( U \).  
		 The \emph{\textbf{support}} of a Dirac chain \(  \sum_{i=1}^s (p_i; \a_i)   \) is defined as \( \supp(\sum_{i=1}^s (p_i; \a_i) ) :=  \cup_{i=1}^s p_i \).  We say that a Dirac chain \( A \) is \emph{\textbf{supported in}} \( X \subset U \) if \( \supp(A) \subset X \). Let \( \A_k(U) \) be the vector space of Dirac \( k \)-chains supported in \( U \) and \( \B_k(U) \) the Fr\'echet space of differential \( k \)-forms defined on \(U \subseteq \R^n \) each with uniform bounds on each of its directional derivatives.  Then \( (\A_k(U), \B_k(U)) \) is a dual pair.     In this paper,  we define a family of norms \( \|\cdot\|_{B^r} \) on \( \A_k(U) \) so that the resulting Banach spaces\footnote{\( \hB_k^1 \) is isomorphic to the sharp space of polyhedral chains \cite{whitney}.}  \( \hB_k^r(U) := \overline{ (\A_k(U), \|\cdot\|_{B^r})} \)  form a direct system.    The inductive limit \( \hB_k(U) := \varinjlim \hB_k^r(U) \) is endowed with the inductive limit topology, which turns out to be the Mackey topology \( \t_k(\A_k(U), \B_k(U)) \) on \( \A_k(U) \) (\cite{topological}).   We prove that \( \hB_k(U) \)  a well-defined Hausdorff locally convex topological vector space, indeed a \( (DF) \)-space (\cite{AG}, \cite{topological}).  Note that \(  (\hB_k(U))' =   \B_k(U) \), by definition of the Mackey topology. We call elements of elements of  \( \hB_k(U) \) ``differential \( k \)-chains of class \( B \) in \( U \).''           
 
  We next present our explicit definition of the Mackey topology \( \t_k \) on \( \A_k(U) \) with topological dual \( \B_k(U) \).   The construction begins with a decreasing family of norms on \( \A_k(U) \), starting with the mass norm:
             
\subsection{Mass norm} 
\label{sub:mass_norm}

\begin{defn}\label{def:mass}
	An inner product \( \<\cdot,\cdot\> \) on \( \R^n \) determines the mass norm on \( \L^k(U) \) as follows: Let \[ \<u_1 \wedge \cdots \wedge u_k, v_1 \wedge \cdots \wedge v_k \> := \det(\<u_i,v_j\>). \]  (We sometimes use the notation \( \<\a,\b\>_\wedge = \<\a,\b\> \).) 	                                       The \emph{\textbf{mass}} of a simple \( k \)-vector \( \a = v_1 \wedge \cdots \wedge v_k \) is defined by \( \|\a\| := \sqrt{\<\a,\a\>} \). The \emph{\textbf{mass}} of a \( k \)-vector \( \a \) is \( \|\a\| := \inf\left\{\sum_{j=1}^{N} \|(\a_i)\| : \a_i \mbox{ are simple, } \a = \sum_{j=1}^N \a_i \right\}. \) Define the \emph{\textbf{mass}} of a \( k \)-element \( (p;\a) \) by \(\|(p;\a)\|_{B^0} := \|\a\| \).
   The \emph{\textbf{mass}} of a Dirac \( k \)-chain \( A = \sum_{i=1}^m (p_i; \a_i) \in \A_k(U) \) is given by \[ \|A\|_{B^0} := \sum_{i=1}^m \|(p_i; \a_i)\|_{B^0}. \]
\end{defn}  Another choice of inner product leads to an equivalent mass norm.
   
\subsection{Difference chains} 
\label{sub:difference_chains}

\begin{defn}\label{def:morenotation}
	 Given \( u\in \R^n \) and a \( k \)-element \( (p;\a)\in \A_k(\R^n) \), let \( T_u(p;\a) := (p+u;\a) \) be \emph{\textbf{translation}} through \( u \), and \( \D_u(p;\a) := (T_u -I)(p; \a) \). Extend both operators linearly to \( \A_k(\R^n) \). If \( u_1,\dots,u_j\in \R^n \), define the \emph{\textbf{\( j \)-difference \( k \)-chain}} \( \D_{\{u_1,\dots,u_j\}}(p; \a) := \D_{u_j}\circ\dots\circ\D_{u_1}(p; \a) \). The operators \( \D_{u_i} \) and \( \D_{u_k} \) commute, so \( \D_{\{u_1,\dots,u_j\}} \) depends only on the set \( \{u_1,\dots , u_j\} \) and not on the ordering. We call \( j \) the \emph{\textbf{order}}
	  of \( \D_{\{u_1,\dots,u_j\}} \).  We say \( \D_{\{u_1,\dots,u_j\}} (p;\a) \) is \emph{\textbf{inside}} \( U \) if the convex hull of \( \supp(\D_{\{u_1,\dots,u_j\}} (p;\a))  \) is a subset of \( U \).
\end{defn}   
To simplify our next definition, let \( \D_{\emptyset} (p;\a) := (p;\a) \), let \( s^j:=\{u_1,\dots,u_j\} \) denote a set of \( j \) vectors in \( \R^n \), and let \( s^0=\emptyset \). Let \( \|s^j\|=\|u_1\|\cdots\|u_j\| \), let \( \|s^0\|=1 \), and let \( |\D_{s^j}(p;\a)|_{B^j}=\|s^j\| \|\a\| \).               

\subsection{Norms on Dirac chains} 
\label{sub:norms_on_dirac_chains}
                                                                           
\begin{defn}\label{def:norms} 
	For \( A \in \A_k(U)\) and \( r \ge 0 \), define the norm \[ \|A\|_{B^r} = \|A\|_{B^{r,U}} := \inf \left \{ \sum_{i=0}^m \|\s_i^{j(i)}\|\|\a_i\|\,:\, A = \sum_{i=0}^m \D_{\s_i^{j(i)}}(p_i;\a_i) ,\,\, 0\le j(i)\le r \right\}\] where \( p_i \in U \), \( \a_i \in \L_k(\R^n) \),  \( m \) is arbitrary, and \( \D_{\s_i^{j(i)}}(p_i;\a_i) \) is inside \( U \).  That is, we are taking the infimum over all ways to write \( A \) as a finite sum of difference chains, of order \( \le r \), each inside \( U \).          
\end{defn} 
               
\begin{remarks} \mbox{} 
	\begin{itemize} 
		\item Although it is clear that \( \|\cdot\|_{B^r} \) is a semi-norm,  it is not obvious that it is a norm on \( \A_k(\R^n)\); this is proved in Theorem \ref{thm:normspace}. 
		\item The  \( B^r \) norms decrease as \( r \) increases, and \( \|A\|_{B^r} < \i \) for all \( A \in \A_k(\R^n)\) since \( \|A\|_{B^r} \le \|A\|_{B^0} < \i \).        
		\item It is not important to know the actual value of \( \|A\|_{B^r} \) for a given chain \( A \).  Well-behaved upper bounds suffice in our proofs and examples, and these are not usually hard to find.  The norms depend superficially on the choice of inner product since different choices lead to comparable norms, and thus to topologically isomorphic spaces. 
		\item Many of the results in this paper hold for \( r = 1 \) where the forms are Lipschitz and the chains are of class \( B^1 \).  Some might be most interested in \( 0 \le r \le 1 \) since it takes so little to define the norms,  and both smooth and rough chains can be found, while others might find the inductive limit more useful because of its operator algebra.  
	\end{itemize} 
\end{remarks}  

\section{Differential forms of class \( B^r \)} 
\label{sub:differential_forms_of_class}

Elements \( \o \in (\A_k(U))^* \) act on Dirac chains.  In this paper we assume that \( \o \) is bounded, and we call it a \emph{\textbf{differential form}}.  Let \( \|\o\|_{B^0} \) denote the sup norm of \( \o \).  

\begin{defn}\label{formsupport}
The \emph{\textbf{support}} of a differential form \( \o \) is given by \[ \supp(\o) := \overline{\{p\in U \,:\,  \o(p;\a) \ne 0 \mbox{ for some } \a \in \L_k \}}. \]	
\end{defn} 
                                                       
\begin{defn}\label{eq:Bj} 
Let \( |\o|_{B^0} := \|\o\|_{B^0} \). For \( j \ge 1 \), define \[   |\o|_{B^j} =  |\o|_{B^{j,U}} := \sup \{ |\o(\D_{\s^j}(p;\a))| \,:\,  \|\s\|\|\a\| = 1,   \D_{\s^j}(p;\a) \mbox{ is inside }  U  \}, \] 
and, for \( r \ge 0 \), define \[ \|\o\|_{B^r} = \|\o\|_{B^{r,U}} := \max\{|\o|_{B^{0,U}}, \dots, |\o|_{B^{r,U}} \}.  \]
\end{defn}    

We always denote differential forms by lower case Greek letters such as \( \o, \eta \) and differential chains by upper case Roman letters such as \( J, K \), so there is no confusion when we write \( \|\o\|_{B^{r}} \) or \( \|J\|_{B^{r}} \).   

We say that \( \o \in (\A_k(U))^* \) is a differential form of \emph{\textbf{class}} \( B^r \) if \( \|\o\|_{B^r} < \i \).   Denote the space of differential \( k \)-forms of class \( B^{r} \) by \( \B_k^r(U) \). Then \( \|\cdot\|_{B^{r}} \) is a norm on \( \B_k^r(U) \). (It is straightforward to see that \( |\o|_{B^{j}} \) satisfies the triangle inequality and homogeneity for each \( j \ge 0 \).  If \( \o \ne 0 \), there exists \( (p;\a) \) such that \( \o(p;\a) \ne 0 \). Therefore \( |\o|_{B^0} > 0 \), and hence \( \|\o\|_{B^{r}} > 0 \).) In \S\ref{sub:isomorphism_theorem} we show that \( \B_k^r(U) \) is topologically isomorphic to \( (\A_k(U), \|\cdot\|_{B^{r}})' \).  

\begin{lem}\label{lem:ineq} 
	Let \( A \in \A_k(U)\) and \( \o \in (\A_k(U))^* \) a differential \( k \)-form. Then \[ |\o(A)| \le \|\o\|_{B^{r}} \|A\|_{B^{r}}  \]for all \( r \ge 0 \).
\end{lem}

\begin{proof}
	Let \( A \in \A_k(U)\) and \( \e > 0 \). By the definition of \( \|A\|_{B^{r}} \), there exist \( \s_i^{j(i)}\) with \( j(i)\in \{0,\dots, r\} \) and \( k \)-elements \( (p_i; \a_i) \), \(i = 1, \dots, m\), such that \( A = \sum_{i=1}^m \D_{\s_i^{j(i)}}(p_i;\a_i) \) and \( \|A\|_{B^{r}} > \sum_{i=1}^m \|\s_i^{j(i)}\|\|\a_i\| - \e. \) The result follows since
	\begin{align*}
 	|\o(A)| \le \sum_{i=1}^m |\o(\D_{\s_i^{j(i)}}) | \le \sum_{i=1}^m |\o|_{B^{j(i)}} \|\s_i^{j(i)}\|\|\a_i\| \le \|\o\|_{B^{r}} (\|A\|_{B^{r}} + \e).
	\end{align*}  
\end{proof}  

\subsection{Differential forms of class  $C^{r-1+Lip}$} 
\label{sub:differential_forms_of_class_}

\begin{defn}
	If \( \o \) is \( r \)-times differentiable, let \[ |\o|_{C^{j,U}} := \sup \left\{\frac{|L_{\s^j}\o(p;\a)|}{\|\s\|\|\a\|},  \D_{\s^{j}}(p;\a) \mbox{ is inside } U \right\}, 0 \le j \le r, \] and \( \|\o\|_{C^r} = \|\o\|_{C^{r,U}} := \max\left\{|\o|_{C^{0,U}}, |\o|_{C^{1,U}}, \dots, |\o|_{C^{r,U}} \right \} \), where \(L_{\s^j}\) denotes the \(j\)-th directional derivative in the directions \( u_1,\dots,u_j\). 
\end{defn} 
                                                                  
\begin{lem}\label{lem:est} 
	If \( \o \) is an exterior form with \( |\o|_{C^{r}} < \i \), then \( \left| \o(\D_{{\s^r}}(p;\a)) \right| \le \|\s\|\|\a\||\o|_{C^{r}} \) for all \( r \)-difference \( k \)-chains \( \D_{{\s^r}} (p;\a) \) where \( \a \) is simple and \( r \ge 1 \).
\end{lem}

\begin{proof} 
	By the mean value theorem, there exists \( q = p + su \) such that \[\frac{\o(p +tu;\a) - \o (p;\a)}{\|u\|} = L_u \o (q;\a). \] It follows that \( \left|\o(\D_u(p;\a)) \right| \le \|u\|\|\a\||\o|_{C^{1}} \).

	The proof proceeds by induction. Assume the result holds for \( r \) and \(|\o|_{C^{r+1}} < \i \). Suppose \( \s = \{u_1, \dots, u_{r+1} \} \), and let \( \hat{\s} = \{u_1,\dots, u_r\} \). Since \( |\o|_{C^{r+1}} < \i \), we may apply the mean value theorem again to see that
	\begin{align*} 
		\left|\o(\D_{{\s^{r+1}}}(p;\a)) \right| &= \left| ( T_{u_{r+1}}^* \o - \o)(\D_{\hat{\s}^r}(p;1)) \right| \\&\le \|u_1\| \cdots \|u_r\| |T_{u_{r+1}}^* \o - \o |_{C^{r}} \\& \le \|\s\||\o|_{C^{r+1}}. 
	\end{align*} 
\end{proof}

\begin{defn}
	Let \[ |\o|_{Lip,U} := \sup_{u \ne 0} \left\{ \frac{|\o(p+u;\a) - \o(p;\a)|}{\|u\|}\,:\, \|\a\|= 1, p(p+u) \subset U \right\}. \] If \( \o \) is \( (r-1) \)-times differentiable, let \[ |\o|_{C^{r-1+Lip},U} := \sup_{\|u_i\|=1} \{ |L_{u_{r-1}} \circ \cdots \circ L_{u_1}\o|_{Lip} \}, \] and \(  \|\o\|_{C^{r-1+Lip}} = \|\o\|_{C^{r-1+Lip},U} := \max\{|\o|_{C^0}, |\o|_{C^{1,U}}, \dots, |\o|_{C^{r-1},U}, |\o|_{C^{r-1+Lip},U} \} \).
\end{defn}

\begin{prop}\label{prop:oncemore} 
	If \( \o \in \B_k^r(U) \) is a differential \( k \)-form and \( r \ge 1 \), then \( \o \) is \( r \)-times differentiable and its \( r \)-th order directional derivatives are Lipschitz continuous with 	\( \|\o\|_{B^{r}} = \|\o\|_{C^{r-1+Lip}} \).
\end{prop}

This is a straightforward result in analysis.  Details may be found in \cite{thesis} or earlier versions of this paper posted on the arxiv.   

The differentiability class of form   \( \o \in \B_k^0(U) \) can thus be verified at a point  \( p \in U \) by taking limits in the usual way, since   \( p + tu \in U \) for all for all \( t \) sufficiently small.  If \( p \in \fr\,U \), however, more care is needed.  

An open subset \( U \subset \R^n \) is \emph{{locally convex}} if for all \( p \in \fr\,U \) and \( p = \lim a_i = \lim b_i \) where \( a_i, b_i \in U \), then \( a_ib_i \subset U \) for sufficiently large \( i \). If \( U \) is not locally convex, then there exists \( \o \in \B_k^0(U) \) which is not extendable to \( \o \in \B_k^0(\R^n) \).  It is not hard to construct examples.   Suppose \( p \in \fr\,U \) fails the test for locally convexity of \( U \).    We can take limits \( a_i \to p \) from different directions and get different answers \( \lim_{i \to \i} \o(p_i; \a) \) where \( \a \in \L_k \).  (See Figure \ref{fig:openset}.)   

\begin{figure}[ht] 
	\centering 
	\includegraphics[height=2in]{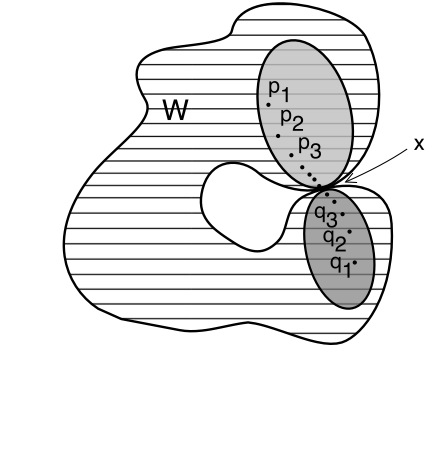}
	\caption{Distinct limit chains supported in \( x \) are separated by a form in \( \hB_0(U) \)} 
	\label{fig:openset} 
\end{figure}

\subsection{The \( B^r \) norm is indeed a norm on Dirac chains in \( U \)} 
\label{sub:the}

\begin{thm}\label{thm:normspace} 
	\( \| \cdot \|_{B^{r}} \) is a norm on Dirac chains \(\A_k(U) \) for each \( r \ge 0 \).
\end{thm}
                                           
\begin{proof} 
	Suppose \( A \ne 0 \) where \( A \in \A_k(U)\). It suffices to find nonzero \( \o \in (\A_k(U))^* \) with \( \o(A) \ne 0 \) and \( \|\o\|_{B^{r}} < \i \). By Lemma \ref{lem:ineq} we will then have \[ 0 < |\o(A)| \le \|\o\|_{B^{r}}\|A\|_{B^{r}} \implies \|A\|_{B^{r}} > 0. \]

	Now \( \supp(A) = \{p_0, \dots, p_N \} \). First assume that \( A(p_0) = e_I \) for some multi-index \( I \). Choose a smooth function \( f \) with compact support that is nonzero at \( p_0 \) and vanishes at \( p_i, 1 \le i \le N \). Let \[ \o_I(p; e_J) =\begin{cases} f(p), & J = I \\ 0, &J \ne I \end{cases} \] and extend to Dirac chains \( \A_k(U)\) by linearity.

	Then \( \o_I(A) = \o_I(p_0; e_I) = 1 \ne 0 \). Since \( f \) has compact support, \( \|f\|_{C^{r}} < \i \). We use this to deduce \( \|\o_I\|_{B^{r}} < \i \). This reduces to showing \( |\o_I|_{B^j} < \i \) for each \( 0 \le j \le r \). This, in turn,  reduces to showing \( |\o_I(\D_{\s^j}(p;e_I))| \le \|\s\| \|f\|_{C^j} \). By Lemma \ref{lem:est} \(|\o_I(\D_{\s^j}(p;e_I))| = \left | f(\D_{\s^j}(p;1)) \right|  \le \|\s\|\|f\|_{C^j} < \i \).
	
	In general, \( A(p_0) = \sum a_I e_I \) where \( a_I \in \R \).  Then \( \o := \sum_I  sign(a_I) \o_I \) satisfies \( \|\o\|_{B^{r}} < \i \) and \[ \o(A) = \sum_I  sign(a_I)  \o_I(A) = \sum_I |a_I|  \ne 0. \]
\end{proof}

\begin{defn}    
	Let \( \hB_k^r(U) \) be the Banach space obtained upon completion of \( \A_k(U)\) with the \( B^{r} \) norm for \( r \ge 0 \) and \( 0 \le k \le n \). Elements of \( \hB_k^r(U) \) are called \emph{\textbf{\textbf{differential \( k \)-chains of class \( B^{r} \) in \( U \)}}}.  	
\end{defn}

The next result follows from Lemma \ref{lem:ineq}.

\begin{thm}\label{thm:bipair}
The bilinear pairing \( \B_k^r(U) \times \hB_k^r(U) \to \R \) is continuous and satisfies \( |\o(J)| \le \|\o\|_{B^{r}}\|J\|_{B^{r}} \) for all \( \o \in \B_k^r(U) \), \( J \in \hB_k^r(U) \) and \( r \ge 0 \).
\end{thm}

\subsection{Inclusion maps} 
\label{sub:inclusion_maps}
        
\begin{defn}
	Suppose \( U \subseteq U' \subseteq \R^n \), \( 0 \le k \le n \), and \( r \ge 0 \).
	Define \( \psi_k^r: \A_k(U) \hookrightarrow \A_k(U') \) by \( \psi_k^r := Id \).  
\end{defn}

\begin{thm}\label{thm:psi}
 The linear map \( \psi_k^r \) extends to a continuous linear map \( \psi_k^r:\hB_k^r(U) \to \hB_k^r(U') \) for each \( r \ge 0 \) and \( 0 \le k \le n \) and satisfies \( \|J\|_{B^{r,U'}} \le \|J\|_{B^{r,U}} \) for all \( J \in \hB_k^r(U) \).   For each \( k \ge 0 \), there exists a continuous linear map \( \psi_k: \hB_k(U) \to \hB_k(U') \) with \(  \psi_k \circ u_k^r = \psi_k^r \) for all \( r \ge 0 \).   
\end{thm}

\begin{proof}   
	 If \( A \in \A_k(U) \), then \( \|A\|_{B^{r,U'}} \le \|A\|_{B^{r,U}} \) for all \( r \ge 0 \) since the definition of \( \|A\|_{B^r, U'} \) considers more difference chains than does that of \( \|A\|_{B^r, U} \).
	The inequality follows by density of \( \A_k(U) \) in \( \hB_k^r(U) \).
  The extension to the inductive limit follows since the hull condition holds: \( \psi_k^r(H_k(U)) \subset H_k(U') \). We can therefore apply Theorem \ref{thm:continuousoperators}.   
\end{proof}

Example \ref{notinjective} shows that these maps are not generally injective. 

\begin{example} \label{notinjective} In Figure \ref{fig:openset} there are two sequences of points \( p_i, q_i \) converging to \( x \) in \( \R^2 \). We know  both \( (p_i;1) \to (x;1) \) and \( (q_i;1) \to (x;1) \) in \( \hB_0^1(\R^2) \). Each of the sequences \( \{(p_i;1)\} \) and \( \{(q_i;1)\} \) is Cauchy in \( \hB_0^1(U) \) since  the intervals connecting \( p_i \) and \( p_j \), and those connecting \( q_i \) and \( q_j \) are subsets of \( U \). Therefore \( (p_i;1) \to A_p \) and \( (q_i;1) \to A_q \) in \(  \hB_0^1(U) \). However,  the interval connecting \( p_i \) and \( q_i \) is \emph{not}  a subset of \( U \). The chains \( A_p \ne A_q \) in \( \hB_0^1(U) \) since they are separated by elements of the dual Banach space \( \B_0^1(U) \).  For example, let \( \o_1 \) be a smooth form defined on \( U \) which is identically one in the lighter shaded elliptical region, and identically zero in the darker shaded region.  Then \( \cint_{A_p} \o_1  = \lim_{i\to \i} \o_1(p_i;1) = 1 \) and \( \cint_{A_q} \o_1  = \lim_{i \to \i}\o_1(q_i;1) = 0 \). Since \(\psi_0^1(A_p) =\psi_0^1(A_q) = (x;1) \) it follows that \( \psi_0^1 \) is not injective.   
\end{example}  

\begin{defn}\label{frontier}
	The \emph{\textbf{frontier}} of an open set \( U \) is defined by \( \fr(U):= \overline{U} \backslash U \).    
\end{defn}
We shall reserve the word ``boundary'' for the operator on differential chains developed in \S \ref{sub:boundary_operator}, as well as for the classical boundary of a submanifold, both of which have orientation.  Frontier is defined for sets, with no algebraic properties.  Thus when we speak of the boundary of an open set, we are thinking of \( U \) as a submanifold, and when we speak of its frontier, we are thinking of \( U \) as a set. 
                   
\begin{thm}\label{thm:injectionsopen}
Suppose \( U \subset U' \subset \R^n \). The  linear maps \( \psi_k^r:\hB_k^r(U) \to \hB_k^r(U') \) and \( \psi_k:\hB_k(U) \to \hB_k(U') \) are injections if the frontier of \( U \) is embedded in \( U' \), or on restriction to the subspace of chains supported in \( U \).  
\end{thm}

\begin{proof}  
	Suppose \( J \in \hB_k^r(U) \) satisfies \( \psi_k^r(J) = 0 \).     It suffices to show that \( \int_J \o = 0 \) for all \( \o \in \B_k^r(U) \), for this will show \( J = 0 \).  Since the frontier of \( U \) is smoothly embedded in \( U' \), each \( \o \in \B_k^r(U) \)  extends to a differential form \( \eta \in \B_k^r(U') \), and hence \( \cint_{\psi_k^r(J)} \eta = 0 \).  Let \( A_i \to J \) in \( \hB_k^r(U) \).  Then \( \psi_k^r(A_i) \to \psi_k^r(J) \) by Theorem \ref{thm:psi}.   Since \( \psi_k^r \) is the identity on Dirac chains,  \(  \cint_J \o = \lim_{i \to \i} \cint_{A_i} \o = \lim_{i \to \i} \cint_{\psi_k^r A_i} \eta  = \cint_{\psi_k^r(J)} \eta = 0 \).  
\end{proof} 

  The \emph{\textbf{restriction}} of Dirac chains \( \A_k(U') \to \A_k(U) \) does not extend to a continuous restriction of differential chains.   Let \( U \) be the open set \( \R^2 \) less the nonnegative \( x \)-axis.  The series \( \sum_{n=1}^\i((n, 1/n^2);1) + ((n,-1/n^2);-1) \) of Dirac \( 0 \)-chains converges in \( \hB_0^1(\R^2) \), but not in \( \hB_0^1(U) \).  The problem is that the difference chains needed to prove convergence in \( \hB_0^1(\R^2) \) are not inside \( U \).    

\subsection{Characterizations of the \( B^r \) norms} 
\label{sub:characterizations_of_the_B}

\begin{lem}\label{lem:trnsl} 
	\( \|\D_v A\|_{B^{r+1}} \le \|v\|\|A\|_{B^{r}} \) for all  \( A \in \A_k(U)\), \( v \in \R^n \), and \( r \ge 0 \). 
\end{lem}
                                 
\begin{proof} 
	Since Dirac chains \( \A_k(U)\) are dense in \( \hB_k^r(U) \), it suffices to prove this for \( A \in \A_k(U)\). Let \( \e > 0 \). We can write \( A = \sum_{i=1}^m \D_{\s_i^{j(i)}}(p_i; \a_i) \) as in the proof of Lemma \ref{lem:ineq}, such that \[ \|A\|_{B^{r}} > \sum_{i=1}^m \|\s_i^{j(i)}\|\|\a_i\| - \e.\] Then
	\begin{align*} 
		\|\D_v A\|_{B^{r+1}}= \|T_v A - A\|_{B^{r+1}} \le \sum_{i=1}^m \|T_v \D_{\s_i^{j(i)}} - \D_{\s_i^{j(i)}} \|_{B^{r+1}} &\le \sum_{i=1}^m | T_v \D_{\s_i^{j(i)}} - \D_{\s_i^{j(i)}} |_{B^{j+1}}   \\
		&\le \|v\|\sum_{i=1}^m \|\s_i^{j(i)}\|\|\a_i\| \\
		&< \|v\|(\|A\|_{B^{r}} + \e ). 
	\end{align*} 
\end{proof} 

\begin{lem}\label{lem:lar} 
	The norm \(\|\cdot\|_{B^{r}} \) is the largest seminorm \( |\cdot|' \) on Dirac chains \( \A_k(U)\) such that \( |\D_{\s^j}(p;\a)|' \le \|\s\|\|\a\| \) for all \( j \)-difference \( k \)-chains \( \D_{\s^j}(p;\a) \), \( 0 \le j \le r \). 
\end{lem} 
                     
\begin{proof}
	First observe that the \( B^{r} \) norm itself satisfies this inequality by its definition. On the other hand, suppose \( |\quad|' \) is a seminorm satisfying \( |\D_{\s^j}(p;\a)|' \le \|\s\|\|\a\| \). Let \( A \in \A_k(U)\) be a Dirac chain and \( \e > 0 \). We can write \( A = \sum_{i=1}^m \D_{\s_i^{j(i)}}(p_i;\a_i) \) as in the proof of Lemma \ref{lem:ineq}, with \( \|A\|_{B^{r}} > \sum_{i=1}^m \|\s_i^{j(i)}\|\|\a_i\| - \e \). Therefore, by the triangle inequality, \( |A|' \le \sum_{i=1}^m |\D_{\s_i^{j(i)}}(p_i;\a_i)|' \le \sum_{i=1}^m \|\s_i^{j(i)}\|\|\a_i\| < \|A\|_{B^{r}} + \e \). Since this estimate holds for all \( \e > 0 \), the result follows. 
\end{proof} 

\begin{thm}\label{thm:lartheorem} 
	The norm \( \|\cdot\|_{B^{r}} \) is the largest seminorm  \( |\cdot|' \) on Dirac chains \( \A_k(U)\) such that 
	\begin{itemize} 
		\item \( |A|' \le \|A\|_{B^0} \) 
		\item \( |\D_u A |' \le \|u\| \|A\|_{B^{r-1} } \). 
  \end{itemize} 
for all \( r \ge 1 \) and \( A \in \A \). 
\end{thm}

\begin{proof} 
	\( \|\cdot\|_{B^{r}} \) satisfies the first part since the \( B^{r} \) norms are decreasing on chains. It satisfies the second inequality by Lemma \ref{lem:trnsl}. On the other hand, suppose \( |\cdot|' \) is a seminorm satisfying the two conditions. For \( j=0 \), we use the first inequality \( | (p;\a)|' \le \| (p;\a) \|_{B^0} = \| \a\| \).  Fix \( 0 < j \le r \). Using induction and recalling the notation \( \hat{\s} \) from Lemma \ref{lem:est}, it follows that 
	\begin{align*} 
		|\D_{\s^j}(p;\a)|' = |\D_{\hat{\s}^{j-1},u_1} (p;\a)|' \le \|u_1\|\|\D_{\hat{\s}^{j-1}} (p;\a)\|_{B^{j-1} } \le \|\s\|\|\a\|. 
	\end{align*} 
	Therefore, the conditions of Lemma \ref{lem:lar} are met and we deduce \( |A|' \le \|A\|_{B^{r}} \) for all Dirac chains \( A \). 
\end{proof}

Let \( \A = \oplus_{k=0}^n \A_k(U)\). Recall \( \D_u A = T_u A - A \) where \( T_u (p;\a) = (p+u; \a) \).
     
\begin{cor}\label{cor:opr} 
	If \( T: \A \to \A \) is an operator satisfying \( \|T(\D_{\s^j}(p;\a))\|_{B^{r}} \le C \|\s\|\|\a\| \) for some constant \( C > 0 \) and all \( j \)-difference \( k \)-chains \( \D_{\s^j}(p;\a) \) with \( 0 \le j \le s \) and \( 0 \le k \le n \), then \( \|T(A)\|_{B^{r}} \le C\|A\|_{B^s} \) for all \( A \in \A_k(U)\).
\end{cor}

\begin{proof} 
	Let \( |A|' = \frac{1}{C} \|T (A)\|_{B^{r}} \). Then \( |A|' \) is a seminorm, and the result follows by Theorem \ref{thm:lartheorem}.
\end{proof}

\subsection{Isomorphism of differential forms and differential cochains } 
\label{sub:isomorphism_theorem}
 
 In this section we show that the Banach space \( \B_k^r(U) \) of differential forms is topologically isomorphic to the Banach space \( (\hB_k^r(U))' \) of \emph{\textbf{differential cochains}}. 

\begin{thm}\label{thm:seminorm} 
Let \( r \ge 0 \).	If \( \o \in \B_k^r(U) \) is a differential \( k \)-form, then \( \|\o\|_{B^{r}} = \sup_{ 0 \ne A \in \A_k} \frac{ \left| \o(A) \right|}{\|A\|_{B^{r}}} \).
\end{thm}

\begin{proof} 
	We know \( \left|\o(A)\right| \le \|\o\|_{B^{r}} \|A\|_{B^{r}}\) by Lemma \ref{lem:ineq}. On the other hand,  
	\begin{align*} 
		|\o|_{B^{j}} = \sup_{0 \ne \D_{\s^j (p;\a)}} \frac{|\o(\D_{\s^j}(p;\a))|}{\|\s^j\|\|\a\|} \le \sup_{0 \ne \D_{\s^j (p;\a)}} \frac{|\o(\D_{\s^j}(p;\a))|}{\|\D_{\s^j}(p;\a)\|_{B^{r}}} \le \sup_{0 \ne A} \frac{ | \o(A)|}{\|A\|_{B^{r}}}.
	\end{align*} 
 	It follows that \( \|\o\|_{B^{r}} \le \sup_{0 \ne A} \frac{ |\o(A)|}{\|A\|_{B^{r}}}. \)
\end{proof} 
              
\begin{thm}\label{thm:derham} 
	The linear map \( \Psi_r: (\hB_k^r(U))' \to \B_k^r(U) \) determined by \( \Psi_r(X)(p;\a) := X(p;\a) \) is a topological isomorphism for each \( r \ge 0 \). Furthermore, \[ \|\Psi_r(X)\|_{B^{r}} = \|X\|_{B^{r}}. \]
\end{thm}

\begin{proof} 
Let \( X \in ( \hB_k^r(U))' \) be a cochain. We show \(
\|\Psi_r(X)\|_{B^{r}} < \|X\|_{B^{r}} \). Now \[ \frac{\left| \Psi_r(X)(\D_{\s^j}(p;\a))
\right|}{\|\s^j\|\|\a\|} = \frac{|X(\D_{\s^j}(p;\a))|}{\|\s^j\|\|\a\|} \le
\frac{|X(\D_{\s^j}(p;\a))|}{\|\D_{\s^j}(p;\a)\|_{B^{r}}} \le \|X\|_{B^{r}} \] Therefore, \(
|\Psi_r(X)|_{B^{j}} \le \|X\|_{B^{r}} \), and thus \( \|\Psi_r(X)\|_{B^{r}} \le \|X\|_{B^{r}} \).

 It follows that \( \Psi_r: (\hB_k^r(U))' \to \B_k^r(U) \) is a continuous linear map. We show \( \Psi_r \) is a continuous isomorphism. Let \( \o \in \B_k^r(U) \). Let \( \Theta_r:\B_k^r(U) \to (\hB_k^r(U))' \) be given by \( \Theta_r \o (A) := \o(A) \). Then \( \Theta_r \o \) is a linear functional on \( {\cal A}_k(U) \). By Theorem \ref{thm:seminorm} \[ \|\Theta_r \o\|_{B^{r}} = \sup_{A \ne 0} \frac{ |\Theta_r \o(A)| }{\|A\|_{B^{r}}} = \sup_{A \ne 0} \frac{ | \o(A)| }{\|A\|_{B^{r}}} = \|\o\|_{B^{r}}. \] Therefore, \( \Theta_r \o \in (\hB_k^r(U))' \). We conclude that \( \Psi_r \) is a topological isomorphism with inverse \( \Theta_r \). 
\end{proof} 

\begin{cor}\label{cor:han}
If \( J \in \hB_k^r(U) \) is a differential chain and \( r \ge 0 \), then \( \|J\|_{B^{r}} = \sup_{ 0 \ne \o \in \B_k^r(U)} \frac{ \left| \o(J) \right|}{\|\o\|_{B^{r}}} \).
\end{cor}
   
\subsubsection{Integration} 
\label{ssub:integration}

Although the space \(  \B_k^r(U) \) of differential forms is dual to \( \hB_k^r(U) \), and thus we are perfectly justified in using the notation \( \o(J) \), this can become confusing since we are no longer evaluating \( \o \) at a set of points, as we were with elements of \( \A_k(U)\).  Instead we shall think of \( \o(J) \) as ``integration'' over \( J \) and denote it as \( \cint_J \o := \o(J) \).   Indeed, the integral notation is justified by the approximation of \( \o(J) \) by its analogue to Riemann sums.  That is, \[ \cint_J \o = \lim_{i \to \i} \o(A_i), \] where \( A_i \in \A_k(U)\) and \( A_i \to J \) in \( \hB_k^r(U) \). In this new notation, Theorem \ref{thm:bipair}, for example, becomes 
\begin{equation}\label{basic}
	\left|\cint_J \o \right| \le \|J\|_{B^r}\|\o\|_{B^r}. 
\end{equation} 

    We shall see later in Theorems \ref{thm:opensets}, \ref{thm:cells}, \ref{thm:jX} and  \S\ref{sec:representatives_of_domains_of_integration}  that there are natural \emph{\textbf{representatives}}
     of classical domains \(A\) of integration, e.g., open sets, polyhedral chains, vector fields or submanifolds, as differential chains \( J \).  That is, the Riemann integral  \(\int_A \o \) and \( \cint_J \o\) agree.  
 
 The next result follows from Theorem \ref{thm:derham}. 
\begin{cor}\label{thm:integralpair}
   \( \cint: \hB_k^r(U) \times \B_k^r(U) \to \R \) is a nondegenerate bilinear pairing for \( r \ge 0 \).  
\end{cor}      
        
Suppose \( U \subset U' \) are open.  The Banach space \( \B_k^r(U) \) contains all differential forms \( \B_k^r(U') \) restricted to \( U \) since difference chains inside \( U \) are also inside \( U' \).    Elements of \( \B_k^r(U) \) extend to elements of \( \B_k^r(U') \) if \( U \) is an embedded submanifold with boundary. It is an interesting problem to characterize the collection of open sets \( U \) for which elements of \( \B_k^r(U) \) extend to elements of \( \B_k^r(\R^n) \).     
   
\begin{examples}\label{ex:QQ}\mbox{}
	\begin{enumerate} 
		\item The characteristic function \( \chi_U \) is an element of \( \B_n^r(U) \) for all \( r \ge 0 \).
		\item  Let \( U \subset \R^2 \) be the open set depicted in Figure \ref{fig:openset}.  It is easy to construct a differential form \( \o \in \B_0^r(U) \) which is identically one in the shaded region above the point \( x \) and identically zero in the shaded region below \( x \). Clearly, \( \o \) is not extendable to a neighborhood of \( U \).       
		\item Let \( Q\) be the open unit disk in \( \R^2 \). The sequence \( ((1-1/n,0); e_1\wedge e_2) \) converges to \( ((1,0); e_1\wedge e_2) \) in both \( \hB_2^1(Q) \) and \( \hB_2^1(\R^2) \). Now \(  (1,0) \notin Q  \), but \( ((1,0); e_1\wedge e_2) \) is still a chain ready to be acted upon by elements \( \o \in \B_2(Q) \). For example, \[ \chi_Q dxdy((1,0); e_1\wedge e_2)  = \lim_{n \to \i} \chi_Q dxdy ((1-1/n,0); e_1\wedge e_2)  = 1. \]   
		\item \label{o0} Let \( Q' \) be  the unit disk \( Q \) in \( \R^2 \) minus the closed set \( [0,1] \times \{0\} \), and define \[ \o_0(x,y) =
			\begin{cases} 
				\inf\{x,1\}, &\mbox{ if } 0 < x < 1, 0 < y < 1\\
				0, &\mbox{ else}
			\end{cases}. 
			\]  
  			This Lipschitz function \( \o_0 \in \B_0^1(Q') \) is not extendable to a Lipschitz function on \( \R^2 \).   (See Figure \ref{fig:Mapping} for a related example.)
			\end{enumerate}
\end{examples}

\section{Representatives of domains} 
\label{sec:representatives_of_domains_of_integration}
                                  
\subsection{Open sets} 
\label{sub:representatives_of_open_sets} 

\begin{defn}
	We say that \( J \in \hB_k^r(U)  \) \emph{\textbf{represents}} an oriented \( k \)-dimensional submanifold (possibly with boundary) \( M \) of \( U \) if \( \cint_J \o = \int_M \o \) for all \( \o \in \B_k^r(U) \), where the integral on the right hand side is the Riemann integral.
\end{defn}

In particular, an \( n \)-chain \( \widetilde{U} \) \emph{\textbf{represents}} an open set \( U \), if \( \cint_{\, \widetilde{U}} \o = \int_U \o \) for all forms \( \o \in \B_n^1(U) \).  We show there exists a unique element \( \widetilde{U} \) in \( \hB_n^1(U) \) representing \( U \)  for each bounded, open subset \( U \subset  \R^n \). 

Let \( Q \) be an \( n \)-cube in \( U \) with unit side length. For each \( j \ge 1 \), subdivide \( Q \) into \( 2^{nj} \) non-overlapping sub-cubes \( Q_{j_i} \), using a binary subdivision. Let \( q_{j_i} \) be the midpoint of \( Q_{j_i} \) and \( P_j = \sum_{i=1}^{2^{nj}} (q_{j_i}; 2^{-nj} \mathbb{1}) \) where \( \mathbb{1} = e_1 \wedge \cdots \wedge e_n \).   

\begin{lem}\label{lem:pjisone}
\( \|P_j\|_{B^r} = 1 \) for all \( r \ge 0 \).	
\end{lem}
 
\begin{proof}  
For each \( r \ge 0 \), 
\( 	\|P_j\|_{B^r} \le \sum_{i=1}^{2^{nj}} \|(q_{j_i}; 2^{-nj} \mathbb{1})\|_{B^r} \le 2^{nj} 2^{-nj} = 1 \).  By Corollary \ref{cor:han} \( \|P_j\|_{B^r} \ge \cint_{P_j} dV = 1 \) since \( \|dV\|_{B^r} = 1 \).  
\end{proof}
 
\begin{prop}\label{prop:cuberep} 
	The sequence \( \{P_j\} \) is Cauchy in the \( B^1 \) norm.
\end{prop}

\begin{proof}
  Here is the basic idea to estimate \( \|P_j - P_{j+1} \|_{B^1} \):  Both Dirac \( n \)-chains \( P_j \) and \( P_{j+1} \) have mass one. They are supported in sets of points placed at midpoints of the binary grids with side length \( 2^{-j} \) and \( 2^{-(j+1)} \), respectively. Each of the \( n \)-elements of \( P_j \) has mass \( 2^{-nj} \) and those of \( P_{j+1} \) have mass \( 2^{-n(j+1)} \). The key idea is to think of each simple \( n \)-element of \( P_j \) as \( 2^n \) duplicate copies of simple \( n \)-elements of mass \( 2^{-n(j+1)} \) supported at the same point. This gives us a \( 1-1  \) correspondence between the simple \( n \)-elements of \( P_j \) and those of \( P_{j+1} \). We can choose a bijection so that the distance between points paired is less than \( 2^{-j+1} \). Use the triangle inequality with respect to \( P_j - P_{j+1} \) written as a sum of differences of these paired \( n \)-elements and Lemma \ref{lem:trnsl} to obtain \( \|P_j - P_{j+1} \|_{B^1} \le 2^{-j+1} \). It follows that \( \{P_j\} \) is Cauchy in the \( B^1 \) norm since the series \( \sum 2^{-j+1} \) converges. 
\end{proof}

Denote the limit  \( \widetilde{Q} := \lim_{j \to \i} P_j \) in the \( B^1 \) norm.  Then \( \widetilde{Q} \in \hB_n^1(U) \) is a well-defined differential \( n \)-chain.  By continuity of the differential chain integral \eqref{basic} and the definition of the Riemann integral as a limit of Riemann sums, we have \[ \cint_{\widetilde{Q}} \o = \lim_{j \to \i} \cint_{\sum (p_{j_i}; \a_{j_i})} \o = \int_Q \o.\] That is, \( \widetilde{Q} \) represents \( Q \).
 If \( W  = \cup_i Q_i \) is a finite union of non-overlapping cubes, then \( \widetilde{W} := \sum \widetilde{Q_i} \) represents \( U \).

\begin{thm}\label{thm:opensets} 
  Let \( U \subset W \) be bounded and open.  There exists a unique differential \( n \)-chain \( \widetilde{U} \in \hB_n^1(W) \) such that \( \cint_{\widetilde{U}} \o = \int_U \o \) for all \( \o \in \B_n^1(W) \), where the integral on the right hand side is the Riemann integral.   Furthermore, \( \|\widetilde{U}\|_{B^1} = \int_U dV \), the volume of \( U \).                                          
\end{thm}

\begin{proof}
	Let \( U = \cup_{i=1}^\i Q_i \) be  a Whitney decomposition of \( U \) into a union of non-overlapping \( n \)-cubes, and \( U_s = \sum_{i=1}^s \widetilde{Q_i} \).  Ue first show that the differential chains \( \widetilde{U_s} \) form a Cauchy sequence in \( \hB_n^1(W) \).  Now \( \|\widetilde{U_t} - \widetilde{U_s}\|_{B^1} = \sum_{i=s+1}^t \|\widetilde{Q_i}\|_{B^1} \le  \sum_{i=s+1}^t \|\widetilde{Q_i}\|_{B^0} =  \sum_{i=s+1}^t  \int_{Q_i} dV \) tends to zero as \( s \le t \to \i \) since the last sum is bounded by the volume of a small neighborhood of the frontier of \( U \), a bounded open set.  Therefore \( \sum_{i=1}^{\i} \widetilde{Q_i} \) converges to a well-defined chain  \( \widetilde{U} \in \hB_n^1(W) \).  That is, 
	\begin{equation}\label{whitdecomp}
		 \widetilde{U} = \sum_{i=1}^\i \widetilde{Q_i}
	\end{equation}
where convergence is in the \( B^1 \) norm.  
	 Suppose \( \o \in \B_n^1(W) \). Then \( \cint_{\widetilde{U}} \o = \lim_{s \to \i} \cint_{\sum_{i=1}^s \widetilde{Q_i}} \o  = \int_U \o \) by the definition of the Riemann integral.  
 	
	Ue first prove the last assertion for positively oriented unit \( n \)-cubes \( Q \subset W \).  Recall from Proposition \ref{prop:cuberep} that \( \widetilde{Q} = \lim_{j \to \i} P_j \) where the \( P_j \in \A_n(W) \).  Using Lemma \ref{lem:pjisone}, we have \( \|\widetilde{Q}\|_{B^1} = \lim_{j \to \i} \|P_j\|_{B^1} = 1 = \int_Q dV\). The result follows for cubes with arbitrary side length  by linearity. 
	Now let \( U \) be any bounded and open subset of \( W \). By \eqref{whitdecomp} we know 
	\( \|\widetilde{U}\|_{B^1}  = \lim_{s \to \i}  \|\sum_{i=1}^s \widetilde{Q_i}\|_{B^1} \le \lim_{s \to \i} \sum_{i=1}^s \|\widetilde{Q_i}\|_{B^1} = \lim_{s \to \i} \sum_{i=1}^s \int_{Q_i} dV =  \int_U dV  \).  Finally, since \( dV \in \B_n^1(W) \) and \( \|dV\|_{B^1} = 1 \) we use Corollary \ref{cor:han} to obtain \( \|\widetilde{U}\|_{B^1} \ge \cint_{\widetilde{U}} dV = \int_U dV \). 
\end{proof}   

\subsection{Polyhedral chains} 
\label{sub:polyhedral_chains}

\begin{defn}
	Recall that an \emph{\textbf{affine \( n \)-cell}} in \( U \) is the intersection of finitely many affine half spaces in \( U \) whose closure is compact.  The half spaces may be open or closed.  An affine \( n \)-cell can be partly open and partly closed. That is, it is the union of an open \( n \)-cell with possibly some of its faces.    An \emph{\textbf{affine \( k \)-cell}} in \( U \) is an affine \( k \)-cell in a \( k \)-dimensional affine subspace of \( U \).  
\end{defn}

\begin{thm} \label{thm:cells} 
	If \( \s \) is an oriented affine \( k \)-cell in \( U \), there is a unique differential \( k \)-chain \( \widetilde{\s} \in \hB_k^1(U) \) such that \[ \cint_{\widetilde{\s}} \o = \int_\s \o \]  for all \( \o\in \B_k^1(U) \), where the right hand integral is the Riemann integral.
\end{thm}

\begin{proof}
	The result follows from  Theorem \ref{thm:opensets} applied to the $k$-dimensional subspace of \( U \) containing \( \s \).
\end{proof}

\begin{defn}\label{poly}
	A \emph{\textbf{polyhedral \(k\)-chain in \( U \)}}\footnote{An equivalent definition uses simplices instead of cells. Every simplex is a cell and every cell can be subdivided into finitely many simplices} is a finite sum \( \sum_{i=1}^s a_i \widetilde{\s_i} \) where \( a_i \in \R \) and \( \widetilde{\s_i} \in \hB_k^1(U) \) represents an oriented affine \( k \)-cell \( \s_i \) in \( U \). 
\end{defn}

In \S\ref{sub:algebraic_chains} we use the pushforward operator to create smooth versions of polyhedral chains, resulting in representatives of compact submanifolds.
\subsubsection{Polyhedral chains are dense in differential chains}\label{sub:poly} 
 The main result of this section Theorem \ref{thm:cellsmore} is of central importance, for it allows us to use Dirac chains in our definitions and proofs, when we might be primarily interested in applications to submanifolds, polyhedral chains, or even soap films.   Its simple proof drops out of the definition of the \( B^r \) norms.  

Let \( p \in U \) and \( e_I \) a basis element of \( \L_k \).    Let \( Q_j \subset U \) be a \( k \)-cube centered at \( p \), with side length \( 2^{-j} \), and contained in the affine plane containing \( p \) and parallel to the \( k \)-direction of \( e_I \). Orient \( Q_k \) to match the orientation of \( e_I \). 

\begin{lem}\label{lem:limitpoint} 
	\( (p;e_I) = \lim_{j \to \i} 2^{j k}\widetilde{Q}_j \) in the \( B^1 \) norm.
\end{lem}

\begin{proof} 
By Theorem \ref{thm:opensets} the cube \( Q_j \) is represented by \( \widetilde{Q}_j \in \hB_k^1(U) \).  We know from Proposition \ref{prop:cuberep} that \( \widetilde{Q}_j \)  is the limit of Dirac \( k \)-chains \( A_{j,i} \) supported in the midpoints of the \( 2^{-(j+i)} \) binary subdivision of \( Q_j \). The total mass of \( A_{j,i} \) is \( 2^{-jk} \) and each of the \( k \)-elements has \( k \)-direction and orientation the same as that of \( e_I \).   We can translate each of the simple \( k \)-elements of \( A_{j,i} \) to \( p \).  The distance translated is less than \( 2^{-j} \).   The triangle inequality shows that \(\| 2^{j k}A_{j,i} - (p; e_I)\|_{B^1} \le 2^{-j} \).  The result follows since \( \|2^{j k}\widetilde{Q}_j - (p;e_I)\|_{B^1} \le 
\|2^{j k}\widetilde{Q}_j -  2^{j k}A_{j,i}\|_{B^1} + \| 2^{j k}A_{j,i} - (p;e_I)\|_{B^1} \).   
\end{proof}

It follows that if \( \a \in \L_k \) is simple, then  \( (p;\a) \) is the limit of renormalized representatives of \( k \)-cubes in the \( B^1 \) norm by writing \( \a = \sum a_I e_I \).  

\begin{thm}\label{thm:cellsmore} 
	Polyhedral \( k \)-chains are dense in the Banach space of differential \( k \)-chains \( \hB_k^r(U) \) for all \( r \ge 1 \). 
\end{thm}

\begin{proof} 
This follows from Lemma \ref{lem:limitpoint} and since Dirac chains are dense.  
\end{proof}

\begin{remark}
	Lemma  \ref{lem:limitpoint}  does not rely on any particular shape to approximate \( (p;\a) \). It certainly does not need to be a cube, nor do we need to use a sequence of homothetic replicas of a given shape. We may use any sequence of chains \( J_i \) whose supports tend to \( p \)  and  \( \vec_k(J_i) \to \a \). (The operator \( \vec_k:\hB_k \to \L_k \)   defined by \( \vec_k(\sum (p_i;\a_i)) := \sum \a_i \) is continuous on \( \hB_k^r \).  See \cite{whitney} for a proof for \( r = 1 \).  The general result for \( \hB_k^r \), \( 1 \le r \le \i \) is similar.)
\end{remark}

\subsection{Fractals} 
\label{fractals}
\begin{examples} [Representatives of fractals] \mbox{}
\begin{enumerate}  
	\item \label{cantor}  The middle third Cantor set can be represented as a limit of polyhedral \( 1 \)-chains: Let \( E_1 \) be the chain representing the oriented interval \( (0,1) \). Let \( C_1 \) represent \( (1/3, 2/3) \), and let \( E_2 = E_1 + (- C_1) \). We have replaced the word ``erase'' with the algebraically precise ``subtract''.  Recursively define \( E_n \) by subtracting the middle third of \( E_{n-1} \). The mass of \( E_n \) is \( (\frac{2}{3})^n \). It is not hard to show that the sequence \( \{ (\frac{3}{2})^n E_n \} \) forms a Cauchy sequence in \( \hB_1^1(U) \).   Therefore its limit is a differential \( 1 \)-chain \( \G \) in \( \hB_1^1(U) \).  (See \S\ref{sub:boundary_operator} where the boundary operator is applied to \( \G \).)         
	\item   The interior of the Sierpinski triangle \( T \) can be represented by a \( 2 \)-chain \( \widetilde{T} = \lim_{k \to \i} (4/3)^k \widetilde{S_{k_i}} \) in the \( B^1 \) norm where the \( S_{k_i} \) are oriented simplices filling up the interior of \( T \) in the standard construction\footnote{If we had started with polyhedral chains instead of Dirac chains, convergence would be in the \( B^0 \) norm.}(see Figure \ref{fig:Sierpinski}). The support of its boundary \( \p \widetilde{T} \) is \( T \). We may integrate smooth forms and apply the integral theorems to calculate flux, etc. Other applications to fractals may be found in \cite{continuity, earlyhodge}.
	\begin{figure}[ht]
	 	\centering
	 		\includegraphics[height=1in]{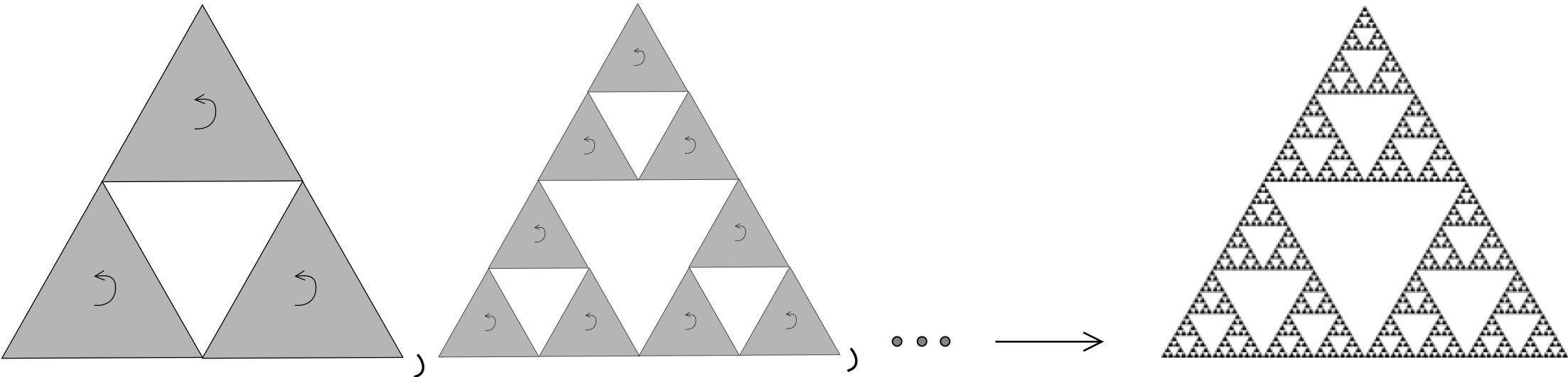}
	 	\caption{Sierpinski triangle}
	 	\label{fig:Sierpinski}
	\end{figure}  

 	\item  Recall that an \emph{\textbf{iterated function system}}  is a collection of functions 
			   \(  \{f_i:B\to B |i=1,2,\dots,N\},\ N\in\mathbb{N} \)  that are contractions on \( B \), an open ball of \( \R^n \).   	Let \( A \subseteq B \) be compact.  Define \( H(A) := \cup_{i=1}^N f_i(A) \).  Then \( S = \lim_{n \to \i} H^n(A) \)    is a well-defined subset of \( B \) which is  independent of \( A \).  The existence and uniqueness of \( S \) is a consequence of the contraction mapping principle. In particular, we can find \( S \) by setting \( A = \{p\} \) where  \( p \in B \).    Let   \( J_0:= \sum_{i=1}^N f_{i*}(p; e_1 \wedge \cdots \wedge e_n /N) \), \( J_1 :=  \sum_{i=1}^N f_{i*} J_0/N \) and \( J_n:= \sum_{i=1}^N f_{i*} J_{n-1}/N \).
			It is straightforward to show that \( \{J_n\} \) is Cauchy in the \( B^1 \) norm:   First note that \( \|J_n\|_{B^0} = 1 \).  We only have to bound the distance points are moved at the \( n \)-th stage, but such a bound is given by \( m^n \) where \( m = \max c_i \) and \( c_i < 1 \) is the contraction constant of \( f_i \).    
     	Let \( J := \lim_{n \to \i} J_n  \in \hB_n^1(B) \).  It follows that its mass \( M(J) = 1 \) and its support \( \supp(J) = S \).     
\end{enumerate}  
\end{examples} 

\section{The inductive limit and its topology} 
\label{sub:the_inductive_limit_topology}  
Some important applications only require the spaces \( \hB^r(U) \) for \( 0 \le r \le 2 \) (see, e.g., \cite{plateau10}).  But for a full theory with a continuous boundary operator, we need to take the inductive limit. The reader might wish to skip this section for the first reading since it uses some methods of topological vector spaces which are not widely known.  Results in sections after this will be safe to read as long as they are with regard to the Banach spaces \( \hB_k^r(U) \).    

Since the \( B^r \) norms are decreasing, the identity map   \( (\A_k(U), \|\cdot\|_{B^r}) \to \hB_k^s(U)  \) is continuous and linear whenever \( r \le s \), and therefore extends to the completion, \( \hB_k^r(U) \).  The resulting continuous linear maps \( u_k^{r,s}:\hB_k^r(U) \to \hB_k^s(U)  \) are well-defined \emph{\textbf{linking maps}}.   The next lemma is straightforward:
\renewcommand{\labelenumi}{(\alph{enumi})}
\begin{lem}\label{lem:varinj} 
	The maps \( u_k^{r,s}: \hB_k^r(U) \to \hB_k^s(U)  \) satisfy 
	\begin{enumerate} 
		\item \( u_k^{r,r} = Id \); 
		\item \( u_k^{s,t} \circ u_k^{r,s} = u_k^{r,t} \) for all \( r \le s \le t \);
		\item  The image   \( u_k^{r,s}(\hB_k^r(U)) \) is dense  in \( \hB_k^s(U)  \).
	\end{enumerate}
\end{lem}
                                                              
In Corollary \ref{cor:inj} below we will show that each 
 \( u_k^{r,s}: \hB_k^r(U) \to \hB_k^s(U) \) is an injection.  This will follow from knowing that forms of class \( B^r \) are approximated by forms of class \( B^{r+1} \) in the \( B^r \) norm. Whitney proved this for \( C^r \) forms using convolution product (\cite{whitney} Chapter V, Theorem 13A), and our proof is similar.

For \( c > 0 \), let \( \overline{\kappa_c}:[0,\i)] \to \R \) be a nonnegative smooth function that is monotone decreasing, constant on some interval \( [0, t_0] \) and equals \( 0 \) for \( t \ge c.\) Let \( \zeta_{c}: U \to \R \) be given by \( \zeta_{c}(v) = \overline{\zeta_c}(\|v\|)  \). Let \( dV = dx_1 \wedge \cdots \wedge dx_n \) the unit \( n \)-form. We multiply \( \zeta_c \) by a constant so that \( \int_{U} \zeta_c (v) dV = 1 \). Let \( a_r \) be the volume of an \( n \)-ball of radius \( r \) in \( U \). For \( \o \in  \B_k^r(U) \) a differential $k$-form and \( (p;\a) \) a simple $k$-element, let \[ \o_c(p;\a) = \int_{U} \zeta_c(v)\o(p+v;\a)dV. \] 

\begin{thm}\label{thm:injection} 
	If \( \o \in  \B_k^r(U) \) and \( c > 0 \), then \( \o_c \in \B_k(U) \) and 
	\begin{enumerate} 
		\item \( L_u(\o_c) = (L_u \o)_c \) for all \( u \in \R^n \); 
		\item \( \|\o_c\|_{B^r} \le \|\o\|_{B^r} \); 
		\item \( \o_c \in \B_k^{r+1}(U)  \); 
		\item \( \o_c(J) \to \o(J) \) as \( c \to 0 \) for all \( J \in \hB_k^r(U) \).
	\end{enumerate} 
\end{thm}

\begin{proof} (a):
	\begin{align*} 	
  L_u (\o_c)(p;\a) &= \lim_{\e \to 0} \o_c((p+\e u;\a) - (p;\a))/\e = 
	\lim_{\e \to 0}  \int_{U} \zeta_c(v)\o(p+ \e u +v;\a) - (p + v;\a)/\e dV \\& =
	 \int_{U} \zeta_c(v)(L_u \o(p+v;\a))dV   = (L_u\o)_c (p;\a). 
	\end{align*} 

	(b): Since \( \int_{U} \zeta_c dV = 1 \), we know 
	\begin{align*}
		\frac{|\o_c(\D_{\s^j}(p;\a))|}{\|\s\|\|\a\|} &= \left| \int_{U} \zeta_c(v)\o(T_v\D_{\s^j}(p;\a)/\|\s\|\|\a\|)dV \right|  
	 \le \sup_{v \ne 0}\frac{|\o(T_v \D_{\s^j}(p;\a))|}{\|\s\|\|\a\|}  
	\\& = \sup_{v \ne 0} \frac{|\o( \D_{\s^j}(p+v;\a))|}{\|\s\|\|\a\|}\le \sup_{ q \in U} \frac{|\o(\D_{\s^j}(q;\a))|}{\|\D_{\s^j}(q;\a)\|_{B^r}} 
	 \le \sup_{0 \ne A \in \A_k} \frac{|\o(A)|}{\|A\|_{B^r}}  = \|\o\|_{B^r} 
	\end{align*} 
	for all \( 0 \le j \le r \). Therefore, \( |\o_c|_{B^j} \le \|\o\|_{B^r} \), and hence \( \|\o_c\|_{B^r} \le \|\o\|_{B^r} \).

 	(c): Suppose \( 0 \le j \le r \). Let \( \eta = L_{\s^j} \o \). Then \( |\eta|_{B^0} \le \|\o\|_{B^r} < \i \). By (a) we know \( \eta_c = L_{\s^j} (\o_c) \). Now \[ \eta_c(T_u (p;\a)) = \int_{U} \zeta_c(v) \eta(T_{v+u} (p;\a))dV = \int_{U} \zeta_c(v-u) \eta (T_v (p;\a)) dV, \] and \[ \eta_c (p;\a) = \int_{U} \zeta_c(v) \eta(T_v (p;\a))dV. \] Since the integrand vanishes for \( v \) outside  ball of radius \( c \) about the origin, we have
	\begin{align*} 
		|\eta_c (p+u;\a)  - \eta_c (p;\a))| &= \left|\int_{U} (\zeta_c(v -u) - \zeta_c(v)) \eta (p+v;\a)dV \right| \\
		&\le \int_{U} |\zeta_c(v -u) - \zeta_c(v)| |\eta(p+v;\a)|dV \\
		&\le a_c |\zeta_c|_{Lip}\|u\||\eta|_{B^0}\|\a\|. 
	\end{align*} 
	Therefore, \( |\eta_c|_{Lip} \le a_c |\zeta|_{Lip} \|\o\|_{B^r} \). Using (b), Proposition \ref{prop:oncemore} and Theorem \ref{thm:seminorm}, we deduce \( \|\o_c\|_{B^{r+1}} < \i \).

 	(d): First of all \[(\o_c - \o) (p;\a)   =  \int_{U} \zeta_c(v)\|v\| \o((p+v;\a) - (p;\a))/\|v\|dV =  \int_{U} \zeta_c(v)\|v\| L_v\o(p;\a)dV + r_c   \] where \( r_c \to 0 \).  
   Let \( A \) be a Dirac chain. Then \[ |(\o_c - \o) (A)| \le c \left|  \int_{U} \zeta_c(v)  L_v\o(A)dV \right| \le c \left|  \sup_v\{ \zeta_c(v)\}  \int_{U}   L_v\o(A)dV \right| \le c \sup_v \{ \zeta_c(v)\} a_c |\o(A)| \le  c' \|\o\|_{B^r}\|A\|_{B^r}\] where \( c' \to 0 \) as \( c \to 0 \).  
 For   \( J \in \hB_k^r(U) \) choose \( A_i \to J \) in \( \hB_k^r(U) \) and use \( \|A_i\|_{B^r}\to \|J\|_{B^r} \).
	\end{proof}

\begin{cor}\label{thm:flat} 
	If \( J \in \hB_k^r(U) \) is a differential \( k \)-chain and \( 0 \le r \), then 
	\[ 
	\|J\|_{B^r} = \sup_{0 \ne \o \in \B_k^{r+1}} \frac{|\o(J)|}{\|\o\|_{B^r}}. 
	\]
\end{cor}

\begin{proof}    
	According to Corollary \ref{cor:han}, we have \( \|J\|_{B^r} =  \sup_{0 \ne \o \in \B_k^{r}}  \frac{|\o(J)|}{\|\o\|_{B^r}}.  \) Suppose \( \o \in \B_k^r(U)  \). Let \( \e > 0 \). By Theorem \ref{thm:injection} (b)-(d), there exists \( c > 0 \) such that 
	\begin{align*} 
		\frac{|\o(J)|}{\|\o\|_{B^r}} \le \frac{|\o_c(J)| + \e}{\|\o\|_{B^r}}
		 \le \frac{|\o_c(J)| + \e}{\|\o_c\|_{B^r}} \le \sup_{\eta \in \B_k^{r+1}} \frac{|\eta(J)| + \e}{\|\eta\|_{B^r}}. 
	\end{align*}
	Since this holds for all \( \e > 0 \), we know \( \|J\|_{B^r} =  \sup_{\o \in \B_k^r(U) } \frac{|\o(J)|}{\|\o\|_{B^r}} \le \sup_{\eta \in \B_k^{r+1}(U) } \frac{|\eta(J)|}{\|\eta\|_{B^r}} \).  Equality holds since \( \B^{r+1} \subset \B^r \).
\end{proof}

\begin{cor}\label{cor:inj} 
	The linking maps \( u_k^{r,s}: \hB_k^r(U) \hookrightarrow \hB_k^s(U)  \) are injections for each \( r \le s \).
\end{cor}

\begin{proof}  It suffices to prove this for \( s = r+1 \).
	Suppose there exists \( J \in \hB_k^r(U) \) with \( u_k^{r,r+1}(J) = 0 \).   Then \( \o( u_k^{r,r+1}(J)) = 0 \) for all \( \o \in \B_k^{r+1}(U) \).   By Corollary \ref{thm:flat},  this implies that \( \|J\|_{B^r} = 0 \), and hence \( J = 0 \). 
\end{proof} 

\begin{defn}
 Let \( \hB_k(U) = \hB_k^\i(U)  := \varinjlim\,\hB_k^r(U) \), the inductive limit as \( r \to \i \).  Let \( u_k^r:\hB_k^r(U) \to  \hB_k(U) \)  the \emph{\textbf{canonical inclusion}} maps into the inductive limit.	  
\end{defn} 

\begin{cor}\label{cor:injectioninductive}
The canonical inclusions \( u_k^r:\hB_k^r(U) \to  \hB_k(U) \) are injective.
\end{cor}

\begin{proof} 
We know from Corollary \ref{cor:inj} that the linking maps are injective. If the linking maps \( u_k^{r,s} \) defining an inductive limit of vector spaces are injective, then so are the canonical inclusions \( u_k^r \) (\cite{kothe}, p. 219).
\end{proof}  
 
Each \( u_k^r(\hB_k^r(U))  \) is a Banach space with norm \( \|u_k^r(J)\|_r = \|J\|_{B^r} \).  Since the \( B^r \) norms are decreasing, then \( \{u_k^r(\hB_k^r(U))\} \) forms a nested sequence of increasing Banach subspaces of  \(  \hB_k(U) \).  We endow \( \hB_k(U) \) with the  the inductive limit topology \( \t_k \); it is the finest locally convex topology such that the maps \( u_k^r:\hB_k^r(U) \to  \hB_k(U) \) are continuous. If \( F \) is locally convex, a linear map \( T:( \hB_k(U),\t_k) \to F \) is continuous if and only if each \( T \circ u_k^r: \hB_k^r(U) \to F \) is continuous\footnote{One can also introduce H\"older conditions into the classes of differential chains and forms as follows: For \( 0 < \b \le 1 \), replace \( \|\s^j\|\|\a\| = \|u_j\| \cdots \|u_1\|\|\a\| \) in definition \eqref{def:norms} with \(  \|u_j\|^\b \cdots \|u_1\|\|\a\| \) where \( \s^j = = u_j \circ \dots \circ u_1 \). The resulting spaces of chains \( \hB_k^{r-1+\b}(U) \) can permit fine tuning of the class of a chain. The inductive limit of the resulting spaces is the same \(  \hB_k(U) \) as before, but we can now define the \emph{\textbf{extrinsic dimension}} of a chain \( J \in  \hB_k(U) \) as \( \dim_E(J) := \inf \{ k+ j-1 + \b: J \in \hB_k^{j-1+\b}(U) \} \). For example, \( \dim_E(\widetilde{\s}_k) = k \) where \( \s_k \) is an affine \( k \)-cell since \( \widetilde{\s}_k \in \hB_k^1(U) \) (set \( j=0, \b = 1 \)), while \( \dim_E(\widetilde{S}) = \ln(3)/\ln(2) \) where \( S \) is the Sierpinski triangle. The dual spaces are differential forms of class \( B^{r-1+\b} \), i.e., the forms are of class \( B^{r-1} \) and the \( (r-1) \) order directional derivatives satisfy a \( \b \) H\"older condition. The author's earliest work on this theory focused more on H\"older conditions, but she has largely set this aside in recent years.}. However, the inductive limit is not strict.  For example, $k$-dimensional dipoles are limits of chains in \( u_0(\hB_k^0(U)) \) in \( (\hB_k(U), \t_k) \), but dipoles are not found in the Banach space \( u_0(\hB_k^0(U)) \) itself.  

\begin{defn}
	If \( J\in \hB_k(U) = \varinjlim\,\hB_k^r(U) \), the \emph{\textbf{class}} of \( J \) is the infimum over all \( r > 0 \) such that there exists  \( J^r\in\hB_k^r(U) \) with \( u_k^r (J^r) = J \).
\end{defn} 
 \emph{\textbf{Class}} is well-defined since the Banach spaces \( u_k^r( \hB_k^r(U))\) generate  \(  \hB_k(U) \) (\cite{AG}, p. 136).  

It is sometimes the  case that results in \( \hB_k^r(U) \) carry over easily to the inductive limit \( \hB_k(U) \), but not always\footnote{Grothendiek wrote, ``Some questions arise concerning a space which is an inductive limit, which often receive negative answers, even for the inductive limits of a sequence of Banach spaces, and which often present serious difficulties. '' and ``We remark that in practice the difficulties which we encounter in inductive limits are the `converse' of those met in projective limits (the coarsest topology for which...); here it is nearly always easy to show that the space is complete, and to determine whether its bounded subsets are weakly compact or compact..., an in particular to recognize it as either a reflexive or a Montel space.''  See \cite{AG}, p 138.}.  For example, is the inductive limit Hausdorff? Is it complete? Is it reflexive? Is it Montel? Does every bounded subset of the inductive limit come from a bounded subset of one of the Banach spaces?  Much of \cite{topological}  is devoted to answering questions of this type, as well as the current paper.  

\begin{thm}\label{thm:frechet} 
	The locally convex space \(  ( \hB_k(U),\t_k) \) is Hausdorff and separable. 
\end{thm}

\begin{proof}  
Suppose   \(  J \in \hB_k(U) \) is nonzero.   Since the class of \( J \) is well-defined, there exists \( J^r \in \hB_k^r(U) \) with \( u_k^r(J^r) = J \) where \( u_k^r:\hB_k^r(U) \to  \hB_k(U)  \) is the canonical injection.  Thus \( J^r \ne 0 \).  By Theorem \ref{thm:injection} there exists \( \o \in B^\i \) with \( \o(J^r) \ne 0 \). It follows that  \( \o(J) \ne 0 \) since \( J \) and \( J^r \) are approximated by the same Cauchy sequence of Dirac chains.  Therefore \(  \hB_k(U) \) is separated.  This implies that \(   \hB_k(U)  \) is Hausdorff (\cite{AG} p. 59).    
 
	For the second part, since Dirac chains are dense in each \( \hB_k^r(U) \), we can approximate a given Dirac chain \( D \) by a sequence of Dirac chains \( D_i=\sum_j (p_{j,i}; \a_{j,i}) \) whose points \( p_{j,i} \) and \( k \)-vectors \( \a_{j,i} \) have rational coordinates.  
\end{proof}  

\begin{defn}
	Let \( \B_k(U)  \) denote the Fr\'echet space of bounded \( C^\i \)-smooth differential \( k \)-forms, with bounds on the \( j \)-th order directional derivatives for each \( 0\le j<\i \).  The defining seminorms can be taken to be the \( B^r \) norms.
\end{defn}

\begin{prop}\label{prop:isomo} 
	The topological dual to \( \hB_k(U) \) is isomorphic to \( \B_k(U) \). 
\end{prop} 

\begin{proof}   
 It is well known that the dual of the inductive limit topology is the projective limit of the duals, and vice versa. (\cite{kothe}, \S 22.7, for example.). 	Since \( \hB_k(U) = \varinjlim \hB_k^r(U) \), we know \( (\hB_k(U))' = \varprojlim (\hB_k^r(U))' \cong \varprojlim \B_k^r(U) = \B_k(U) \).  
\end{proof}  

Let \( F \) be the Fr\'echet topology on the space of differential $k$-forms \( \B_k(U) \) and \(  \b(\B_k(U), \hB_k(U)) \) the strong topology of the dual pair \( (\B_k(U), \hB_k(U)) \).     
   
 The following theorem was first established in \cite{topological}: 
 
\begin{thm}\label{thm:inclu}  
	\( (\B_k(U), \b(\B_k(U), \hB_k(U))) = (\B_k(U), F) \). 
\end{thm}  

	We immediately deduce the fundamental topological isomorphism of cochains and forms by J. Harrison and H. Pugh in \cite{topological}, without which results in this paper involving the inductive limit \( \hB_k(U) \)  would have less substance.  
	  
\begin{thm}\label{thm:isot} 
		The space of differential \( k \)-cochains \( (\hB_k(U))' \) with the strong (polar) topology is topologically isomorphic to the Fr\'echet space of differential \( k \)-forms \( \B_k(U) \).
\end{thm}

We next describe an equivalent way to construct the inductive limit \( ( \hB_k(U), \t_k) \) is via direct sums and quotients due to K\"othe \cite{kothe}.  This is the approach we shall use for it provides a clear way to establish continuity of operators in the inductive limit.
\begin{defn}\label{hk}      
	For \( 0 \le k \le n \), let \( H_k = H_k(U) \) be the \emph{\textbf{hull}} of the linking maps, that is, the linear span of the subset \( \{ J^r - u_k^{r,s}(J^r) \,:\, J^r \in \hB_k^r(U), s \ge r \ge 0 \} \subset \oplus_r \hB_k^r(U) \). 
\end{defn} 

Endow \( \oplus_{r = 0}^\i \hB_k^r(U) \) with the direct sum topology. Then \( \oplus_r \hB_k^r(U) /H_k \), endowed with the quotient topology, is topologically isomorphic to \( ( \hB_k(U),\t_k) \) (\cite{kothe}, \S 19.2, p.219). Since \( ( \hB_k(U),\t_k) \) is Hausdorff (Theorem \ref{thm:frechet}), it follows that \( H_k \) is closed. (\cite{kothe}, (4) p.216). Hence the projection \( \pi_k: \oplus_r \hB_k^r(U) \to \oplus_r \hB_k^r(U) /H_k \) is a continuous linear map (\cite{kothe}, \S 10.7 (3),(4)). Since the canonical inclusion \( \nu_k^r: \hB_k^r(U) \to \oplus_r \hB_k^r(U) \) is   continuous, the inclusion \(u_k^r := \pi_k \circ \nu_k^r: \hB_k^r(U) \to ( \hB_k(U),\t_k) \) is continuous.  
 
Suppose \( T: \oplus_{k=0}^n \oplus_{r\ge 0} \hB_k^r(U) \to \oplus_{k=0}^n \oplus_{r\ge 0} \hB_k^r(U) \) is  continuous and bigraded with \( T(\hB_k^r(U)) \subset \hB_\ell^s(U) \), and \( T(H_k) \subset H_\ell \). By the universal property of quotients, \( T \) factors through a graded linear map \( \hat{T}: \oplus_k \oplus_r \hB_k^r(U)/H_k \to \oplus_k \oplus_r \hB_k^r(U)/H_k \) with \( T = \hat{T} \circ \pi \) where  \( \pi \) is the canonical projection onto the quotient. That is,  \( \hat{T}[H_k + J] = [H_k + T(J)] \).  Let   \( \hB(U) = \hB^\i(U) := \oplus_{k=0}^n \hB_k(U) \), endowed with the direct sum topology. 
 
\begin{thm}\label{thm:continuousoperators}  
	If \( T: \oplus_{k=0}^n \oplus_{r=0}^\i \hB_k^r(U) \to \oplus_{k=0}^n \oplus_{r=0}^\i \hB_k^r(U') \) is a  continuous bigraded linear map  with \( T(\hB_k^r(U)) \subset \hB_\ell^s(U') \), and \( T(H_k(U)) \subset H_\ell(U') \), then \( T \) factors through a continuous graded linear map  \( \hat{T}: \hB(U) \to \hB(U') \) with \( T = \pi \circ \hat{T} \).
\end{thm}

\begin{proof}
	 By the universal property of quotients, \( T \) factors through a graded linear map \( \hat{T}: \oplus_k \oplus_r \hB_k^r(U)/H_k(U) \to \oplus_k \oplus_r \hB_k^r(U')/H_k(U') \) with \( T = \hat{T} \circ \pi \) where  \( \pi \) is the canonical projection onto the quotient. That is,  \( \hat{T}[H_k(U) + J] = [H_k(U') + T(J)] \). 	Furthermore, \( \hat{T}: \hB(U) \to \hB(U') \) is continuous and graded (\cite{kothe} \S 19.1 (7), p. 217).   
\end{proof}      

A dual result holds for projective limits and can be found in \cite{kothe} (\S 19.6 (6), p. 227), replacing  locally convex hulls with locally convex kernels and reversing all the arrows, but we do not need such formality.

\section{Density functions, partitions of unity, and support} 
\label{sec:multiplication_by_a_function_and_partitions_of_unity}

\subsection{Change of density}\label{sub:the_operator}

We prove that the topological vector space of differential chains \( \hB(U) \) is a graded module over the ring of functions \( \B_0(U) \).   

\begin{lem}\label{lem:X} 
	The Banach space \( \B_0^r(U) \) is a ring with unity via pointwise multiplication of functions. Specifically, if \( f, g \in \B_0^r(U) \), then \( \|f \cdot g\|_{B^r} \le n 2^r \|f\|_{B^r} \|g\|_{B^r} \). 
\end{lem}

\begin{proof} 
	The unit function \( u(x) = 1 \) is an element of \( \B_0^r(U) \) for \( r \ge 0 \) since \( \|u\|_{B^r} = 1 \). The proof \( \|f \cdot g\|_{B^r} \le 2^r \|f\|_{B^r} \|g\|_{B^r} \) for \( n = 1 \) is a straightforward application of the product rule and we omit it.  The general result follows by taking coordinates.  
\end{proof}  

Let \( {\cal L}(\hB(U)) \) be the algebra of operators on \( \hB(U) \) and \( {\cal L}(\B(U)) \) the algebra of operators on \( \B(U) \).  
 
\begin{defn}\label{def:mf}
	Define the bilinear map \emph{\textbf{multiplication by a function}}
	\begin{align*}
	 m: \B_0^r(U) \times \A_k(U) &\to \A_k(U)\\(f, (p;\a)) &\mapsto (p; f(p)\a)
	\end{align*} where \( (p;\a) \) is an arbitrary \( k \)-element with \( p \in U \) and extend linearly to \( \B_0^r(U) \times \A_k(U) \).
	Denote \( m_f(A):= m(f, A) \). Denote the dual operator by \( f \cdot \in {\cal L}(\B(U)) \) where \( f \in \B_0^r(U) \).   
	
\end{defn}
                    
\begin{thm}\label{thm:mf} 
 	Let \( U \) be open in \( \R^n \), \(0 \le k \le n \), and \( r \ge 0 \).	The bilinear map \( m = m_k^r: \B_0^r(U) \times \A_k(U) \to \A_k(U) \) extends to a continuous bilinear map \( m = m_k^r: \B_0^r(U) \times \hB_k^r(U) \to \hB_k^r(U) \) with \[ \|m(f,J) \|_{B^{r}} \le   n 2^r\|f\|_{B^{r}}\|J\|_{B^{r}} \mbox{ for all } J \in \hB_k^r(U). \]    For each \( 0 \le k \le n \), there exists a separately continuous\footnote{Joint continuity of the bilinear operators on the inductive limit is an open question.} bilinear map \( m = m_k:\B_0(U) \times \hB_k(U) \to \hB_k(U) \) which restricts to \( m_k^r \) on each \(  \B_0^r(U) \times u_k^r(\hB_k^r(U)) \).  
\end{thm}

\begin{proof} 
  We first prove the inequality for nonzero \( A \in \A_k(U) \). By Lemma \ref{lem:X}    \[ \frac{|\cint_{m_f A} \o|}{\|\o\|_{B^r}} \le \frac{\|f \cdot \o\|_{B^r} \|A\|_{B^r}}{\|\o\|_{B^r}} \le nr\|f\|_{B^r}\|A\|_{B^r}. \]  The inequality follows since \( \|m_f A\|_{B^r} = \sup \frac{|\cint_{m_f A} \o|}{\|\o\|_{B^r}} \).  We can therefore extend \( m_k^r \) to \( \hB_k^r(U) \) via completion, and the operator \( m_k^r \) will still satisfy the inequality.

	For the inductive limit,  first fix \( f \in \B_0(U) \).  Since \( m_f(\hB_k^r) \subset \hB_k^r \) and \( m_f(H_k(U) \subset H_k(U))  \), we can apply Theorem \ref{thm:continuousoperators} to extend \( m_f:\hB_k^r(U) \to \hB_k^r(U) \) to a continuous linear map \( m_f:\hB_k(U) \to \hB_k(U) \).  Define \( m:\B_0(U) \times \hB_k(U) \to \hB_k(U) \) by \(  m(f,J) := m_f (J) \).  This is well-defined since \( f \in \B_0(U) \) implies \( f \in \B_0^r(U) \).   This establishes continuity in the second variable.    
	
 Continuity in the first variable:    
	   Fix \( J \in \hB_k(U) \).  There exists \( r \ge 0 \) and \( J^r \in \hB_k^r(U) \) such that \( u_k^r(J^r) = J \). Suppose \( f_i \to 0 \) in \( \B_0(U) \).  Since the inclusion \( \B_0(U) \to \B_0^r(U) \) is the identity map and continuous,  then \( f_i \to 0 \) in \( \B_0^r(U) \).  Thus \( m(f_i, J) = m_{f_i}(J) \le  n 2^r\|f_i\|_{B^{r}}\|J\|_{B^{r}} \to 0  \). 
 
\end{proof}

\begin{thm}[Change of density]\label{thm:changeofden} 
The space of differential chains \( \hB(U) \) is a graded module over the ring \( \B_0(U) \) satisfying
	\begin{equation}
		\cint_{m_f J} \o = \cint_J f \cdot \o
	\end{equation} 
	for all matching pairs  \( J \in \hB_k^r(U) \), \( \o \in \B_k^r(U) \), and \( 0 \le r \le \i  \).
\end{thm}

\begin{proof}                                                            
	Since \( \o(m_f (p;\a)) = f \cdot \o (p;\a)) \) the integral relation holds for  Dirac chains, and thus to pairs of all chains and forms of matching class by continuity of the integral pairing of Theorem \ref{thm:mf} (for \( 0 \le r < \i \)) and Corollary \ref{thm:integralpair} (for \( r = \i \)).   
\end{proof}

We can actually say a bit more about the continuity of \( m_f \) with respect to \( f \in \B_0^r(U)\).  Convergence in \( \B_0^r(U) \) is restrictive as it does not lend itself nicely to bump functions and partitions of unity.  In particular, if \( \{\phi_i\} \) is a partition of unity, then \( \sum_{i=1}^{N} \phi_i \nrightarrow 1  \)   in the \( B^r \) norm.   In \S\ref{sub:support} we need \( m_f \) to be continuous under a more local notion of convergence (the \( B^r \) version of the compact-open topology.)   The following lemma and its proof are due to H. Pugh. 

\begin{lem}\label{lem:hpugh}
	Let \( f_i,\, f \in \B_0^r(U),\, r \ge 0 \), such that \( f_i \to f \) pointwise and \( \| f - f_i \|_{B^r} < C_r \) for some \( C_r > 0 \) independent of \( i \).  Then \( m_{f_i} J \to m_f J \) in the \( B^r \)-norm for all \( J \in \hB_k^r(U) \).   Suppose \( J \in \hB_k(U) \).  If \( f_i, f \in \B_0(U) \), \( f_i \to f \) pointwise and there exists \( C_r \) with \( \| f - f_i \|_{B^r} < C_r \) for all \( r \ge 0 \), then \( m_{f_i} J \to m_f J \) in \( \hB_k(U) \) .  
\end{lem}

\begin{proof}
	Let \( \e>0 \). We show there exists \( N \) such that if \( i>N \) then \( \| m_f J - m_{f_i} J \|_{B^r} < \e \). Suppose not.  Then for all \( N \) there exists \( i_N > N \) and \( \o_N \in \B_k^r(U)\) with \( \| \o_N \|_{B^r} = 1 \) such that \[ ((f-f_{i_N}) \o_N)(J) \ge \e. \]    (This is by the definition of the \( B^r \) norm on chains as a supremum over forms of norm \( 1 \).)	Now, let \( A_j \to J \) be Dirac chains.  Then there exists \( M \) such that if \( i>M \), then \( \o(J)- \o(A_j) < \e/2 \) for all \( \o\in \B_k^r(U)\) with \( \|\o\|_{B^r} < nrC \).  	Therefore, putting these two together, and using the fact that \( \|((f-f_{i_N}) \o_N)\|_{B^r} \le nr \| f - f_{i_N} \|_{B^r} \|\o_N\|_{B^r} < nrC \), we have \[ ((f-f_{i_N}) \o_N)(A_j) > \e/2 \] for all \( N \).  Now, fix such a \( j \).  Say \( A_j = \sum_{s=1^K} (p_s;\a_s) \).  Then \[ ((f-f_{i_N}) \o_N)(A_j) = \sum_s (f(p_s)-f_{i_N}(p_s))\o_N(p_s;\a_s). \] Since \( f_i\to f \) pointwise, we can make \( N \) large enough such that \[ \left| f(p_s)-f_{i_N}(p_s)\right| < \frac{\e}{2K\max_s\|\a_s\|_0}\] for all \( s \). Therefore, \[ \sum_s (f(p_s)-f_{i_N}(p_s))\o_N(p_s;\a_s) < \e/2, \] which is a contradiction.
	The last part follows by applying the first to \( J^r \in \hB_k^r(U) \) with \( u_k^r(J^r) = J \).  Then \( m_{f_i} J \to m_f J \) in the \( B^s \) norm for each \( s \ge r \).  
\end{proof}
 
\subsection{Partitions of unity}\label{ssub:partitions_of_unity}     
Since differential chains are a module over the ring \( \B_0 \) it is possible to form partitions of unity.

\begin{defn}
	A partition of unity \( \{f_i\} \) subordinate to a collection of open sets \( \{U_i\} \) is said to be \emph{\textbf{uniform}} if there exists a constant \( C_r > 0  \) such that 
	 if \( \|f_i -1\|_{B_r}  \le C_r\).  
\end{defn}

\begin{thm}\label{thm:pou} 
	Suppose \( \{U_i\}_{i=1}^\i \) is a locally finite bounded open cover of \( U \) and \( \{\Phi_i\}_{i=1}^\i \) is a uniform partition of unity subordinate to \( \{U_i\}_{i=1}^\i \) with \( \Phi_i \in \B_0^r(U) \). If \( J \in \hB_k^r(U) \), then \[ J = \sum_{i=1}^\i (m_{\Phi_i} J) \]  where the convergence is in  the \( B^r \) norm.  If \( J \in \hB_k(U) \), then \[ J = \sum_{i=1}^\i (m_{\Phi_i} J) \]  where the convergence is in  the inductive limit topology.
\end{thm}

\begin{proof}  	
 Apply  Pugh's Lemma \ref{lem:hpugh} to \( \sum_{i=1}^N \Phi_i \) and the function \( 1 \in \B_0(U) \).  Then \( \sum_{i=1}^N (m_{\Phi_i} J) =  (\sum_{i=1}^N m_{\Phi_i}) J  \to J \) in the \( B^r \) norm, assuming \( J \in \hB_k^r(U) \), or the inductive limit topology, assuming \( J \in \hB_k(U) \). 
\end{proof}

We immediately deduce

\begin{cor}[Fundamental Lemma]\label{cor:fundamentallemma}
If \( J \in \hB_k^r(U) \) satisfies \( m_fJ = 0 \) for all \( f \in \B_0(U) \) with compact support, then \( J = 0 \).   
\end{cor}

This implies the fundamental lemma in the calculus of variations.   For \( 0 = \cint_U f\cdot g dV = \cint_{m_f \widetilde{U}} g dV  \) for all \( g \) implies \( m_f \widetilde{U} = 0 \).   
 
\subsubsection{Partitions of unity and integration} 

If \( I = (0,1) \) and \( \{\phi_s\} \) is a partition of unity subordinate to a covering of \( I = \cup_{s=1}^k I_s \), then we can write the chain representative \( \widetilde{I} \) of \( I \) as a sum \( \widetilde{I} = \sum_{s=1}^k m_{\phi_s} \widetilde{I_s} \). This ties the idea of an atlas and its overlap maps to integration.  Roughly speaking, the overlapping subintervals have varying density, but add up perfectly to obtain the unit interval. We obtain \( \cint_{\widetilde{I}} \o = \sum_{s=1}^k \cint_{m_{\phi_s} \widetilde{I_s} } \o \). This is in contrast with our discrete method of approximating \( \widetilde{I} \) with Dirac chains \( A_m \) to calculate the same integral \( \cint_{\widetilde{I}} \o = \lim \cint_{A_m} \o \).

\subsection{Support of a chain} 
\label{sub:support}
  We show that associated to each nonzero chain \( J \in \hB_k(U) \) is a well-defined nonempty closed subset \( \supp(J) \subseteq \overline{U} \) called the \emph{\textbf{support}} of \( J \). 

\begin{defn}
	Let \( \O_\e(p) := \{ q \in \R^n \,:\, \|p-q\|< \e\} \), the open ball of radius \( \e > 0 \) about \( p \).  
\end{defn}

\begin{defn}\label{support}
		If \( J \in \hB_k^r(U) \) and \( 1 \le r \le \i \), let \[ \supp(J) := \left\{ p \in \overline{U} \,:\, \mbox{ for all  } \e > 0, \mbox{ there exists } \eta \in \B_k^r(U), \supp(\eta) \subset \O_\e(p),   \mbox{ and } \cint_J \eta \ne 0  \right\}. \]  This definition agrees with our previous definition of support on the subspace of Dirac chains \( \A_k(U) \).
\end{defn}

\begin{prop}\label{pro:supportsame}
Let \( J \in \hB_k^r(U) \).  Then \( \supp(J)= \supp(u_k^{r,s}(J)) = \supp(u_k^r(J)) \) for all \( 0 \le r \le s \). 
\end{prop}

\begin{proof} 
Suppose \( p \in \supp(J)  \) and \( \e > 0 \).  Then there exists \( \eta \in \B_k^r(U) \) with \( \supp(\eta) \subset \O_\e(p) \) and \( \cint_J \eta \ne 0 \). By Theorem \ref{thm:injection} \( \eta \) can be approximated by \( \xi \in \B_k(U) \) with \( \supp(\xi) \subset \O_\e(p) \) and \( \cint_J \xi \ne 0 \).  Thus \( p \in \supp(u_k^r(J)) \), as well as \( \supp(u_k^{r,s}(J)) \).   
\end{proof}
      By working with forms supported in \( \e \)-neighborhoods of points, it is not hard to see that the support of a chain agrees with our previous definition of the support of a Dirac chain \( A \).

If \( J = 0 \), then \( \supp(J) = \emptyset \).  	If \( X \) is any subset of \( \overline{U} \) with \( \supp(J) \subseteq X \), then we say \( J \) is \emph{\textbf{supported in \( X \)}}. If  \( \supp(J) \) is compact,  we say that \( J \) has \emph{\textbf{compact support}}.

\begin{thm}\label{thm:nonzero}
	If \( J \) is a nonzero differential chain in \( U \), then \( \supp(J) \) is a closed, nonempty set.
\end{thm}

\begin{proof} 
Suppose \( J \in \hB_k^r(U) \)  and \( \supp(J) = \emptyset \). 
Choose a locally finite cover of \( \{\O_1(p)\}_p \) of \( U = \supp(J)^c \).    Let \( \{f_i\} \) be a uniform partition of unity subordinate to this cover.  Then \( \sum_i m_{f_i} J \) converges to \( J \) in the \( B^r \) norm by Theorem \ref{thm:mf}.  Hence \( \cint_J \o = \sum_i \cint_{m_{f_i} J} \o = \sum_i \cint_J f_i \cdot \o = 0 \) for all \( \o \in \B_k^r(U) \).  Thus \( J = 0 \). Now suppose \( J \in \hB_k(U) \) and \( \supp(J) = \emptyset \).  Then there exists \( J^r \in \hB_k^r(U) \) with \( u_k^r(J^r) = J \).  By Proposition \ref{pro:supportsame} we know \( \supp(J) = \supp(J^r) = \emptyset \).  By the first part,  \( J^r = 0 \) and thus \( J = u_k^r(J^r) = 0 \). 

We show \( \supp(J) \) is closed:  Suppose \( p_i \to p \) and \( p_i \in \supp(J) \). Let \( \e > 0 \).   Then \( d(p_i,p) < \e/2 \) for sufficiently large \( i \).  Since \( p_i \in \supp(J) \), there exists \( \eta \in \B_k^r(U) \), respectively, \( \eta \in \B_k(U) \), supported in \( \O_{\e/2}(p) \) such that  \( \cint_J \eta \ne 0 \).  Thus \( p \in \supp(J) \) since \( \O_{e/2}(p) \subset \O_\e(p) \).
\end{proof}

The complement of the support of a chain has an immediate characterization from Definition \ref{support}: 
	\[ \supp(J)^c = \left\{ p \in U \,:\, \mbox{ for all  } \O_\e(p) \cap \supp(J) = \emptyset, \mbox{ and all } \eta \in \B_k^r(U), \supp(\eta) \subset  \O_\e(p) \mbox{ then } \cint_J \eta = 0  \right\}. \]

\begin{prop}\label{prop:complement}
 If \( U \) is a bounded open set with \( \overline{U} \subset \supp(J)^c \), then \( \cint_J \eta = 0 \) for all differential forms \( \eta \) supported in \( U \). If \( U \cap \supp(J) \ne \emptyset \), there exists a differential form \( \eta \) supported in \( U \) with \( \cint_J \eta \ne 0 \).  
\end{prop}

\begin{proof} 
Suppose  \(\overline{B} \subset \supp(J)^c \).  Cover \( \overline{U} \) with finitely many open balls \( \{B_i\} \) and \( \{f_i\} \) a uniform partition of unity with respect to \( \{B_i\} \).  By Theorem \ref{thm:changeofden} \( \cint_J \eta = \sum \cint_{m_{f_i}J} \eta = \sum \cint_{J} f_i \eta = 0   \). 	

The second part follows directly from Definition \ref{support}.
\end{proof}
  
 	Support is a map from \( \hB_k^r(U) \) to the power set of \( \overline{U} \).  It is not linear, nor is it continuous in the Hausdorff metric. Consider, for example, \( A_t =(p; t\a) + (q; (1-t)\a) \). Then \( \supp(A_t) = \{p,q\} \) for all \( 0 < t < 1 \), but \( \supp(A_0) = q \) and \( \supp(A_1) =p \).   However, we can say something about the support of limits of differential chains:

\begin{thm}\label{thm:ifjas}
Suppose \( X \subset U \) is closed and  \( \supp(J_i) \subset X \) for a sequence of chains  \( J_i \to J \) in \( \hB_k^r(U) \) for some \( r \ge 0 \),  or \( J_i \to J \) in \( \hB_k(U) \).    Then \( \supp(J) \subset X \).
\end{thm}

\begin{proof} 
	If \( J = 0 \), we are finished since \( \supp(J) = \emptyset \) by Definition \ref{support}. Suppose \( J \ne 0 \).  Then \( \supp(J) \ne \emptyset \) by Theorem \ref{thm:nonzero}.

Assume \( J_i \to J \)in \( \hB_k^r(U), r \ge 0 \), or \( J_i \to J \) in\( \hB_k(U) \)
Suppose there exists \( p \in \supp(J) \) and \( p \notin X \).  There exists \( \e > 0 \) such that  \( \O_\e(p)  \cap  X = \emptyset \), and \( \eta \in \B_k^r(U) \) with \( \cint_J \eta \ne 0 \) and \( \supp(\eta) \subset \O_\e(p) \).   On the other hand, \( \cint_J \eta = \lim_{i \to \i} \cint_{J_i} \eta = 0 \) since \( \supp(J_i) \subset X \).  The proof for \( J_i \to J \) in \( \hB_k(U) \) is essentially the same.
\end{proof}
                       
\begin{prop}\label{prop:supa}
 Suppose \( J \in \hB_k^r(U) \), \( f \in \B_0^r(U) \) and \( r \ge 0 \), or \( J \in \hB_k(U) \) and \( f \in \B_0(U) \).  Then \[ \supp(m_f J) \subseteq \supp(f) \cap \supp(J). \]   
\end{prop}

\begin{proof} 
	Let \( p \in \supp(m_f J) \).  If \( p \notin \supp(f) \cap \supp(J) \), there exists \( \e > 0 \) with \( \O_\e(p) \cap \supp(f) \cap \supp(J) = \emptyset \) and \( \eta \in \B_k^r(U)\) (respectively, \( \eta \in \B_k(U) \)) with \( \cint_J f \cdot \eta = \cint_{m_f J} \eta \ne 0 \) and \( \supp(\eta) \subset \O_\e(p) \).  Thus \( p \in \supp(J) \).  
	If \( p \notin \supp(f) \), we can choose \( \e > 0 \) so that \( \O_\e(p) \cap \supp(f) = \emptyset \).  Then \( \cint_J f \cdot \eta = 0 \), which is a contradiction.  Hence \( p \in \supp(f) \cap \supp(J) \), as claimed.  
\end{proof}

\begin{prop}\label{prop:operator} 
	If \( T \) is a continuous operator on \( \hB(U) \) with \( \supp(T (A)) \subseteq \supp(A) \) for all \( A \in {\cal A}(U) \), then \( \supp(T(J)) \subseteq \supp(J) \) for all \( J \in \hB(U)\). 
\end{prop}

\begin{proof}  Let \( S \) be the dual operator on \( \B(U) \) given by  \( S \o := \o T \).  
Suppose \( p \in \supp(T(J)) \) and \( p \notin \supp(J) \). There exists \( \e > 0 \) such that  \( \O_\e(p) \cap \supp(J) = \emptyset \).  Then \( \cint_J \o = 0 \) for all \( \o \in \B_k(U) \) supported in \( \O_\e(p) \).   Since \( p \in \supp(T(J)) \), there exists \( \eta \in \B_k(U) \) supported in \( \O_\e(p) \) so that \( \cint_J S(\eta) =  \cint_{T(J)} \eta \ne 0 \).    This is a contradiction since \( S(\eta) \in \B(U) \). 
\end{proof}

This result also holds for operators \( T:\hB_k^r(U) \to \hB_j^s(U) \) with suitable modifications to the subscripts and superscripts.

\begin{examples}\mbox{}
	\begin{enumerate}
		\item   If \( J \in \hB_k^r(U) \) and \( f \in \B_0^r(U) \), then \( \supp(m_f J) \subset \supp(J) \) by Proposition \ref{prop:supa}. If \( 1/f \in \B_0^r(U) \), then \( \supp(m_f J) = \supp(J) \) since \( \supp(J) = \supp(m_{1/f}m_f J) \subset \supp(m_f J) \).  Thus \( \supp(m_f J) = \supp(J) \).
		
		\item Any closed set \( S \) supports a chain. Simply choose a countable dense subset \( \{ p_i \} \subseteq S \), and note that \( S \) supports the chain \( J= \sum_{i=1}^{\i}(p_i; 1/2^i) \). There are uncountably many distinct chains with support \( S \) since \( \supp (m_f J) \subset \supp(J) \) by Proposition \ref{prop:supa}.
	\end{enumerate}
\end{examples}

\begin{prop}\label{prop:supportW}
  If \( W \subset U \) is bounded and open, then \( \supp(\widetilde{W}) = \overline{W} \).   If \( \s \) is an affine \( k \)-cell in \( U \), then \( \supp(\widetilde{\s}) = \overline{\s} \). 
\end{prop}  

\begin{proof}
 By Theorem \ref{thm:opensets} \( \widetilde{W} \in \hB_n^1(U) \).
	We first show that \( \supp(\widetilde{W}) \subseteq \overline{W} \):  Suppose \( p \in \supp(\widetilde{W}) \) and \( p \notin \overline{W} \).  There exists \( \e > 0 \) such that \( \O_\e(p) \cap \overline{W} \ne \emptyset \).   Furthermore, there exists \( \eta \in \B_k^r(U) \) supported in \( \O_\e(p) \) with \( \int_W \eta =  \cint_{\widetilde{W}} \eta \ne 0 \) by Theorem \ref{thm:opensets}.   But \( \int_W \eta = 0 \) for all \( \eta \) supported in \( \O_\e(p) \).   

	We next show \( W \subset \supp(\widetilde{W}) \):  Suppose there exists \( p \in W \) and \( p \notin \supp(\widetilde{W}) \).  Then \( \cint_{\widetilde{W}} \o = 0 \) for all \( \o \in \B_n^1(U) \) supported away from \( \supp(\widetilde{W}) \).   Since \( p \in W \), there exists a differential chain \( \eta \in \B_n(U) \) supported in \( B_\e(p) \),  and \( \int_W \eta = \cint_{\widetilde{W}} \eta \ne 0 \) by Theorem \ref{thm:opensets}.  

	Now suppose \( p \in \overline{W} \) and \( p \notin \supp(\widetilde{W}) \).   There exists  \(  \e > 0 \) with \( \O_\e(p) \cap \supp(\widetilde{W}) = \emptyset \).  Therefore, there exists \( q \in W \cap \O_\e(p) \).  It follows that  \( q \in \supp(\widetilde{W}) \), and therefore \( \cint_{\widetilde{W}} \eta \ne 0 \) for some \( \eta \in \B_k^r(U) \) supported in \( \O_\e(p) \). On the other hand, \( \cint_{\overline{W}} \o = 0  \) for all \( \o \) supported in \( \O_\e(p) \).     

The result extends readily to affine \( k \)-cells by working within the affine subspace containing \( \s \) and applying what we just proved.
\end{proof}

\section{Pushforward and change of variables} 
\label{sub:pushforward_and_change_of_variables}

\begin{defn}\label{push}
	Suppose \( U \subseteq \R^n\) and \( U' \subseteq \R^m \) are open,  and \( F:U \to U' \) is a differentiable map. For \( p \in U \), and \( 1 \le k \le n \),  define \emph{\textbf{linear pushforward}} \( F_{p*}(v_1 \wedge \cdots \wedge v_k) := DF_p(v_1) \wedge \cdots \wedge DF_p(v_k) \) where \( DF_p \) is the total derivative of \( F \) at \( p \). For \( k = 0 \), set \( F_{p*}(m) := m \).  Define \emph{\textbf{pushforward}} \( F_*(p;\a) := (F(p), F_{p*}\a) \) for all \( k \)-elements \( (p;\a) \) and extend to a linear map \( F_*:\A_k(U) \to \A_k(U') \).
\end{defn}   
                                                                     
\begin{figure}[ht] 
	\centering 
	\includegraphics[height=2in]{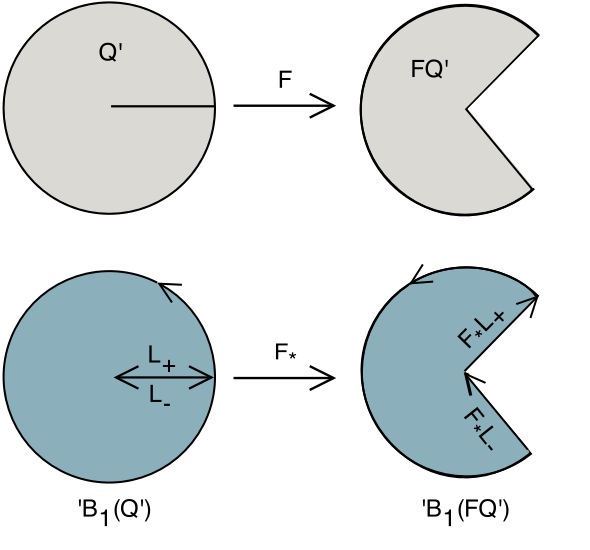}
\caption{A smooth map \( F \) that is not extendable to a neighborhood of \( \overline{Q'} \), and nevertheless determines a well-defined pushforward map \( F_*:\hB_1(Q') \to \hB_1(F(Q')) \)} \label{fig:Mapping} 
\end{figure}
                                                                         
 The classical definition of \emph{\textbf{pullback}} \( F^* \) satisfies the relation \( F^*\o (p;\a) = \o(F(p); F_* \a)= \o F_* (p;\a) \) for all differentiable maps \( F:U \to U' \), exterior forms \( \o \in (\A_k(U'))^* \), and simple \( k \)-elements \( (p;\a) \) with \( p \in U \).

\begin{defn}\label{def:Mr}  
	For \( r \ge 1 \), let \( {\cal M}^r(U,U') \) be the vector space of differentiable maps \( F:U \to U' \) whose coordinate functions \( F_i \) satisfy \( \p F_i/\p e_j \in B^{r-1}(U) \) for all \( j \).  Let \( {\cal M}^\i(U, U') \) be the vector space of differentiable maps \( F:U \to U' \) whose coordinate functions  \( F_i \) satisfy \( \p F_i/\p e_j \in B^{r-1}(U) \) for all \( j \) and all \( r \ge 1 \). 
\end{defn}

A map \( F \in {\cal M}^1(U, U') \) may not be bounded, but its directional derivatives must be. An important example is the identity map \( x \mapsto x \) which is an element of \( {\cal M}^1(\R^n,\R^n) \). Thus the space \( \B_0^r(U) \) is a proper subset of \( {\cal M}^r(U, \R) \). 

A \( 0 \)-form \( f_0 \in \B_0^r(U) \) trivially determines a map \( f(x):= f_0(x;1) \) in \( {\cal M}^r(U,\R) \), a multiplication operator \( m_f(p;\a) := (p; f(p)\a)  \) in the ring \( \B_0^r(U) \),   and a  linear map (pushforward) \( f_*: \hB_0^r(U)	 \to \hB_0^r(\R)) \) with \( f_*(p;\a) = (f(p); f_* \a) \).

\begin{defn}
For \( 1 \le r < \i \), define a seminorm on \( {\cal M}^r(U,U') \) by 
  \[ \rho_r(F)  := \max_{i,j} \{\|L_{e_j} F_{e_i}\|_{B^{r-1, U}}\}. \]
\end{defn} 
  
Endow  \( {\cal M}^r(U,U') \)  with the coarsest topology on \( {\cal M}^r(U,U') \) such that each map \( F \mapsto \rho_r(F-F_0) \) where \( F_0 \in {\cal M}^r(U,U') \) is continuous.  (This is similar to the compact-open topology, but this way we can avoid polynomial mappings beyond the identity.)   A base of neighborhoods of \( F_0 \in  {\cal M}^r(U,U') \) for this topology is obtained in the following way: for every \( \e > 0 \), let
    \[
    	U_{r,\e}(F_0) = \{F\in {\cal M}^r(U,U') \,:\, \rho_r(F - F_0) < \e \}.  
    \]
That the vector space operations are continuous in this topology follows from the definition of a seminorm. The resulting topological vector space is locally convex because each \( U_{r,\e}(0) \) is absolutely convex and absorbent.   Similarly, create a base of neighborhoods of  \( F_0 \in {\cal M}^\i(U,U') \) by using \( U_{r,\e}(F_0) \) for every \( r \ge 1 \) and every \( \e > 0 \).  
 
\begin{lem}\label{lem:Y}  
		If \( F \in {\cal M}^r(U,U') \) and \( 1 \le r < \i \),  then  \[ \|F^* \o\|_{B^{r,U}} \le n2^r \max \{ 1,\rho_r(F) \}\|\o\|_{B^{r,U'}} \] for all   \( \o \in \B_k^r(U') \).   
\end{lem}

\begin{proof}    
	For \( k = 0 \), the proof follows using the product and chain rules.  The general result follows by using coordinates.  (The actual constant \( n2^r \) which we obtained is not important, only that it is finite.)
\end{proof}	
  
\begin{thm}\label{thm:x}
 If \( F \in {\cal M}^r(U,U') \)  and \( 1 \le r < \i \), then
	\begin{equation}
		\|F_* A\|_{B^{r,U'}} \le n2^r \max \{ 1,\rho_r(F) \}\|A\|_{B^{r,U}}.
	\end{equation}
	for all \( A \in \A_k(U) \). 
\end{thm}

Suppose \( J \in \hB_k^r(U) \). Then \( J = \lim_{i \to \i} A_i \) for some \( A_i \in \A_k(U) \).  Since \( \{A_i\}_i \) is a Cauchy sequence, so is \( \{F_* A_i\}_i \) by Theorem \ref{thm:x}. Define \( F_*J := \lim_{i \to \i} F_* A_i \).  
   
\begin{thm}\label{thm:prop2}  
	Let   \( F \in {\cal M}^r(U,U') \), \(0 \le k \le n \), and \( r \ge 1 \).  The linear maps \(  F_*: \A_k(U) \to \A_k(U')  \) extend to continuous linear maps \( F_*:\hB_k^r(U) \to \hB_k^r(U') \) with 
	\[
		\|F_*J\|_{B^{r,U'}} \le n 2^r \max \{ 1,\rho_r(F) \} \|J\|_{B^{r,U}}.
	\] 
  Furthermore, for each \( 0 \le k \le n \), there exists a unique continuous operator \( F_*: \hB_k(U) \to \hB_k(U') \) which restricts to \( F_{k*} \) on each \( u_k^r(\hB_k^r(U))  \).
\end{thm}
 
\begin{proof}    Let \( J \in \hB_k^r(U) \).   By \eqref{basic}   and Lemma \ref{lem:Y},
\begin{align*}   
\|F_*J\|_{B^{r,U'}} &
= \sup_{0 \ne \o\in \B_k^r}\frac{\left| \cint_J F^* \o \right| }{\|\o\|_{B^r}} \\&
\le \sup_{0 \ne \o \in \B_k^r}\frac{ \|F^*\o\|_{B^{r,U}}}{\|\o\|_{B^r}} \|J\|_{B^{r,U}} \\&
\le C \max \{ 1, \rho_r(F) \} \|J\|_{B^{r,U}}.
\end{align*} 
If \( F(U) \subset U' \subset \R^m \), then \( F_*(\hB_k^r(U)) \subset \hB_k^r(U') \) and  \( F_*(H_k(U)) \subset H_k(U') \), so we may apply Theorem \ref{thm:continuousoperators} to uniquely extend the \( F_{k*} \) to \( F_*: \hB_k(U) \to \hB_k(U') \).
\end{proof}

\begin{example} 
	Suppose \( M \) is a smoothly embedded surface in \( \R^3 \).  
	Let \( F: M \to S^2 \) be the Gauss map.  Then \( F \) is as smooth as \( M \) and extends to a smooth map in a neighborhood \( U \) of \( M \).  Then \( F:U \to U' \) where \( U' \) is an open neighborhood of \( S^2  \) in \( \R^3 \).   The \emph{\textbf{shape operator}} is given by \( u \mapsto F_*u \) where \( u \in TM \).  Thus the shape operator is the pushforward operator \( F_* \) where \( F \) is the Gauss map. 
\end{example}

It is straightforward to see that if \( F \in {\cal M}^r(U_1, U_2),  G \in {\cal M}^r(U_2,U_3) \), then \( G_* \circ F_* = (G \circ F)_* \).

\begin{cor}[Change of variables]\label{cor:pull}  
	Suppose \( F \in {\cal M}^r(U,U') \) and \( 0 \le k \le n \). Then 
	\begin{equation}\label{eq:pushf}
				\cint_{F_*J} \o = \cint_J F^* \o 
	\end{equation} 
for all matching pairs \( J \in \hB_k^r(U) \), \( \o \in \B_k^r(U') \), and \( 1 \le r \le \i  \).
\end{cor}

\begin{proof} 
	This follows from Theorems \ref{thm:prop2} and \ref{thm:seminorm}. 
\end{proof}
  
 Compare \cite{lax1} \cite{lax2} which is widely considered to be the natural change of variables for multivariables.  The differential forms version of change of variables holds for \( C^1 \) diffeomorphisms \( F \) and domains of bounded open sets.  Our result presents a coordinate free version in arbitrary dimension and codimension \footnote{Whitney defined the pushforward operator \( F_* \) on polyhedral chains and extended it to sharp chains in \cite{whitney}. He proved a change of variables formula \eqref{eq:pushf} for Lipschitz forms. However, the important relation \( F_* \p = \p F_* \) does not hold for sharp chains since \( \p \) is not defined for the sharp norm. (See Proposition \ref{prop:whitneybd} below.) The flat norm of Whitney does have a continuous boundary operator, but flat forms are highly unstable. The following example modifies an example of Whitney found on p. 270 of \cite{whitney} which he used to show that components of flat forms may not be flat. But the same example shows that the flat norm has other problems. The author includes mention of these problems of the flat norm since they are not widely known, and she has seen more than one person devote years trying to develop calculus on fractals using the flat norm.   \textbf{Example}: In \( \R^2 \), let \( \o_t(x,y) =\begin{cases} e_1 + e_2 + tu, &x+y < 0\\ 0, & x+y > 0 \end{cases} \) where \( t \ge 0 \) and \( u \in \R^2 \) is nonzero. Then \( \o_0 \) is flat, but \( \o_t \) is not flat for any \( t > 0 \). In particular, setting \( t = 2, u = -e_2 \), we see that \( \star \o_0 \) is not flat.}.

\begin{prop}\label{pro:support2}  
 If \( F \in {\cal M}^r(U,U') \) is a closed map and  \( U \) is a bounded open set,  then \( \supp (F_*J) \subseteq F(\supp(J)) \subseteq U' \) for all \( J \in \hB_k^r(U) \) supported in \( U \).
\end{prop}

\begin{proof}  
 	Suppose \( \supp(J) \subset U \).    Suppose there exists \( p \in \supp(F_*J) \), but \( p\notin F(\supp(J)) \).  Let \( \O_\e(p) \) be an open ball disjoint from \( F(\supp(J)) \).  There exists \( \eta \)  supported in \( \O_\e(p) \) and \( \cint_{F_*J} \eta \ne 0 \).     Then \( U = F^{-1}(\O_e(p)) \) is disjoint from \( \supp(J) \).  Hence \( \cint_J F^* \eta = 0 \) since \( F^* \eta \) is supported in \( U \).   

\end{proof}      

\subsection{Algebraic chains} 
\label{sub:algebraic_chains}

 An \emph{\textbf{algebraic \( k \)-cell}} in an open set  \( U' \subseteq \R^n \) is a differential \( k \)-chain \( F_* \widetilde{Q} \) where \( Q \) is an affine \( k \)-cell contained in \( U \subseteq \R^n\) and the map \( F:U \to U' \) is an element of \( {\cal M}^r(U,U') \). We say that \( F_* \widetilde{Q} \) is \emph{non-degenerate} if \( F:Q \to U' \) is a diffeomorphism onto its image. An \emph{algebraic \( k \)-chain \( A \) in \( U' \)} is a finite sum of algebraic \( k \)-cells \( A = \sum_{i=1}^N a_i F_{i*} \widetilde{Q_i} \) where \( a_i \in \R \).   An algebraic \( k \)-chain \( A = \sum_{i=1}^m F_{i*}\widetilde{Q_i} \) is \emph{\textbf{non-overlapping }} if for each \( 1 \le i \le j \le m \) then \( F_i(Q_i) \) and \( F_j(Q_j) \) intersect at most along their boundaries.                   

 Singular cells are not the same as algebraic cells. A \emph{\textbf{singular \( k \)-cell}} is defined to be a map of a closed \( k \)-cell \( G: Q \to U \) and \( G \) might only be continuous. For example, let \( G: [-1,1] \to \R \) be given by \( G(x) = x \) if \( x \ge 0 \) and \( G(x) = -x \) if \( x \le 0 \). This singular \( 1 \)-cell is nonzero, but the algebraic cell \( G_*(\widetilde{[0,1]}) = 0 \). This problem of singular cells vanishes in homology, but some algebraic and geometric properties of  algebraic chains might be lost after passing to homology.   

\begin{examples} \mbox{}  
	\begin{enumerate}   
		\item  A \( k \)-chain \( \widetilde{M} \in \hB_k^r(U) \) \emph{represents} a \( k \)-submanifold \( M \) of class \( C^{r-1+Lip} \) in \( U \) if \( \cint_{\widetilde{M}} \o = \int_M \o \) for all forms \( \o \in \B_k^r(U) \).   It is not hard to piece together a non-overlapping algebraic \( k \)-chain which represents a compact \( k \)-submanifold.
 		 \item The quadrifolium and Boys surface can be represented by algebraic chains (see Figure \ref{fig:Quadrifolium}). 
\begin{figure}[ht] \centering \includegraphics[height=1.25in]{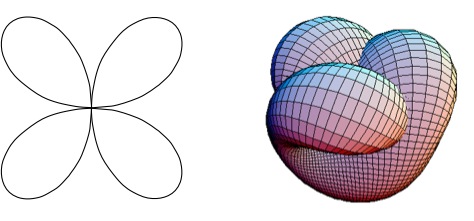} 
\caption{Quadrifolium } \label{fig:Quadrifolium} 
\end{figure}
			\item   Whitney stratified sets can be represented by algebraic chains since stratified sets can be triangulated \cite{goresky} (see Figure \ref{fig:WhitneyUmbrella}). 

				\begin{figure}[ht]
					\centering
						\includegraphics[height=1.5in]{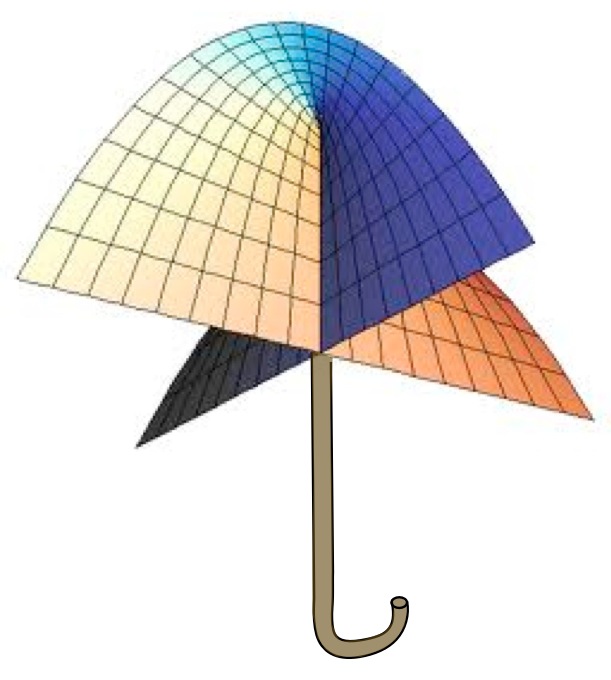}
					\caption{Whitney umbrella}
					\label{fig:WhitneyUmbrella}
				\end{figure}
 \end{enumerate} 
\end{examples}
   
\section{Primitive operators} 
\label{sec:primitive_operators}
  
In this section we define three ``primitive'' operators on the space \( \hB(U) \) of differential chains.  These operators are defined with respect to a vector field.

\subsection{Vector fields}\label{vectorfields}
 
Let \( V:U \to \R^n \)  be a vector field on \( U \).  Recall \( V^\flat \) is the differential \( 1 \)-form associated to \( V \) by way of the inner product\footnote{The choice of inner product on \( \R^n \) has no significant influence on the theory.  The spaces are independent of the choice, as the norms are comparable.}.         Let \( {\cal V}^r(U) \) be the Banach space of vector fields \( V \) defined on \( U \) such that \( V^\flat\in \B_1^r(U) \), equipped with the norm \( \|V\|_{B^r}:=\|V^\flat\|_{B^r} \).  It follows that the projective limit \( {\cal V}(U) = {\cal V}^\i(U) := \varprojlim {\cal V}^r(U) \), endowed with the projective limit topology, is a Fr\'echet space. 
 
\begin{lem}\label{lem:vectornorm} 
  \( V \in  {\cal V}^r(U) \) if and only if  each coordinate function \( f_i  \) satisfies \( \|f_i\|_{B^{r}} \le \|V^\flat\|_{B^{r}} \) for all \( 0 \le r <\i \).  
\end{lem} 

\begin{proof} 
The proof is immediate from the definitions.     
\end{proof}            

\subsection{Extrusion} \label{ssub:extrusion} 

\begin{defn} \label{extrusionX}
  Define the bilinear map, called  \emph{\textbf{extrusion}},  defined by its action on \( k \)-elements: 
\begin{align*}
E: {\cal V}^r(U) \times \A_k(U) &\to \A_{k+1}(U) \\ (V, (p;\a))  &\mapsto (p; V(p) \wedge \a).
\end{align*}  
Let \( E_V(p;\a) := E(V, (p;\a)) \). 
\end{defn}
 It follows that  \( i_V \o (p;\a) = \o E_V (p;\a) \) where \( i_V \) is classical interior product. 

\begin{thm}\label{thm:IIAV} 
	Let \(0 \le k \le n-1 \) and \( 1 \le r < \i \).   The bilinear map \( E = E_k^r: {\cal V}^r(U) \times \A_k(U) \to \A_{k+1}(U) \) extends to a continuous bilinear map \( E = E_k^r: {\cal V}^r(U) \times \hB_k^r(U) \to \hB_{k+1}^r(U) \) with  \[ \|E(V,J) \|_{B^{{r}}} \le n^2 2^r\|V^\flat\|_{B^{r}}\|J\|_{B^{{r}}} \mbox{ for all }  J \in  \hB_k^r(U). \] 
  Furthermore, for each \( 0 \le k \le n-1 \), there exists a separately continuous bilinear map \( E = E_k:{\cal V}^\i(U) \times \hB_k(U) \to \hB_k(U) \) which restricts to \( E_k^r \) on each \( \nu_k^r({\cal V}^r(U)) \times u_k^r(\hB_k^r(U)) \).  
\end{thm}  
 
\begin{proof} 
Let \( T(p;\a) =  E_u (p;\a) \) where \( u \) is a unit vector in \( \R^n \).  . Then
	\begin{align*} 
		\|T(\D_{\s^j} (p;\a))\|_{B^r} &= \| \D_{\s^j} T(p;\a) \|_{B^r} \\
		&= \|\D_{\s^j} (p; u \wedge \a) \|_{B^r} \le \|\s\|\|u\|\|(p;\a)\|_{r-j} \le \|u\|\|\s\|\|\a\|.
	\end{align*}
	By Corollary \ref{cor:opr} it follows that \( \|E_v(A)\|_{B^r} \le \|v\|\|A\|_{B^r} \). 
	 
  Since \( V \in {\cal V}^r(U) \) we know \( V = \sum f_i e_i \) where \( f_i \in \B_0^r(U) \). Lemma \ref{lem:extX} readily extends to \( f_i \in \B_0^r(U) \), and thus \( E_V = \sum_{i=1}^n E_{f_i e_i} = \sum_{i=1}^n m_{f_i} E_{e_i} \).
By Theorem \ref{thm:mf} \( \|m_{f_i} A\|_{B^{r}} \le n2^r \|f_i\|_{B^{r}} \|A\|_{B^{r}} \). Therefore, by Lemma  \ref{lem:vectornorm}, 
\begin{align*} 
	\|E_V A\|_{B^{r}} \le \sum_{i=1}^{n}
\|m_{f_i} E_{e_i} A\|_{B^{r}} \le n 2^r \sum_{i=1}^{n} \|f_i\|_{B^{r}} \|E_{e_i} A\|_{B^{r}}&\le n2^r \sum_{i=1}^{n} \|f_i\|_{B^{r}} \|A\|_{B^{r}} \\&\le n^2 2^r
\|V^\flat\|_{B^{r}}\|A\|_{B^{r}}. 
\end{align*} 
Define \( E_V(J) := \lim_{i \to \i} E_V(A_i) \) for \( J \in \hB(U) \)  (see Figure \ref{fig:Extrusion}).  It follows that  \( \|E(V,J) \|_{B^{{r}}} \le n^2r\|V^\flat\|_{B^{r}}\|J\|_{B^{{r}}} \). 

For the inductive limit,  first fix \( V \in {\cal V}^\i(U) \).  Since \( E_V(\hB_k^r(U)) \subset \hB_{k+1}^r(U) \) and \( E_V(H_k(U) \subset H_{k+1}(U))  \), we can apply Theorem \ref{thm:continuousoperators} to extend \( E_V:\hB_k^r(U) \to \hB_{k+1}^r(U) \) to a continuous linear map \( E_V:\hB_k(U) \to \hB_{k+1}(U) \).  Define \( E:{\cal V}^\i(U) \times \hB_k(U) \to \hB_k(U) \) by \(  E(V,J) := E_V (J) \).  This is well-defined since \( V \in {\cal V}^\i(U) \) implies \( V \in {\cal V}^r(U)  \). This establishes continuity in the second variable.

   Continuity in the first variable:    
   Suppose \( J \in \hB_k(U) \). Then there exists \( r \ge 0 \) and \( J^r \in \hB_k^r(U) \) such that \( u_k^r(J^r) = J \). Suppose \( V_i \to 0 \) in \( {\cal V}^\i(U) \).  Since the inclusion \( {\cal V}^\i(U) \to {\cal V}^{r+1}(U) \) is the identity map and continuous,  then \( V_i \to 0 \) in \( {\cal V}^r(U) \).   Hence \( E(V_i,J) \le n^2r \|V_i^\flat\|_{B^{r}}\|J\|_{B^{{r}}} \to 0   \).   
 \end{proof}  

\begin{thm}[Change of dimension I]\label{thm:extV}
	Let \(0 \le k \le n-1 \). Then
	\begin{equation}\label{primitiveext} \cint_{E_V J} \o = \cint_J i_V \o
	\end{equation} for all matching triples  \( V \in {\cal V}^r(U) \),   \( J \in \hB_k^r(U) \), \( \o \in \B_{k+1}^r(U) \),  and \( 1 \le r \le \i  \).
\end{thm} 

\begin{proof} 
 This follows since \( \o(E_V (p;\a)) = i_V \o (p;\a)) \) and by continuity of \( E_V \) and \( i_V \). 
\end{proof}
 
If \( S \) and \( T \) are operators, let \( [S,T] = ST - TS \) and \( \{S, T\} = ST + TS \).  The following relations follow directly from the definitions on \( k \)-elements and continuity.

\begin{prop}\label{prop:EE} 
	Let \( V,W \in {\cal V}^r(U) \). Then 
	\begin{enumerate} 
		\item \( E_V^2= 0 \); 
		\item \( \{E_V, E_W\} = 0 \);
		\item \label{lem:extX}\( E_{fV} = m_f E_V \) for all \( V \in  {\cal V}^r(U) \) and \( f \in \B_0^r(U), r \ge 1\). 
	\end{enumerate} 
\end{prop}      
 
\begin{figure}[ht] 
	\centering 
	\includegraphics[height=1.5in]{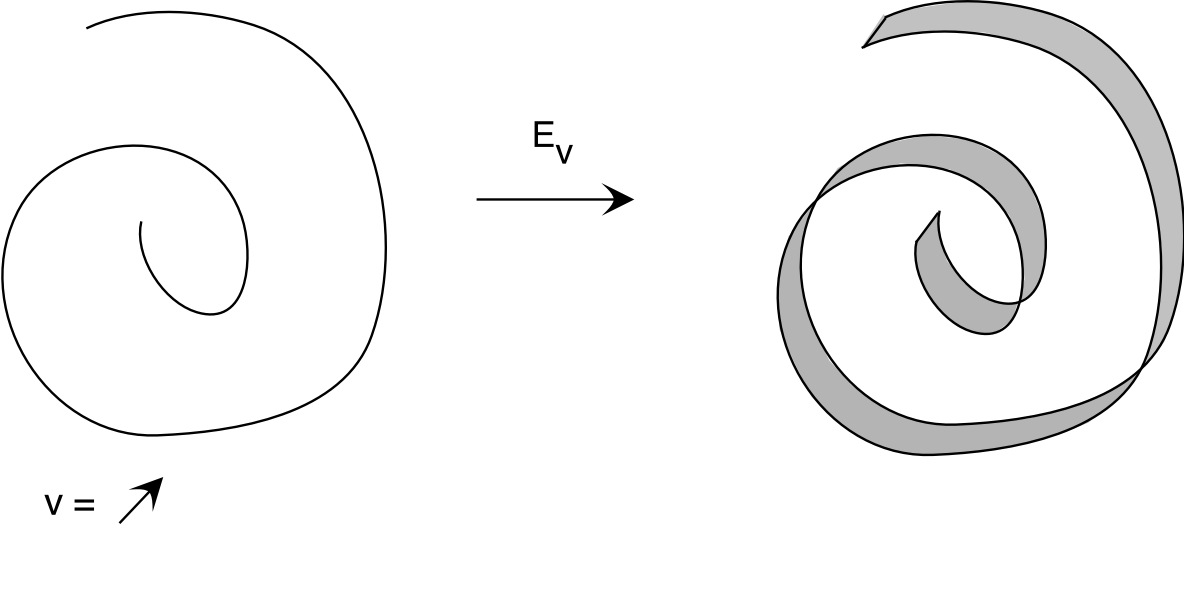}
	\caption{Extrusion of a curve through the vector field \( \frac{\p}{\p x} + \frac{\p}{\p y} \)} 
	\label{fig:Extrusion} 
\end{figure}
       
\subsection{Retraction} 
\label{sub:retraction}  
 
 Let \( 1 \le k \le n \) and  \( v \in V = \R^n \). Define \( M_v(x_1, \dots, x_k)  : = \sum_{i=1}^k (-1)^{i+1} \<v,x_i\>(x_1 \wedge \dots \wedge \hat{x}_i \wedge \dots \wedge x_k)  \). The following lemma is a standard result in multilinear algebra:
	    \begin{lem}\label{lem:Mv}  \( M_v \) extends to a multilinear alternating map
	     \( M_v: V^{\times k} \to \L_{k-1}(V) \).
	    \end{lem}
		
		\begin{prop}\label{prop:Mv}
		There exists a unique linear map  \( E_v^\dagger: \L_k(V) \to \L_{k-1}(V) \) satisfying  \[ E_v^\dagger(x_1 \wedge \cdots \wedge x_k) = \sum_{i=1}^k (-1)^{i+1} \<v,x_i\>(x_1 \wedge \dots \wedge \hat{x}_i \wedge \dots \wedge x_k)  \] for all  \( x_1, \dots, x_k \in V \).
		\end{prop}

		\begin{proof} 
		This follows from the universal property of exterior product (see, e.g., \cite{linearalgebra})
		\end{proof}
		
		The operator  \( E_v^\dagger \) extends to Dirac chains and is bounded since  \( E_v^\dagger(p;\a) \le \|v\|\|\a\| \).


\begin{defn}\label{retractionX}
	 Define the bilinear map, called  \emph{\textbf{retraction}}, by its action on simple \( k \)-elements 
	\begin{align*}
	E^\dagger: {\cal V}^r(U) \times \A_{k+1}(U) &\to \A_k(U) \\	(p; v_1 \wedge \cdots \wedge v_{k+1})
		&\mapsto \sum_{i=1}^{k+1} (-1)^{i+1} \<V(p),v_i\> (p; v_1 \wedge \cdots \wedge
		\widehat{v_i} \wedge \cdots \wedge v_{k+1}).
	\end{align*}  Let \( E_V^\dagger(p;\a) := E^\dagger(V, (p;\a))   = \sum_{i=1}^{k+1} (-1)^{i+1}\<V(p),v_i\> (p; v_1 \wedge \cdots \hat{v_i} \cdots \wedge v_k). \)
\end{defn}       

	\begin{lem}\label{lem:evdag}
		The operator dual to \( E_V^\dagger \) is \( V^\flat \wedge \cdot \).
	\end{lem}

	\begin{proof} 
		It is enough to show that \( \o E_v^\dagger (p;\a) = v^\flat \wedge \o (p;\a) \) for all \( k \)-elements \( (p;\a) \) and all \( v \in \R^n \). This follows from the standard formula of wedge product of differential forms:  \[ \o E_v^\dagger (p;\a) = \sum_{i=1}^n (-1)^{i+1} \o(p; \<v,v_i\>\a_i) = v^\flat \wedge \o (p;\a). \] 
	\end{proof}               
                       
\begin{thm}\label{thm:retX} 
	Let \(1 \le k \le n\) and \( r \ge 1 \).  The bilinear map \( E^\dagger =( E^\dagger)_k^r: {\cal V}^r(U) \times \A_k(U) \to \A_{k-1}(U) \) extends to a continuous bilinear map 
\begin{align*} E^\dagger = (E^\dagger)_k^r: {\cal V}^r(U) \times \hB_k^r(U) &\to \hB_{k-1}^r(U)  \end{align*} 
  with \[ \|E^\dagger(V, J)\|_{B^{r}} \le k {n \choose k} \|V^\flat\|_{B^{r}} \|J\|_{B^{r}}. \] Furthermore, for each \( 1 \le k \le n \), there exists a continuous bilinear map \( E^\dagger = E_k^\dagger:{\cal V}^\i(U) \times \hB_k(U) \to \hB_{k-1}(U) \) which restricts to \( (E^\dagger)_k^r \) on each \( u_k^r(\hB_k^r(U)) \). 
\end{thm}   

\begin{proof} 
	
		Let \( T \in {\cal L}(\A) \) be the operator determined by \( T(p;\a) = E_v^\dagger  (p;\a)/\|v\| \) for all \( k \)-elements \( (p;\a) \). By Corollary \ref{cor:opr} it suffices to show that \( \|T(\D_{\s^j}(p;\a))\|_{B^r} \le C \|\s\|\|\a\| \) for all \( 0 \le j \le r \). This follows since \[ \|T(\D_{\s^j}(p;\a))\|_{B^r} = \|E_v^\dagger(\D_{\s^j}(p;\a)/\|v\|\|_{B^r} \le \sum_{i=1}^k |(-1)^{i+1} \<v,v_i\> \D_{\s^j} (p; \hat{\a}_i) |_{B^j} \le k \|\s\|\|\a\| \] using \( \|\<u,v_i\>\hat{\a}_i\| \le \|\a\| \).  Hence  \(  \|E_v^\dagger A\|_{B^r} \le k \|v\|\|A\|_{B^r}   \).


	 Since \( V \in {\cal V}^r(U) \) we know \( V = \sum f_i e_i \) where \( f_i \in \B_0^r(U) \). Lemma \ref{lem:extX} readily extends to \( f_i \in \B_0^r(U) \), and thus \( E_V^\dagger = \sum_{i=1}^n E_{f_i e_i}^\dagger = \sum_{i=1}^n m_{f_i} E_{e_i}^\dagger \).
	By Theorem \ref{thm:mf} \( \|m_{f_i} A\|_{B^{r}} \le n2^r \|f_i\|_{B^{r}} \|A\|_{B^{r}} \). Therefore, by Lemma  \ref{lem:vectornorm}, 
	\begin{align*} 
		\|E_V^\dagger A\|_{B^{r}} \le \sum_{i=1}^{n}
	\|m_{f_i} E_{e_i}^\dagger A\|_{B^{r}} \le n 2^r \sum_{i=1}^{n} \|f_i\|_{B^{r}} \|E_{e_i}^\dagger A\|_{B^{r}}&\le n2^r  \sum_{i=1}^{n} \|f_i\|_{B^{r}}k {n \choose k} \|A\|_{B^{r}} \\&\le n^2 2^r k {n \choose k}
	\|V^\flat\|_{B^{r}}\|A\|_{B^{r}}. 
	\end{align*} 
	Define \( E^\dagger(V,J) = E_V^\dagger(J) = \lim_{i \to \i} E_V^\dagger(A_i) \) for \( J \in \hB(U) \) (see Fig. \ref{fig:Retraction}).  It follows that  \( \|E^\dagger(V,J) \|_{B^r} \le n^2rk {n \choose k}\|V^\flat\|_{B^{r}}\|J\|_{B^{{r}}} \). 

	For the inductive limit,  first fix \( V \in {\cal V}^\i(U) \).  Since \( E_V^\dagger(\hB_k^r(U)) \subset \hB_{k+1}^r(U) \) and \( E_V^\dagger(H_k(U) \subset H_{k+1}(U))  \),  we can apply Theorem \ref{thm:continuousoperators} to extend \( E_V^\dagger:\hB_k^r(U) \to \hB_{k-1}^r(U) \) to a continuous linear map \( E_V^\dagger:\hB_k(U) \to \hB_{k-1}(U) \).     This establishes continuity in the second variable.

	   Continuity in the first variable:    
	   Suppose \( J \in \hB_k(U) \). Then there exists \( r \ge 0 \) and \( J^r \in \hB_k^r(U) \) such that \( u_k^r(J^r) = J \). Suppose \( V_i \to 0 \) in \( {\cal V}^\i(U) \).  Since the inclusion \( {\cal V}^\i(U) \to {\cal V}^{r+1}(U) \) is the identity map and continuous,  then \( V_i \to 0 \) in \( {\cal V}^r(U) \).  Hence \( E^\dagger(V_i,J) \le n^2rk {n \choose k}\|V_i^\flat\|_{B^{r}}\|J\|_{B^{{r}}} \to 0   \).

\end{proof}

\begin{thm}[Change of dimension II]\label{thm:retXint} 
	Let \(1 \le k \le n\). Then
	\begin{equation}\label{primitiveretractV} 
		\cint_{E_V^\dagger J} \o = \cint_J V^\flat \wedge \o
	\end{equation} 
for all matching triples  \( V \in {\cal V}^r(U) \),   \( J \in \hB_k^r(U) \), \( \o \in \B_{k-1}^r(U) \),  and \( 1 \le r \le \i  \).
\end{thm} 

The proof is similar to that of Corollary \ref{thm:extV}.    

\begin{figure}[ht] 
	\centering
	\includegraphics[height=2in]{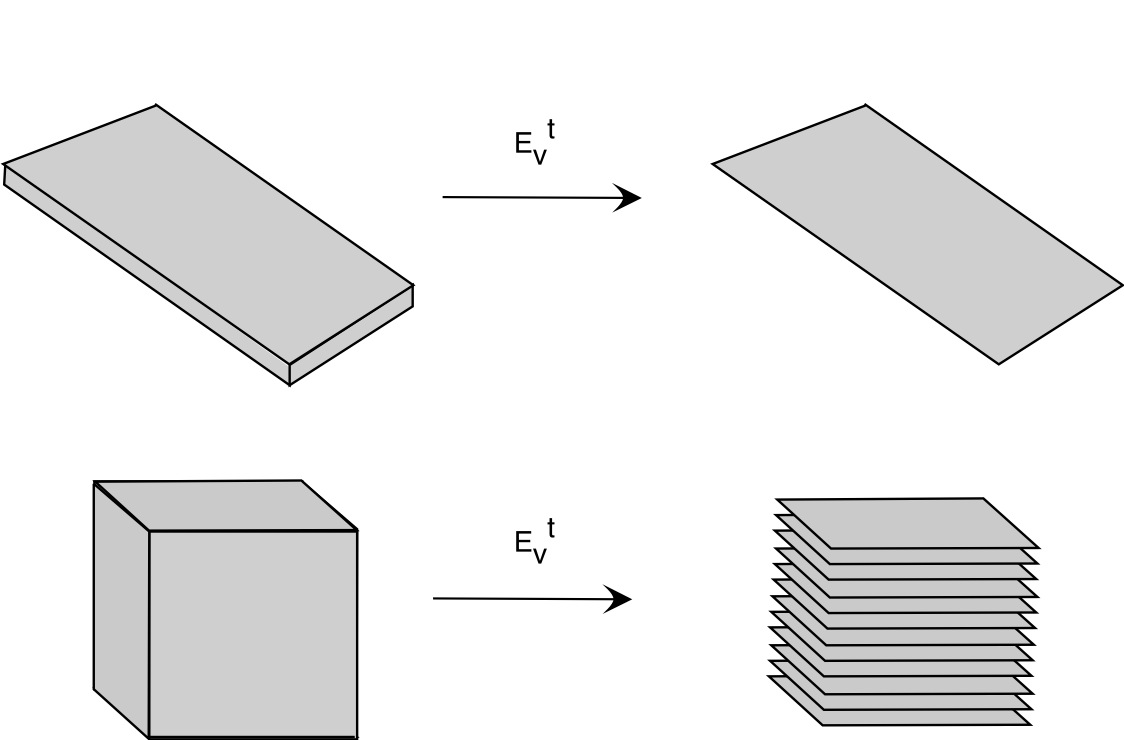} 
	\caption{Retraction of an extruded rectangle and a unit cube} 
	\label{fig:Retraction} 
\end{figure}
    Recall the inner product \( \<\cdot, \cdot\>_\wedge \) in \S\ref{sub:mass_norm}.

\begin{prop}\label{pro:asgg}
	The following relations hold  for all \( V, W  \in {\cal V}^r(U) \) and \( f \in \B_0^r(U), r \ge 1\):
	\begin{enumerate} 
  	\item \( \< E_V (p;\a), (p;\b) \>_\wedge = \< (p;\a), E_V^\dagger(p;\b) \>_\wedge \) for all \( k \)-vectors \( \b \) and \( (k-1) \)-vectors \( \a \);
		\item \( E_V^\dagger \circ E_V^\dagger = 0 \);
		\item \( [E_V^\dagger, E_W^\dagger] = 0 \); 
		\item \label{car} \( \{E_V^\dagger, E_W\} = \{E_V, E_W^\dagger\} = \<V,W\> Id \);
		\item \( (E_V + E_V^\dagger)^2 = \<V,V\> Id \);  
		\item \( [m_f, E_V] = [ m_f,E_V^\dagger] =   0 \);
		\item\label{lem:extXd} \( E_{fV}^\dagger = m_f E_V^\dagger \);
	\end{enumerate}
\end{prop} 

\begin{proof}  
  Property (a) comes from the well-established result that \( E_V^\dagger \) is the adjoint to \( E_V \) as operators on the exterior algebra.  The relation can be verified directly by writing everything in terms of an orthonormal basis \( \{e_1, \dots, e_n\} \).  Here is the rough idea: We may assume \( \a = e_I = e_1 \wedge \cdots \wedge e_{k-1} \) and   \( V = e_k \).  For if \( V = e_i, 1 \le i \le k-1 \), then both sides are zero.   We may also assume \( \b = e_k \wedge e_I \), else both sides are again zero.   But if \( \b = e_k \wedge e_I \), then both sides are one.   The general result follows using bilinearity of the inner product.  

(b)-(f) follow directly from the definitions on \( k \)-elements and continuity.

(g):   
\begin{align*} 
E_{fV}^\dagger (p; \a) = \sum_{i=1}^{k+1} (-1)^{i+1} \<f(p)V(p),e_i\> (p; \widehat{\a_i}) = f(p) \sum_{i=1}^{k+1} (-1)^{i+1} \<V(p),e_i\> (p; \hat{\a_i}) = m_f E_V^\dagger (p;\a).
\end{align*}    	
\end{proof} 

 
\subsection{Prederivative} 
\label{sub:prederivative}  
	Prederivative gives us a way to ``geometrically differentiate'' a differential chain in the infinitesimal directions determined by a vector field, and without use of test forms. While the operator extrusion \( E_V \) is an operator determined by a vector field \( V \), prederivative \( P_V \) is an operator determined by the tensor field \( V \otimes 1 \) (see \S\ref{sub:chainlets}).   We begin with a constant vector field \( v \).

 \begin{lem}\label{lem:pred} 
	Let \( (p;\a) \) be a simple \( k \)-element and \( v \in \R^n \). Then the sequence \( \{(T_{2^{-i}v} -I)(p; 2^i\a)\}_{i \ge 1} \) is Cauchy in \( \hB_k^2(U) \).
\end{lem}  

\begin{proof}  For simplicity of notation, set \( p = 0 \).
  Telescoping gives us 
	\begin{equation}\label{equation1} 
		( 2^{-i}v;\a) - (0;\a) = \sum_{m=1}^{2^j} (m2^{-i+j}v; \a) - ((m-1)2^{-i+j}v; \a). 
	\end{equation} 
	Expanding, we have \[ (T_{2^{-i}v} - Id)(0;2^i\a) - (T_{2^{-(i+j)}v} - Id)(0;2^{i+j}\a) = ((2^{-i}v; 2^i\a) - (0;2^i\a )) - ((2^{-(i+j)}v; 2^{i+j}\a) -(0; 2^{i+j}\a )). \] We apply \eqref{equation1} to the first pair, and consider the second pair as \( 2^j \) copies of \( ((2^{-(i+j)}v; 2^i\a) -(0; 2^i\a )) \). Then
	\begin{align*} 
		(T_{2^{-i}v} &- Id)(0;2^i\a) - (T_{2^{-(i+j)}v} - Id)(0;2^{i+j}\a) \\
		&= \sum_{m=1}^{2^j} (m2^{-(i+j)}v; 2^i\a) - ((m-1)2^{-(i+j)}v; 2^i\a) - ((2^{-(i+j)}v; 2^i\a) -(0; 2^i\a)).
	\end{align*}

	Rewrite the right hand side as a sum of \( 2 \)-difference chains and set \( v = 2^{-(i+j)}v \).  Then
	\begin{align}\label{eqn2} 
		(T_{2^{-i}v} - Id)(0;2^i\a) - (T_{2^{-(i+j)}v} - Id)(0;2^{i+j}\a) = \sum_{m=1}^{2^j} \D_{(v,(m-1)v)}(0; 2^i\a).
	\end{align}
                           
	We then use the triangle inequality to deduce 
	\begin{align*} 
		\| (T_{2^{-i}v} - Id)(0;2^i\a) - (T_{2^{-(i+j)}v} - Id)(0;2^{i+j}\a)\|_{B^2}  &\le \sum_{m=1}^{2^j} \|\D_{(v, (m-1)v)}(0; 2^i\a) \|_{B^2} \\  & \le  \sum_{m=1}^{2^j} (m-1)\|v\|^2(m-1)2^i\|\a\| \\&= \|v\|^2 2^{-2(i+j)}2^i\|\a\| 2^{2j}/2  
		   \le  2^{-i} \|v\|^2 \|\a\|.
	\end{align*}
\end{proof}        
 
\begin{defn}\label{prederivative}
	Define the differential chain \( P_v (p;\a) := \lim_{i \to \i} (T_{iv} -I)_*(p;\a/i) \) and extend to a linear map of Dirac chains \( P_v: \A_k(U) \to \hB_k^2(U) \) by linearity.  Let \( P: \R^n \times \A_k(U) \to \hB_k^2(U) \) be the resulting bilinear map: \( P(v, (p;\a)) := P_v (p;\a) \) called \emph{\textbf{prederivative}}. Using the natural inclusion  \(u_k^{2,r}: \hB_k^2(U) \to \hB_k^r(U) \), we have \(u_k^{2,r} \circ P_v: \A_k(U) \to \hB_k^r(U) \), but we usually write  \( P_v:\A_k(U) \to \hB_k^r(U) \) instead of \( u_k^{2,r} \circ P_v:\A_k(U) \to \hB_k^r(U) \).  
\end{defn} 

\begin{lem}\label{lem:preext}
For each \( v \in \R^n \) and \( r \ge 1 \), the  map \( P_v: \A_k(U) \to \hB_k^{r+1}(U) \) satisfies 
	\begin{enumerate}
	\item \( \|P_v(A)\|_{B^{r+1}} \le \|v\|\|A\|_{B^r} \) for all \( A \in \A_k(U) \);
	\item \(  P_v(H_k) \subset H_k \).  
	\end{enumerate}
\end{lem}
 
\begin{proof} 
(a):  Since \( P_v \) commutes with translation, and \( P_v \) is defined as a limit of \( 1 \)-difference chains, we have 

	\begin{align*} 
		\|P_v(\D_{\s^j}(p;\a))\|_{B^{j+1}} = \lim_{t\to 0} \| \D_{tv}(\D_{\s^j}(p;\a/t))\|_{B^{j+1}}  \le \lim_{t \to 0}| \D_{tv}(\D_{\s^j}(p;\a/t))|_{B^{j+1}}   \le \|v\| \|\s\|\|\a\|.
	\end{align*}

	Let \( A \in \A_k(U) \) be a Dirac chain, \( r \ge 0 \), and \( \e > 0 \). We can write \( A = \sum_{i=1}^m \D_{\s_i^{j(i)}}(p_i;\a_i) \) as in the proof of Lemma \ref{lem:ineq}, with \( \|A\|_{B^r} > \sum_{i=1}^m \|\s_i\|\|\a_i\| - \e \). Then 
	\begin{align*} 
		\|P_v A\|_{B^{r+1}} \le \sum_{i=1}^m \|P_v \D_{\s_i^{j(i)}}(p_i;\a_i)\|_{B^{r+1}}  \le \|v\| \sum_{i=1}^m \|\s_i\|\|\a_i\|  < \|v\|( \|A\|_{B^r} + \e). 
	\end{align*} 
	(b): This follows since \( P_v \circ u_k^{r,s} = u_k^{r,s}\circ P_v \).  
\end{proof} 

\begin{defn}
	Define \( P_v:\hB_k^r(U) \to \hB_k^{r+1}(U) \) as follows: If \( J \in \hB_k^r(U) \), choose \( A_i \to J \) in the \( B^r  \) norm and define \( P_v(J) := \lim_{i \to \i} P_v(A_i) \) (see Fig \ref{fig:Prederivative}).  
\end{defn}
 
\begin{figure}[htbp] 
	\centering 
	\includegraphics[height=2.5in]{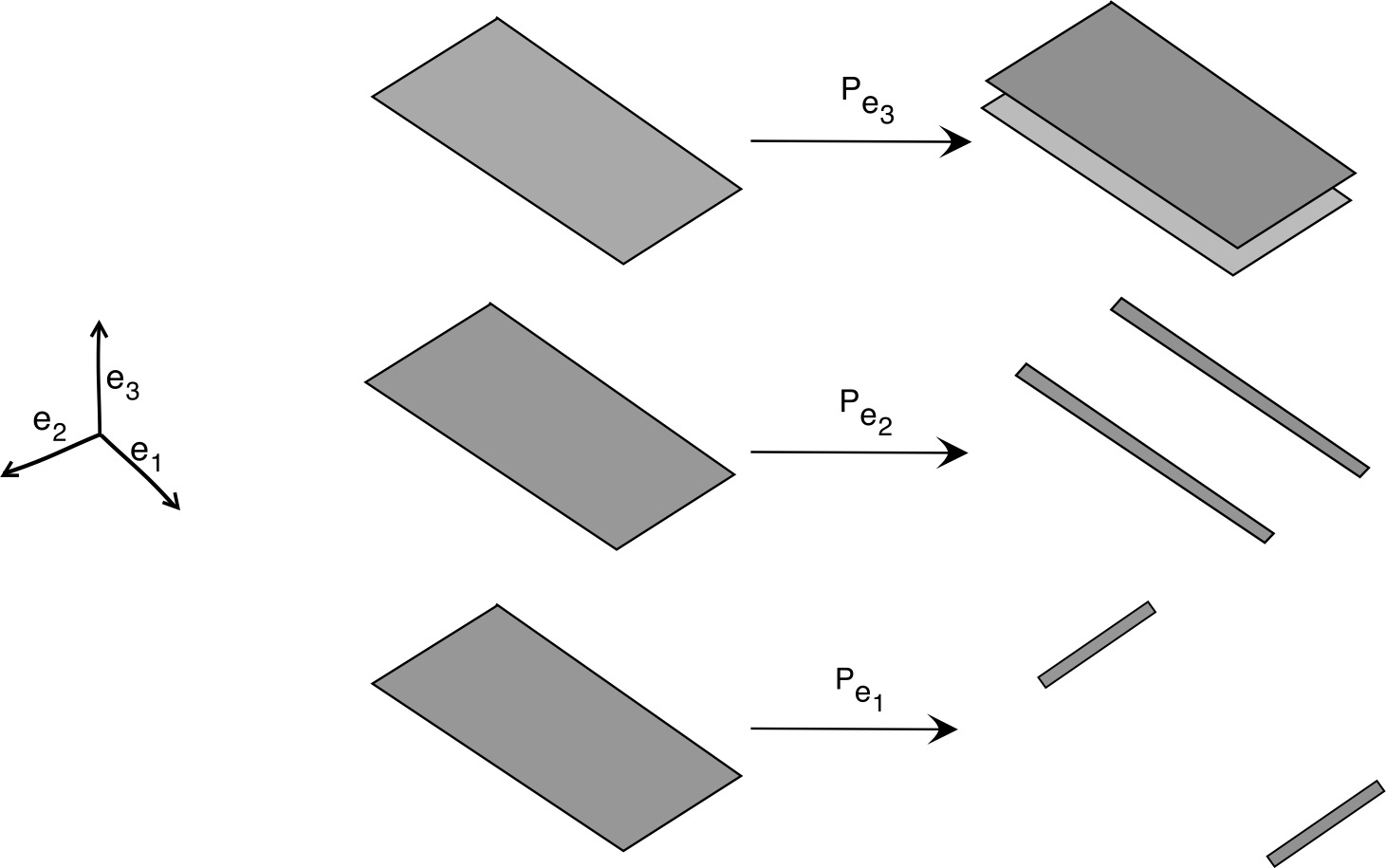}
	\caption{Prederivative of a \( 2 \)-cell in \( \R^3 \) in three different directions}
	\label{fig:Prederivative} 
\end{figure}            

 We deduce 
\begin{thm}\label{thm:IIC} 
	The map \( P:\R^n \times \hB_k^r(U) \to \hB_k^{r+1}(U) \) is bilinear and satisfies \(  \|P_v(J)\|_{B^{r+1}} \le \|v\|\|J\|_{B^r} \)	for all \( J \in \hB_k^r(U) \).   
\end{thm}                       
     
\begin{proof}  
The inequality holds by Lemma \ref{lem:preext}. We show that \( P \) determines a bilinear map on Dirac chains: It is linear in the second variable by definition. Additivity in the first variable reduces to showing \[ \lim_{t \to 0} (p + t(v_1 + v_2); \a/t) - (p + tv_1; \a/t) - (p + tv_2; \a/t) + (p; \a/t) = 0 \] in \( \hB_k(U) \). This follows from the definition of the \( B^2 \) norm of a \( 2 \)-difference chain: \(\|(p + t(v_1 + v_2); \a/t) - (p + tv_1; \a/t) - (p + tv_2; \a/t) + (p; \a/t)\|_{B^2} \le |(p + t(v_1 + v_2); \a/t) - (p + tv_1; \a/t) - (p + tv_2; \a/t) + (p; \a/t)|_2 \le t\|v_1\|\|v_2\|\|\a\| \). Homogeneity is immediate since \( \l(p;\a) = (p;\l \a) \). 
\end{proof}

In particular, Theorem \ref{thm:IIC} shows that \( P_v \in {\cal L}(\hB) \) is a continuous bigraded operator.

\begin{prop}[Commutation relations]\label{prop:ppp}
	\( [P_v,P_w] = [E_v, P_w] = [E_v^\dagger, P_w]= 0 \) for all \( v, w\in \R^n \).
\end{prop}

\begin{prop}\label{prop:mfpv} 
	Let \( f \in \B_0^r(U), 0 \le r \le \i \),  and \( v \in \R^n \). Then \( [m_f, P_v] = m_{L_v f} \). 
\end{prop}

\begin{proof} 
  By the mean value theorem there exists \( q_t = p+stv, 0 \le s \le 1 \) such that \( \frac{f(p+tv) - f(p)}{t} = L_vf(q_t) \). Then
	\begin{align*} 
		m_f P_v (p;\a) = m_f \lim_{t \to 0} (p+tv; \a/t) - (p;\a/t) & = \lim_{t \to 0} (p+tv; f(p+tv)\a/t) - (p; f(p) \a/t) \\
		& = \lim_{t\to 0} (p+tv; (f(p) + t L_v f(q_t)) \a/t) - (p; f(p) \a/t) \\
		& = \lim_{t \to 0} (p+tv; f(p) \a/t) - (p; f(p) \a/t) + \lim_{t \to 0} ( p; L_v f(q_t) \a ) \\
		& = P_v m_f (p;\a) + m_{L_v f} (p;\a).
	\end{align*} 
\end{proof} 
 There are several equivalent ways to define the boundary operator \( \p:\hB_k^r(U) \to \hB_{k-1}^{r+1}(U) \) for \( r \ge 1 \).  We use the operators prederivative and retraction. 	
 
\begin{defn}\label{dirbound}
	 Let \( v \in \R^n \). Define the \emph{\textbf{directional boundary}} on differential chains by \( \p_v: \hB_k^r(U) \to \hB_{k-1}^{r+1}(U) \) by \( \p_v := P_v E_v^\dagger \). See Figure \ref{fig:Partialboundary}.
\end{defn}

\begin{figure}[ht]
	\centering 
	\includegraphics[height=2.2in]{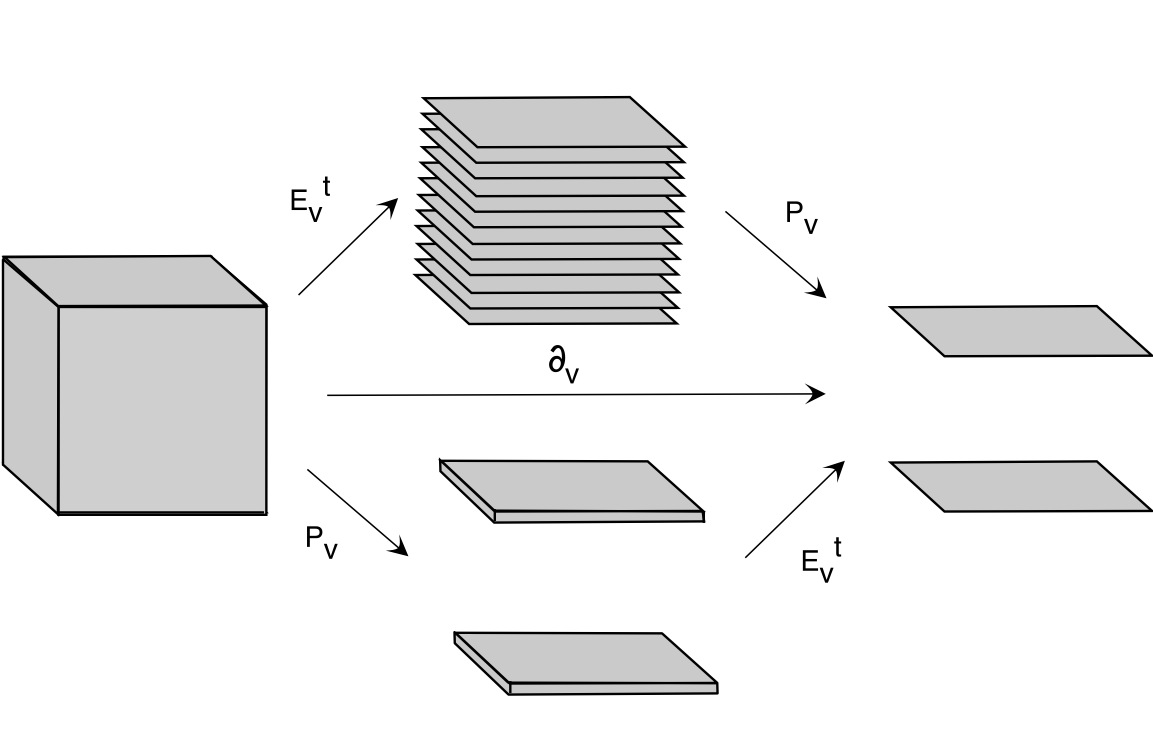}
	\caption{Directional boundary \( \p_v \)} 
	\label{fig:Partialboundary} 
\end{figure}

\subsection{Boundary operator and Stokes' theorem} 
\label{sub:boundary_operator}

For an orthonormal basis \( \{e_i\} \) of \( \R^n \), set
	\begin{equation}
		\label{eq:bd}
		\p := \sum_{i=1}^n \p_{e_i}.
	\end{equation} 
	Since \( P_{e_i} \) and \( E_{e_i}^\dagger \) are continuous, \emph{\textbf{boundary}} \( \p \) is a  continuous operator \( \p:\hB_k^r(U) \to \hB_{k-1}^{r+1}(U) \) that restricts to the classical boundary operator on  polyhedral \( k \)-chains \( \P_k(U) \), and is independent of choice of \( \{e_i\} \). Furthermore, \( \p \circ F_* = F_* \circ \p \).                           

\begin{thm}\label{thm:bod}
	For each \( r \ge 1 \) and \( 0 \le k \le n \),  \( \p: \A_k(U) \to \hB_{k-1}^2(U) \) is a continuous bigraded operator  satisfying
	\begin{enumerate}
		\item \( \p(p_i;v_i) = P_{v_i}(p_i;1) \);
		\item \( \p((p;\a) \cdot (p;\b)) = (\p (p;\a)) \cdot (p;\b) + (-1)^k (p;\a) \cdot \p (p;\b) \) for \( (p;\a), (p;\b) \in \A(p)  \) with \( \dim(p;\a) = k \);
		\item \( \p \circ \p= 0 \);
		\item \(  \p(\A_0(U)) = 0 \).
	\end{enumerate} 
\end{thm}

\begin{proof}   
	These follow the definition of \( \p \) and  properties of \( E_{e_i}^\dagger \) and \( P_{e_i} \).  Continuity follows from continuity of \( P_v \) and \( E_v^\dagger \) (Theorems  \ref{thm:retX} and \ref{thm:IIC}).  
\end{proof}  

\begin{thm}\label{thm:bdcont} 
	The linear map \( \p: \hB_k^r(U) \to \hB_{k-1}^{r+1}(U) \) is continuous with \( \|\p J\|_{B^{r+1}} \le kn \|J\|_{B^r} \) for all \( J \in \hB_k^r(U) \). It therefore extends to a continuous operator \( \p : \hB(U) \to \hB(U) \).
\end{thm}  

\begin{proof} 
This follows from Theorems \ref{thm:retX} and \ref{thm:IIC}.
\end{proof}

\begin{thm}\label{thm:Gpush}
	The following relations hold: 
	\begin{enumerate} 
		\item \( \{\p, E_v \} = P_v \) (Cartan's magic formula for differential chains);
		\item \( [E_v^\dagger, \p] = 0 \); 
		\item \( [P_v, \p] = 0 \). 
		\item \( [m_f, \p] = \sum_i df(e_i) E_{e_i}^\dagger \).
	\end{enumerate} 
\end{thm} 

\begin{proof}          Parts (a)-(c) follow from the definitions and Proposition \ref{pro:asgg}.  
  (d):  This uses Propositions \ref{pro:asgg} \ref{prop:mfpv}, as well as Theorem \ref{thm:bod}.
\begin{align*}
 m_f \p - \p m_f &= m_f \sum_i P_{e_i} E_{e_i}^\dagger - \sum_i P_{e_i} E_{e_i}^\dagger m_f \\&
=  m_f \sum_i P_{e_i} E_{e_i}^\dagger - \sum_i P_{e_i}m_f E_{e_i}^\dagger \\&
= \left(m_f \sum_i P_{e_i} - \sum_i P_{e_i}m_f  \right)E_{e_i}^\dagger\\&
= \left(\sum_i m_fP_{e_i} - P_{e_i}m_f\right) E_{e_i}^\dagger\\&
= \left(\sum_i m_{L_{e_i}f} \right)E_{e_i}^\dagger
\end{align*} 
By Cartan's magic formula \( L_{e_i} = i_{e_i}d + d i_{e_i} \).
Hence \(    L_{e_i}f = i_{e_i}df + d i_{e_i}f = i_{e_i}df = df(e_i) \).  The result follows.
\end{proof}      

\begin{thm}[Stokes' Theorem]\label{thm:stokes} 
Let \( 0 \le k \le n-1 \).  Then	\[ \cint_{\p J} \o = \cint_J d \o \] for all matching pairs  \( J \in \hB_{k+1}^{r-1}(U) \), \( \o \in \B_k^r(U)\), and \( 1 \le r \le \i  \).
\end{thm}

\begin{proof} 
	By Lemma \ref{lem:evdag} and Theorem \ref{thm:preV} \( \o \p (J) = \o \sum_{i=1}^n P_{e_i}E_{e_i}^\dagger (J) = \sum_{i=1}^n L_{e_i} \o E_{e_i}^\dagger(J) = \sum_{i=1}^n d e_i \wedge L_{e_i}\o(J) \).  Since \(  d = \sum_{i=1}^n d e_i \wedge L_{e_i}  \) is equivalent to the classical definition of \( d   \), the result follows.    
\end{proof}   

 	\begin{example} \textbf{Algebraic boundary of the Cantor set}   We revisit the middle third Cantor set \( \G \subset I \), the unit interval from \S\ref{sub:poly}. Recall the chain representatives of approximating open sets obtained by removing middle thirds forms a Cauchy sequence in \( \hB_1^1(I) \). These are algebraic \( 1 \)-chains. The limit \( \widetilde{\G} \) is a differential \( 1 \)-chain that represents \( \G \). Its boundary \( \p \widetilde{\G} \) is well-defined and is supported in the classical middle third Cantor set. We may therefore integrate differential forms over \( \G \) and state the fundamental theorem of calculus where \( \G \) is a domain and \( f\in \B_1^1(I) \): 
\[ \cint_{\p \G} f = \cint_{\G} df. \]    We expand this for actual calculations:  We have   \( \G = \lim_{n\to \i} (3/2)^n\widetilde{E_n}\)  where \( E_n = \cup [qp_{n,i}, q_{n,i}] \) is the set obtained after removing middle thirds at the \( n \)-th stage. Then \( \p \G = \sum_{n=1}^{\i}\sum_i (q_{n,i};(3/2)^n) - (p_{n,i};(3/2)^n). \)   Furthermore, \( \cint_{\p \G} \o = \cint_{\G} d \o \) for all \( \o \in C^{1 + Lip} \).  For example, \[ \cint_{\p \G} x = \cint_{\G} dx = \lim_{n \to \i} \cint_{(3/2)^n\widetilde{E_n}} dx = 1. \]                                                                    

	\end{example}   

 \begin{prop}\label{prop:whitneybd} 
	Suppose \( F \in {\cal M}^r(U, U') \).  Then \( F_* \p = \p F_*: \hB_k^r(U) \to \hB_{k-1}^{r+1}(U') \).     
\end{prop}

\begin{proof}  
	This follows directly from the dual result on differential forms \( F^* d = F^* d \).   A direct calculation on differential chains is possible, but takes longer.
\end{proof}

\begin{example}\label{ex:Qpic} 
	Let \( Q' \) be the open unit disk \( Q \) in \( \R^2 \), less the interval \( [0,1] \times \{0\} \). It follows that \( \widetilde{Q'} = \widetilde{Q} \in \hB_2^1(\R^2) \) since evaluations by \( n \)-forms in \( \B_n(\R^2) \) are the same over \( Q' \) and \( Q \). But \( \widetilde{Q'}_{Q'} \in \hB_2^1(Q') \) and is not the same chain as \( \widetilde{Q'} \). Not only are these chains in different topological vector spaces, but their boundaries are qualitatively different: Let \( L_+ \in  \hB_1^2(Q') \) be the \( 1 \)-chain representing the interval \( [0,1] \times \{0\} \) found by approximating \( [0,1] \times \{0\} \) with intervals \( [0,1] \times \{1/n\} \cap Q \).  Similarly, let \( L_- \in \hB_1^2(Q') \) be the \( 1 \)-chain representing the interval \( [0,1] \times \{0\} \) found by approximating \( [0,1] \times \{0\} \) with intervals \( [0,1] \times \{-1/n\} \). Recall \( \o_0 \) as defined in Example \ref{o0}. Then \( \o_0^2 dx \in \B_1^2(Q') \), \( \cint_{L_+} \o_0^2 dx = \lim_{n \to \i}\cint_{\widetilde{[0,1] \times \{1/n\}}} \o_0^2 dx = 1 \) and \( \cint_{L_+} \o_0^2 dx = \lim_{n \to \i} \cint_{\widetilde{[0,1] \times \{1/n\}}} \o_0^2 dx = 0 \).  It follows that \( L_+ - L_- \ne 0 \) and thus \( \p(\widetilde{Q'}_{Q'}) = L_+ - L_- + \widetilde{S^1}_{Q'} \ne \widetilde{S^1}_{Q'} \) while \( \p \widetilde{Q'}=\widetilde{S^1} \) (see Figure \ref{fig:Mapping}).    
\end{example}

\begin{prop}\label{prop:repalgecells} 
	If \( F \in {\cal M}^r(U,U') \) is a diffeomorphism onto its image, and \( \s \) is an affine \( k \)-cell in \( U \), then \( F_*\widetilde{\s} = \widetilde{F \s} \), \( \supp(F_*\widetilde{\s}) =  F \s  \), and \( \supp(\p F_*\widetilde{\s}) = F\fr(\s) \).    
\end{prop}

\begin{proof} 
Since \( F \) is a diffeomorphism,  \(  \int_\s F^* \o = \int_{F\s} \o \) for all \( \o \in \B_k^r(U') \).  Therefore, by Theorem \ref{thm:opensets} and Corollary \ref{cor:pull}  
 \[ \cint_{F_* \widetilde{\s}} \o = \cint_{\widetilde{\s}} F^* \o  = \int_\s F^* \o = \int_{F\s} \o =  \cint_{\widetilde{F \s}} \o. \]  
 It follows that \(  F_*\widetilde{\s} = \widetilde{F \s} \) and thus \( \supp(F_*\widetilde{\s}) = \supp(\widetilde{F \s}) =  F \s  \).   Use Proposition \ref{prop:whitneybd} and the previous  equation for \( (k-1) \)-cells to deduce \( \supp(\p F_*\widetilde{\s}) = \supp(F_*\p \widetilde{\s}) = \supp(F_* \widetilde{\fr\,\s}) = F\fr(\s)  \).
\end{proof}

\subsection{Prederivative with respect to a vector field} 
\label{sub:prederivatve_wrt}
  \begin{defn}\label{prederivativeX}
	Define the bilinear map \( P: {\cal V}^{r+1}(U) \times \hB_k^r(U) \to \hB_k^{r+1}(U) \) by \[ P(V,J) := \p E(V,J) + E(V,\p(J)). \]
\end{defn}  
   Let \( P_V(J) := P(V,J) \) (see Fig. \ref{fig:Dipolesphere}).  The dual operator \( L_V \) is the classically defined Lie derivative since the relation \( L_V = i_V d + d i_V \) uniquely determines it.

\begin{figure}[ht]
	\centering
		\includegraphics[height=1.25in]{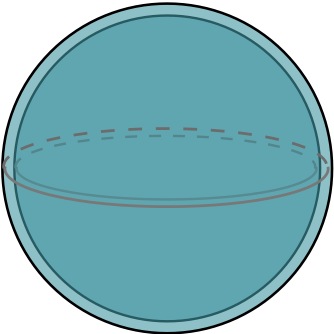}
	\caption{Prederivative of the sphere with respect to the radial vector field}
	\label{fig:Dipolesphere}
\end{figure}

\begin{thm}\label{thm:predfirst} 
Let \(0 \le k \le n\) and \( r \ge 1 \).  The bilinear map determined by 
	\begin{align*} 
		P = P_k^r: {\cal V}^{r+1}(U) \times \hB_k^r(U) &\to \hB_k^{r+1}(U) \\ (V, J) &\mapsto P_V(J) 
	\end{align*} 
is well-defined and continuous with \( \|P(V, J)\|_{B^{r+1}} \le 2 kn^3 2^r \|V^\flat\|_{B^{r+1}} \|J\|_{B^{r}} \). Furthermore, for each \( 0 \le k \le n \), there exists a separately continuous bilinear map \( P = P_k:{\cal V}^\i(U) \times \hB_k(U) \to \hB_k(U) \) which restricts to \( P_k^r \) on each \( u_k^r(\hB_k^r(U)) \).    
\end{thm}  
  
\begin{proof} 
By Definition \ref{prederivativeX}, and Theorems \ref{thm:bod}, and  \ref{thm:IIAV}
\begin{align*} 
      \|P_V J\|_{B^{r+1}} \le  \|\p E_V J\|_{B^{r+1}} +  \| E_V \p J\|_{B^{r+1}} &\le kn \|E_V J\|_{B^{r}}+ n^2 2^r\|V^\flat\|_{B^{r+1}} \|\p J\|_{B^{r+1}} \\&\le  k n^3 2^r\|V^\flat\|_{B^{r}} \|J\|_{B^{r}}+ k(n-1) n^2 2^r \|V^\flat\|_{B^{r+1}} \|J\|_{B^{r}}  \\&\le 2 kn^3 2^r \|V^\flat\|_{B^{r+1}}\|J\|_{B^{r}}.  
\end{align*}   
For the last assertion, we first establish continuity of \( P \) in the first variable:  Suppose \( J \in \hB_k(U) \). Then there exists \( r \ge 0 \) and \( J^r \in \hB_k^r(U) \) such that \( u_k^r(J^r) = J \). Suppose \( V_i \to 0 \) in \( {\cal V}^\i(U) \).  Since the inclusion \( {\cal V}^\i(U) \to {\cal V}^{r+1}(U) \) is the identity map and continuous,  then \( V_i \to 0 \) in \( {\cal V}^{r+1}(U) \).  Now use the fact that \( P_k^r: {\cal V}^{r+1}(U) \times \hB_k^r(U) \to \hB_k^{r+1}(U) \) is continuous in the first variable. Continuity in the second variable of \( P \) follows from Theorem \ref{thm:continuousoperators} since  \( P_V(H_k(U) \subset H_k(U)) \).         
\end{proof}   

\begin{examples}[Representatives of soap films]
	 
	The chain	 \( P_V F_* \widetilde{\s} \) is a  \emph{\textbf{dipole \( k \)-cell}}, and \( \sum  P_{V_i} F_{*i} \widetilde{s_i} \) is a \emph{ \textbf{dipole \( k \)-chain}}.    Dipole \( k \)-chains are useful for representing soap films, Moebius strips, and soap bubbles. (See \cite{soap}, \cite{plateau10} for more details).  
 
\end{examples}

\begin{thm}[Change of order]\label{thm:preV}  
Let \(0 \le k \le n\). Then
\begin{equation}\label{primitiveprederivatives} 
\cint_{P_V J} \o = \cint_J L_V \o
\end{equation} 
for all matching triples  \( V \in {\cal V}^{r+1}(U) \),   \( J \in \hB_k^r(U) \), \( \o \in \B_k^{r+1}(U) \), and \( 1 \le r \le \i  \).
\end{thm} 

\begin{proof} 
This follows from Theorems \ref{thm:stokes} and \ref{thm:retXint}.
\end{proof} 
  
\begin{thm}\label{thm:prederU} 
If \( V \in {\cal V}^{r+1}(U) \) is a vector field and \( J \in \hB_k^r(U) \) has compact support \( \supp(J) \subset U \), then \[ P_V J = \lim_{t \to 0} (\phi_{t*} - Id)/t(J) \] where \( \phi_t \) is the time-\( t \) map of the flow of \( V \). 
\end{thm}   

\begin{proof}   
	 Since  \( \supp(J) \subset U \) there exists an open set \( \supp(J) \subset U' \subset U \).  Since \( \supp(J) \) is compact, there exists \( t_0 > 0 \) such that \( \phi_t(p) \) is defined for all \( 0 \le t \le t_0 \) and \( p\in \supp(J) \).    
If \( \o \in \B_k^{r+1}(U) \), then  \( L_V \o = \lim_{t \to 0}((\phi_t^* - Id)/t (\o)) \in \B_k^r(U) \).   Let \( F_t = \phi_t - Id \).    Then	\( \|L_V \o - F_t^*(\o)\|_{B^{r}} \to 0 \) as \( t \to 0 \).
	
There exists \( \eta \in \B_k^{r+1}(U) \) with \( \|\eta\|_{B^{r+1}} = 1 \) and \( \|(F_{t*} - P_V)J\|_{B^{r+1}} = \cint_{(F_{t*} - P_V)J} \eta   = \left|\cint_{J} (F_t^* - L_V) \eta \right| \le  \|J\|_{B^{r}} \|(F_t^*- L_V)\eta\|_{B^{r}}  \to 0 \) as \( t \to 0 \). 
Hence \(  P_V J = \lim_{t \to 0} F_{t*}J =  \lim_{t \to 0} (\phi_{t*} - Id_*)(J/t)  \).  \end{proof}

 The next result follows directly from the definitions.

\begin{thm}[Commutation Relations] \label{thm:cr} 
	If \( V,W \in {\cal V}^r(U) \), then
	\begin{enumerate}
		\item \( [P_V, P_W] = P_{[V,W]} \);
 		\item \( [E_V, P_W] = [P_V, E_W] =	E_{[V,W]} \).
	\end{enumerate} 	     
\end{thm}
                    
\subsection{Naturality of the operators} \label{ssub:naturality_of_the_operators}
  
\begin{thm}[Naturality of the operators]\label{thm:mF}  
	If \( V \in {\cal V}^r(U) \)   and \( F \in {\cal M}^r(U, U') \), then 
\begin{enumerate} 
	\item \( F_* m_{f \circ F} = m_f F_* \) for all \( f \in \B_0^r(U') \) for \( 1 \le r \le \i \).   \\If \( F \) is a diffeomorphism  onto its image, then 
  \item \( F_* E_V = E_{FV} F_* \); 
	\item \( F_* E_V^\dagger = E_{FV}^\dagger F_* \) if \( F \) preserves the metric; 
	\item \( F_* P_V = P_{FV} F_*\). 
\end{enumerate} 
\end{thm}

\begin{proof} 
	(a): \(F_* m_{f \circ F} (p;\a) = (F(p);f(F(p)) F_*\a) = m_f F_* (p;\a) \).  
(b): is omitted since it is much like (c).  
(c): Since \( \<F_*V, F_* v_i\> = \<V,v_i\> \) we have 
\begin{align*}
 	F_*(E_V^\dagger((p;\a))) = \sum_{i=1}^k(-1)^{i+1} \<V,v_i\> (F(p); F_* v_1 \wedge \cdots \hat{v_i} \cdots \wedge v_k)&= \sum_{i=1}^k (-1)^{i+1}
	\<F_*V, F_* v_i\> (F(p); F_* v_1 \wedge \cdots \hat{v_i} \cdots \wedge v_k) \\&=E_{FV}^\dagger (F(p); F_* \a) \\&=    E_{FV}^\dagger F_* (p;\a)
\end{align*} 
 (d): Observe that if \( V_t \) is the flow of \( V \), then \( F V_t F^{-1} \) is
the flow of \( FV \). It follows that
\begin{align*} 
F_* P_v (p;\a) = F_* \lim_{t \to 0}
(V_{t}(p); V_{t*}\a/t) - (p;\a/t) &= \lim_{t \to 0} (F(V_{t}(p)); F_*V_{t*}\a/t) -
(F(p);F_*\a/t) \\&= \lim_{t \to 0} ((FV_{t}F^{-1})F(p); F_*\a/t) - (F(p);F_*\a/t) \\&=
P_{FV}(F(p); F_*\a) \\&= P_{FV}F_* (p;\a). 
\end{align*}
\end{proof}

%

\section{Perpendicular complement and divergence theorems} 
\label{sub:perpendicular_complement_and_classical_integral_theorems_of_calculus}

\subsection{Clifford algebra and perpendicular complement}\label{sub:perp} 

Consider the subalgebra \( {\cal C\ell}  \subset {\cal L}(\hB) \) of linear operators on \( \hB(U) \) generated by \( \{C_v\,:\, v \in \R^n \} \) where \( C_v = E_v + E_v^\dagger \).  It turns out that \( {\cal C\ell} \) is isomorphic to the Clifford algebra, where the isomorphism depends on the inner product. (This is not used in our paper and we leave it as an exercise for interested readers.)   The product\footnote{Some authors call \(   u \wedge v + \<u,v\> \) the   ``geometric product''.  This is found within our viewpoint of differential chains as \( (E_u + E_u^\dagger)\circ (E_v + E_v^\dagger)(0;1) =  u \wedge v + \<u,v\>  \).} for the Clifford algebra corresponds to composition of operators in \( {\cal C\ell} \): \( C_u \cdot C_v =  (E_u + E_u^\dagger)\circ (E_v + E_v^\dagger) \).  Let us see what \( C_v \) does to a simple \( k \)-element \( (p;\a) \):  For simplicity, assume \( \a = e_I \) and \( v = e_i \), taken from an orthonormal basis \( \{e_i\} \) of \( \R^n \). If \( e_i \) is in the \( k \)-direction of  \( e_I \), then \( C_{e_i} \) ``factors it out of'' \( e_I \), reducing its dimension to \( k-1 \).  If \( e_i \) is not in the \( k \)-direction of \( e_I \), then \( C_{e_i} \) ``factors it into''  \( e_I \), increasing its dimension to \( k+1 \).      

\begin{defn}\label{perp}  
	Let \( \perp: \A_k(U) \to \A_{n-k}(U) \) be the operator on Dirac chains given by  \( \perp  := C_{e_n} \circ \cdots \circ C_{e_1} = \Pi_{i=1}^n (E_{e_i} + E_{e_i}^\dagger) \).    Then \( \perp \) extends to a continuous linear map on \(\hB(U) \) since \( E_v \) and \( E_v^\dagger \) are continuous. We call \( \perp \) \emph{\textbf{perpendicular complement}}.  Perpendicular complement depends on the inner product, but not on the choice of an orthonormal basis.
\end{defn}

The next result is immediate:
\begin{prop}\label{prop:perpagain} 
	If \( \a \) is a \( k \)-vector, then \( \perp(p;\a) = (p; \a^\perp) \), where \( \a^\perp \) is the \( (n-k) \)-vector satisfying \( \a \wedge \a^\perp = (-1)^k \|\a\|^2 e_1 \wedge \cdots \wedge e_n \).  Furthermore, the \( k \)-direction of \( \a \) is orthogonal to the \( (n-k) \)-direction of \( \a^\perp \). Thus, \( \perp \circ \perp = (-1)^{k(n-k)} I\).
\end{prop}

\begin{figure}[ht] 
	\centering 
	\includegraphics[height=1in]{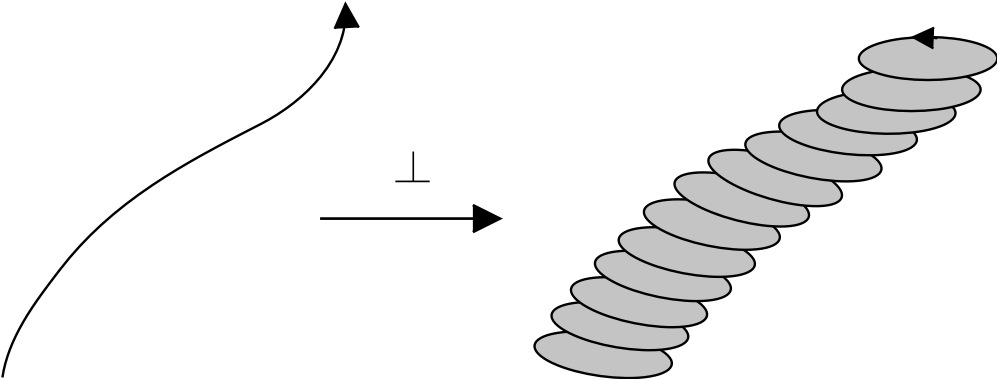}
	\caption{Perpendicular complement \( \perp \)} 
	\label{fig:Perp} 
\end{figure}

It follows that \( \star \o := \o \perp \) where \( \star \) is the classical Hodge \cite{thehodge} star operator on forms. 

\begin{thm}[Star theorem]\label{thm:perpcont}
Let \( 0 \le k \le n \).	The linear map \(	\perp: \A_k(U) \to \A_{n-k}(U) \) determined by \( (p;\a) \mapsto (p; \perp \a), \) for simple $k$-elements \( (p;\a) \), satisfies \[ \|\perp A\|_{B^r} = \|A\|_{B^r} \] for all \( A \in \A_k(U) \). It therefore extends to a continuous linear map \( \perp: \hB_k^r(U) \to \hB_{n-k}^r(U) \) for each \( 0 \le k \le n, 0 \le r <\i \), and to a continuous graded operator \(\perp \in {\cal L}(\hB) \) satisfying \[ \cint_{\perp J} \o = \cint_J \star \o \] for all matching pairs \( J \in \hB_k^r(U) \), \( \o \in \B_{n-k}^r(U) \), and \( 0 \le r \le \i  \).
\end{thm}

\begin{proof} 
	We know \( \|\perp J\|_{B^r} = \|J\|_{B^r} \) since \( \|J\|_{B^r} = \sup\frac{|\o(J)|}{\|\o\|_{B^r}} \) and \( \|\o\|_{B^r} = \|\star \o\|_{B^r} \) for \( 0 \le r <  \i  \). A direct proof using difference chains is straightforward and can be found in \cite{hodge}.

	The integral relation holds on Dirac chains since \( \star \o = \o \perp \) on Dirac chains and by continuity. Continuity of \(\perp: \hB(U) \to \hB(U) \) follows (see Figure \ref{fig:Perp}).
 
\end{proof}   
\begin{cor}[Kelvin-Stokes' Curl Theorem]\label{thm:curl}
	Let \( 0 \le k \le n-1 \).  Then	\[ \cint_{\p \perp J} \o = \cint_J \star d \o \] for all matching pairs  \( J \in \hB_{n-k-1}^{r-1}(U) \), \( \o \in \B_k^r(U)\), and \( 1 \le r \le \i  \).
\end{cor}  

\begin{proof} 
This follows directly from Theorems \ref{thm:stokes} and \ref{thm:perpcont}.
\end{proof}

\begin{cor}[Gauss-Green Divergence Theorem]\label{thm:div}
	Let \(1 \le k \le n \).  Then	\[ \cint_{\perp \p J} \o = \cint_J d  \star \o \] for all matching pairs  \( J \in \hB_{n-k+1}^{r-1}(U) \), \( \o \in \B_k^r(U)\), and \( 1 \le r \le \i  \).       
\end{cor}  

\begin{proof} 
This follows directly from Theorems \ref{thm:stokes} and \ref{thm:perpcont}.
\end{proof}  
\begin{examples} \mbox{} 
\begin{enumerate}
	\item  	If \( \widetilde{S} \) represents an oriented surface \( S \) with boundary in \( \R^3 \) and \( \o \) is a smooth \( 2 \)-form defined on \( S \),   Corollary \ref{thm:div} tells us \( \cint_{\widetilde{S}} d \star \o = \cint_{\perp \p \widetilde{S}} \o \) (see Fig. \ref{fig:Divergence3D}). This result equates the net divergence of \( \o \) within \( S \) with the net flux of \( \o \) (thought of as a \( 2 \)-vector field with respect to the inner product) across the boundary of \( S \).
	\item  Recall be the open set \( W \)  in   Figure \ref{fig:openset}.   Let \( \o \in \B_k^r(W) \) be a differential form which is identically equal to a smooth differential form \( \eta_N \) in the top lobe, and another differential form \( \eta_S \)  in the bottom lobe.  For example,  \( \o_N = -y^2dx/2 + x^2 dy/2 \) and \( \o_S = y^2dx/2 - x^2 dy/2 \). Then \( d\o_N =  (x+y) dx dy   \) and \( d\o_S = -(x+y) dxdy \).  So the net flux across the respective boundaries have opposite signs.  In one lobe, the flux is outward flowing, and in the other, it is inward flowing.  When we draw them together the flow moves from one lobe into the other and the velocity has an instantaneous change in direction and velocity as the flow crosses the frontier.     
\end{enumerate}

\end{examples} 
\begin{figure}[ht]
	\centering
    \includegraphics[height=2.0in]{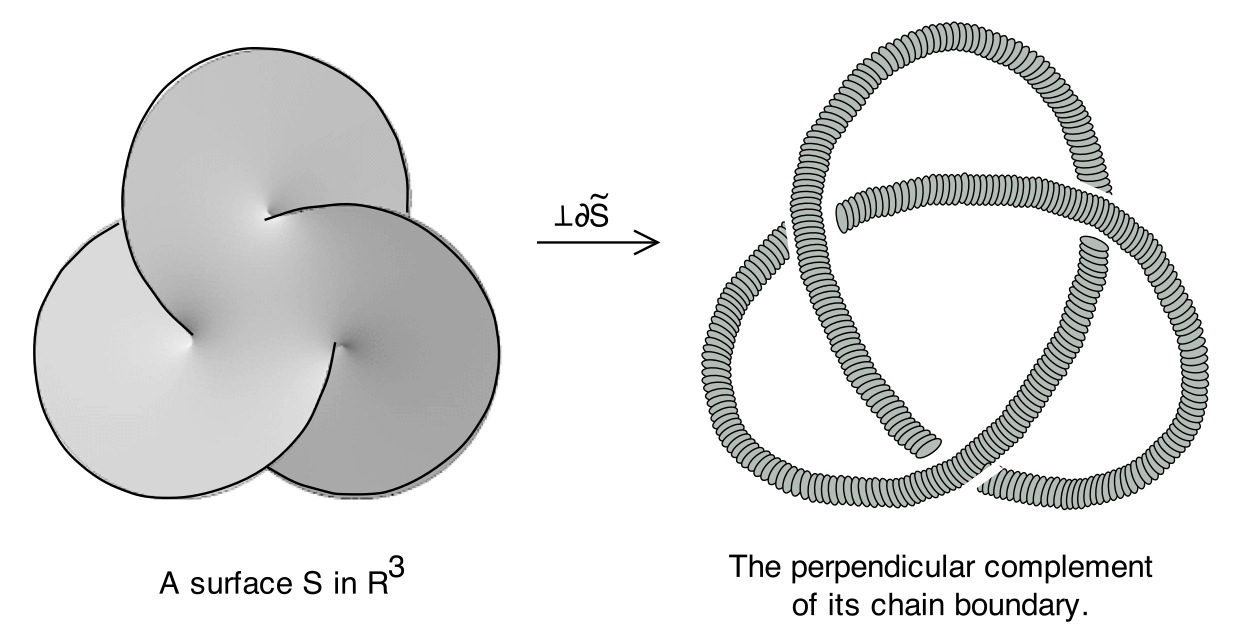}
    \caption{A domain of the divergence theorem for surfaces in \( \R^3 \)}
    \label{fig:Divergence3D}
\end{figure}

\subsection{Geometric coboundary, Laplace, and Dirac operators}\label{sub:geometric_coboundary_laplace_and_dirac_operators}    

In this section we introduce an extension of the divergence theorem \ref{thm:div} to take into account higher order derivatives.  As far as we know, it has no classical precedent. 
                
\begin{defn}\label{coboundary}          
	Define the \emph{\textbf{coboundary}} operator \( \lozenge := \perp \p \perp \). Since \( \p \) decreases dimension, \( \lozenge \) increases dimension. Its dual operator is the codifferential \( \d \) where \( \d = \star d \star \).
\end{defn}

\begin{defn}\label{laplace}
Define the \emph{\textbf{geometric Laplace operator}} \( \square := \lozenge \p + \p \lozenge \).	
\end{defn}
 Since \( \square: \hB_k^r(U) \to \hB_k^{r+1}(U) \) is a continuous operator on \( \hB_k(U) \), we may iterate \( \square \) any number of times \( \square^j = \underbrace{\square \circ \cdots \circ \square}_{j \mbox{ iterates}} \) and obtain a continuous operator \(  \square^j \) on \( \hB_k(U) \).    The dual operator is the classical Laplace operator \( \D = d \d + \d d \) on differential forms \( \B_k(U) \), and may also be iterated.  The \emph{geometric Dirac} operator \( \p + \lozenge \) dualizes to the Dirac operator \( d + \d \) on differential forms (\cite{hodge}). 

\begin{figure}[ht]
	\centering 
	\includegraphics[height=1.5in]{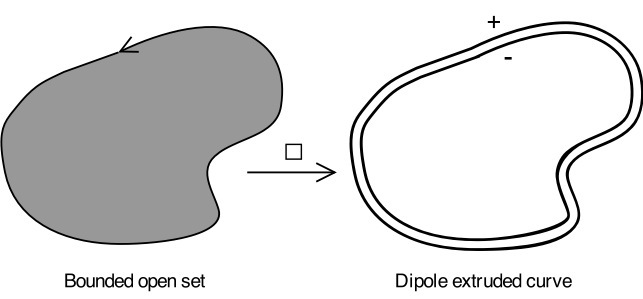} 
	\caption{Geometric Laplace operator \( \square \) of an open set in \( \R^2 \)} 
	\label{fig:Laplace} 
\end{figure} 

\begin{cor}[Higher order divergence theorem]\label{cor:lapl} 
Let \( j \ge 0 \) and \( 0 \le k \le n \).	Then\[ \cint_{\square^j J} \o = \cint_J \D^j \o \]  for all matching pairs  \( J \in \hB_k^r(U) \), \( \o \in \B_k^{r+2j}(U) \), and \( 0 \le r \le \i  \).   
\end{cor} 

Figure \ref{fig:Laplace} shows this Corollary \ref{cor:lapl} has interesting geometric meaning for \( j = 1 \). Whereas the divergence theorem equates net flux of a \( k \)-vector field across an oriented boundary with net interior divergence, Corollary \ref{cor:lapl} equates net flux of the ``orthogonal differential'' of a \( k \)-vector field across an oriented boundary with net interior ``second order divergence''. In this sense, Corollary \ref{cor:lapl} is a ``higher order divergence theorem''.    
 
The following diagram depicts the \emph{\textbf{exterior differential complex}}.  The ``boundary'' operators are boundary \( \p \) (boundary, predual to exterior derivative),  \( \lozenge \) (geometric coboundary, predual to \( *d* \)), extrusion \( E_v \), and retraction \( E_v^\dagger \).  The commutation relations are given in Propositions \ref{prop:EE}, \ref{pro:asgg},  \ref{prop:ppp}, \ref{prop:perpagain}, and Theorem \ref{thm:Gpush}.

\begin{equation}\label{diagramK1} 
	\xymatrixcolsep{4pc}
	\xymatrixrowsep{4pc}
	\xymatrix{
		\vdots \ar@<0.05in>@{->}[d]^{E_v^\dagger} \ar@{<-}[d]_(.52){E_v} \ar@{->}[rd]^<<<<<<<{\p}   & \vdots \ar@<0.05in>@{->}[d]^{E_v^\dagger} \ar@{<-}[d]_(.52){E_v}\ar@{->}[rd]^<<<<<<<{\p}   & \vdots \ar@<0.05in>@{->}[d]^{E_v^\dagger} \ar@{<-}[d]_(.52){E_v}\ar@{->}[rd]^<<<<<<<{\p}   &  \vdots\cdots \\ 
		\hB_2^0 \ar@<0.05in>@{->}[d]^{E_v^\dagger} \ar@{<-}[d]_(.52){E_v}\ar@{->}[rd]^<<<<<{\p}   \ar@<0.05in>@{->}[r]^{P_v} \ar@{->}[ru]^(.68){\lozenge}|\hole  & \hB_2^1 \ar@<0.05in>@{->}[d]^{E_v^\dagger} \ar@{<-}[d]_(.52){E_v}\ar@{->}[rd]^<<<<<{\p}   \ar@<0.05in>@{->}[r]^{P_v} \ar@{->}[ru]^(.68){\lozenge}|\hole & \hB_2^2  \ar@<0.05in>@{->}[d]^{E_v^\dagger} \ar@{<-}[d]_(.52){E_v} \ar@{->}[ru]^(.68){\lozenge}|\hole 	\ar@{->}[rd]^<<<<<{\p}   \ar@<0.05in>@{->}[r]^{P_v}   & \cdots \\
		\hB_1^0 \ar@<0.05in>@{->}[d]^{E_v^\dagger} \ar@{<-}[d]_(.52){E_v}\ar@{->}[rd]^<<<<<{\p}   \ar@<0.05in>@{->}[r]^{P_v}\ar@{->}[ru]^(.70){\lozenge}|\hole & \hB_1^1 \ar@<0.05in>@{->}[d]^{E_v^\dagger} \ar@{<-}[d]_(.52){E_v}\ar@{->}[rd]^<<<<<{\p}   \ar@<0.05in>@{->}[r]^{P_v}\ar@{->}[ru]^(.70){\lozenge}|\hole & \hB_1^2 \ar@<0.05in>@{->}[d]^{E_v^\dagger} \ar@{<-}[d]_(.52){E_v}\ar@{->}[rd]^<<<<<{\p}   \ar@<0.05in>@{->}[r]^{P_v}\ar@{->}[ru]^(.70){\lozenge}|\hole & \cdots \\ 
		\hB_0^0 \ar@<0.05in>@{->}[r]^{P_v}\ar@{->}[ru]^(.70){\lozenge}|\hole & \hB_0^1 \ar@<0.05in>@{->}[r]^{P_v}\ar@{->}[ru]^(.70){\lozenge}|\hole & \hB_0^2 \ar@<0.05in>@{->}[r]^{P_v}\ar@{->}[ru]^(.70){\lozenge}|\hole &\cdots
	} 
\end{equation} 
 
\subsection{Representatives of vector fields as differential chains} 
\label{sec:representatives_of_vector_fields}               
 \begin{defn}
We say that a \( k \)-chain \( \widetilde{X} \in \hB_k^r(U), \, r \ge 1 \) \emph{represents} a \( k \)-vector field \( X:U \to \L_k \) if \[ \cint_{\widetilde{X}} \o = \int_U \o(X(p)) dV \] for all \( \o \in \B_k^r(U) \). 

If \( X \) is a \( k \)-vector field, let \( E_X: \A_k(U) \to A_{k+\ell}(U) \) be the linear map determined by \( E_X := E_{V_\ell} \circ \cdots \circ E_{V_1} \) where \( X = V_1 \wedge \cdots \wedge V_\ell \). 	
\end{defn}
\begin{thm}\label{thm:jX} 
		If \( X \in {\cal V}^r(U) \) is a \( k \)-vector field where \( U \) is a bounded and open set for \(   r \ge 1 \), then there exists a differential \( k \)-chain \( \widetilde{X} \in \hB_k^r(U) \) which represents \( X \).
\end{thm}

\begin{proof} 
	Let   \( \widetilde{U} \in \hB_n^1(U) \) the \( n \)-chain representing \( U \)  (\S\ref{poly}). Given   \( \a \in \L_k \)   then \( E_\a \perp \widetilde{U} \in \hB_k^1(U) \) represents the constant \( k \)-vector field \( (p;\a) \) defined over \( U \) since for all \( \o \in \B_k^1(U) \), we have \[ \cint_{E_\a \perp \widetilde{U}} \o = \cint_{\perp \widetilde{U}} i_\a \o = \cint_{\widetilde{U}} \star i_\a \o = \int_U \star i_\a \o = \int_U\o(p;\a) dV \] using Theorems  \ref{thm:extV},  \ref{thm:perpcont},  \ref{thm:opensets}, and the definition of interior product.  Since the norms decrease, it follows that \( E_\a \perp \widetilde{U} \in \hB_k^r(U)  \) also represents the constant $k$-vector field \( (p;\a) \), but using test forms in \( \B_k^r(U) \).  

	Now suppose \( X(p) = \sum_I (p; f_I(p)\a_I) \) is a $k$-vector field  in \( U \) where \( f_I \in \B_0^r(U) \). Then \( \widetilde{X} = \sum_I m_{f_I} E_{\a_I} \perp \widetilde{U} \in \hB_k^r(U) \) represents \( X \) since \[\cint_{m_{f_I}E_{\a_I}\perp \widetilde{U}} \o = \cint_{ E_{\a_I} \perp \widetilde{U}} f_I \o = \int_U f_I \o (p;\a_I) dV.\] Therefore, \[ \cint_{\widetilde{X}} \o = \int_U (\sum_I f_I \o) (p; \a_I) dV = \int_U \o(\sum_I (p; f_I(p)\a_I))dV = \int_U\o(X(p)) dV. \] 
\end{proof}   

For \( k = 0 \), we have found a \( 0 \)-chain \( \widetilde{f} \) representing a function \( f \in \B_0^1(U) \). That is, \[ \cint_{\widetilde{f}} g = \int_U f\cdot g dV = \int_U f \wedge \star g \] for all \( g \in \B_0^1(U) \).  As an example of this construction, we have the following corollary: 

\begin{cor}\label{cor:df} 
	Suppose \( f:U \to \R \) is a Lipschitz function where \( U \) is bounded and open in \( \R^n \). Then \[ \cint_{\perp \p \perp \widetilde{f}} \o = - \int_U df \wedge \star \o \] for all \( \o \in \B_1^2(U) \). 
\end{cor}

\begin{proof} 
	Let \( \o \in \B_1^2(U) \). By Theorem \ref{thm:perpcont} and Stokes' Theorem \ref{thm:stokes} we have \[ \cint_{\perp \p \perp \widetilde{f}} \o = \cint_{\widetilde{f}} \star d \star \o = \int_U f \wedge d \star \o = - \int_U df \wedge \star \o. \]  
\end{proof}

In closing, we remark that the definitions and results of this section readily extend to \( k \)-vector fields.

\section{Chainlets} 
\label{sub:chainlets}
Let \( \s = \s^s = \{u_1, \dots, u_s\} \) be a list of vectors in \( \R^n \) and \( P_{\s^s} (p;\a) = P_{u_1} \circ \cdots \circ P_{u_s} (p;\a) \) where \( (p;\a) \) is a simple \( k \)-element. We say \( p \) is the \emph{\textbf{support}}\footnote{a formal treatment of support will be given in \S\ref{sub:support}.} of \( P_{\s^s} (p;\a) \). Then the differential chain \( P_{\s^s} (p;\a) \in \hB_k^{s+1}(U) \) is called a \emph{\textbf{simple \( k \)-element of order}} \( s \). For \( k = 0 \), these are \emph{\textbf{singular distributions}} represented geometrically as differential chains, e.g., Dirac deltas \( (s=0) \), dipoles \( (s = 1) \), and quadrupoles \( (s=2 )\). For \( s \ge 0 \), let \( \A_k^s(p) \) be the subspace of \( \hB_k^{s+1}(U) \) generated by \( k \)-elements of order \( s \) supported at \( p \). Elements of \( \A_k^s(p) \) are called \emph{\textbf{Dirac \( k \)-chains of order \( s \) supported at \( p \)}}.  Let \( S(\R^n) \) be the symmetric algebra and \( S^s(\R^n) \) its \( s \)-th symmetric power. 

\begin{prop}\label{prop:symmet} 
	Let \( p \in U \).  The  vector space \( \A_k^s(p) \)  is isomorphic to  \( S^s(\R^n) \otimes \L_k(\R^n) \).
\end{prop}

\begin{proof} 
	Let \( k = 0 \). The linear map \( P_{u_s} \circ \cdots \circ P_{u_1}(p;\a)\mapsto u_s \circ \cdots \circ u_1 \otimes \a \) is an isomorphism preserving the symmetry of both sides since \( P_{u_1} \circ P_{u_2} = P_{u_2} \circ P_{u_1} \) and \( P_{u_1} \circ (P_{u_2} \circ P_{u_3}) = (P_{u_1} \circ P_{u_2}) \circ P_{u_3} \).
\end{proof} 

\begin{remark}\label{notation}
We may therefore use the notation \( (p;u \otimes \a) = P_u(p;\a) \), and, more generally, \( (p; \s \otimes \a) = (p;(u_s \circ \cdots \circ u_1 )\otimes \a) = P_{u_s} \circ \cdots \circ P_{u_1}(p;\a)  = P_\s (p;\a) \) where \( \s = u_s \circ \cdots \circ u_1 \). For example, \( P_u(p;\s \otimes \a) = (p; (u \circ \s) \otimes \a) \). We will be using both notations, depending on the application.  While \( P_\s \) emphasizes the operator viewpoint, the tensor product notation reveals the algebra of the Koszul complex \( \oplus_{s=0}^\i \oplus_{k=0}^n S^s \otimes \L_k \) more clearly. For example, the scalar \( t \) in \( t(\s \otimes \a) = t \s \otimes \a = \s \otimes t \a\) ``floats,'' while this is not as clear when using the operator notation.	
\end{remark} 

For example, in tensor notation, \( \p_{e_1}(p; 1\otimes e_1) = (p; e_1 \otimes 1) \), \( \p_{e_1}(p;1 \otimes e_1 \wedge e_2) = (p;e_1 \otimes e_2) \) and \( \p_{e_2}(p; 1 \otimes e_1 \wedge e_2) = (p;-e_2 \otimes e_1) \).

\begin{defn}[Unital associative algebra at \( p \)]\label{def:symmetric_algebra}
	Define \( \A_\bullet^\bullet(p):= \bigoplus_{k=0}^n \bigoplus_{s\geq 0} \A_k^s(p) \). 
	Let \( \cdot \) be the product on higher order Dirac chains \( \A_\bullet^\bullet(p)   \) given by \( (p;\s \otimes \a ) \cdot (p; \t \otimes \b) := (p; \s \circ \t \otimes \a \wedge \b) \).	
\end{defn}

\begin{remarks}\mbox{}
	\begin{itemize}
	 	\item For each \( p \in U \) \( \A_\bullet^\bullet(p) \) is a unital, associative algebra with unit \( (p;1) \). 
	 	\item This product is not continuous and does not extend to the topological vector space \( \hB(U) \).
	 	\item \( E_v^\dagger \) is a graded derivation on the algebra \( \A_\bullet^\bullet(p) \). Neither \( E_v \) nor \( P_v \) are derivations. 
	\end{itemize}
\end{remarks}  

 Let \(  \A_k^s(U)  \) be the free space \( \R\<\cup_p  \A_k^s(p) \> \).  In particular, an element of \( \A_k^s(U)  \) is a formal sum \( \sum_{i=1}^m (p_i; \s^i \otimes \a_i) \) where \( p_i \in U \) (\S\ref{notation}).  It follows from Proposition \ref{prop:symmet} that \( \A_k^s(U) \) is isomorphic to the free space \( (S^s \otimes \L_k)\<U\> \). Observe that \( \p(\A_k^s(U)) \subset \A_{k-1}^{s+1}(U) \mbox{ for all }  k \ge 1  \) and \( s \ge 0 \).

 Now \( \A_k^s(U) \) is naturally included in \( \hB_k^{s+1}(U) \), and is thus endowed with the subspace topology\footnote{The inner product we chose on \( \R^n \) induces an inner product on the exterior algebra \( \<\cdot,\cdot\>_\wedge \) using determinant (Definition \ref{def:mass}).  It induces an inner product on the  the symmetric algebra \( \<\cdot,\cdot\>_\circ \) using permanent of a matrix \( \<\s,\t\>_{\circ} := \per(\<u_r, v_s\>) \), and thus on  \( \A_k^s(p) \) via \( \<p;\s \otimes \a ), (p; \t \otimes \b) \>_\otimes :=
\< \s , \t \>_\circ \< \a, \b\>_\wedge \).   However, this inner product \( \<\cdot,\cdot\>_\otimes \) on \( \A_k^s(p) \) does not extend to a continuous inner product on \( \A_k^s(U) \), although it can be useful for computations on Dirac chains of arbitrary order and dimension as long as limits are not taken.}. 
\begin{defn}
Let \( \ch_k^s(U) := \overline{(\A_k^s(U), B^{s+1})} \).   Elements of   \(  \ch_k^s(U) \)  are called $k$-\emph{\textbf{chainlets}}\footnote{Differential chains of class \( \B \) were originally called chainlets.  It is only recently that the author has begun to appreciate the importance of what we now call the chainlet complex, which is a subcomplex of the differential chain complex. }  \emph{\textbf{of order}} \( s \).  	
\end{defn} In the author's experience the chainlet complex \(  \ch_k^s(U) \subset  \hB_k^{s+1}(U)  \) is often sufficient for examples and applications.

\begin{defn}\label{reduction}
	 Define the bilinear map, called  \emph{\textbf{reduction}}, by its action on simple \( k \)-elements of order \( s \): 
	\begin{align*}
	P^\dagger: {\cal V}^r(U) \times \A_k^s(U) &\to \A_k^{s-1}(U) \\	(p; u_1 \circ \cdots \circ u_s \otimes \a)
		&\mapsto \sum_{i=1}^s  \<V(p),u_i\> (p; u_1 \circ \cdots \circ
		\widehat{u_i} \circ \cdots \circ u_s).
	\end{align*}  Let \( P_V^\dagger(p;\s \otimes \a) := P^\dagger(V, (p;\s \otimes \a)) \).  
\end{defn} 
Although this operators extends to \( \A_k^s(U) \) by linearity, it is not continuous and does not extend to \( {\cal C\!h}_k^s(U) \).

\section{Cartesian wedge product} 
\label{sec:cartesian_wedge_product}

\subsection{Definition of $ \hat{\times} $}
\label{Cartesian2}

Suppose \( U_1  \subset \R^n \) and \( U_2 \subset \R^m \) are open sets.  For \( p\in U_1 \) and \( q\in U_2 \), let \( \iota_1^q:  U_1 \to U_1 \times U_2  \) and \( \iota_2^p: U_2 \to U_1 \times U_2   \) be the inclusions \(  \iota_1^q(x) = (x,q)  \) and \(  \iota_2^p(y) = (p,y) \).   Let \( \pi_i:U_1 \times U_2 \to U_i \), \( i=1,2 \), be the usual projections. Define \( \hat{\times}: \A_k(U_1) \times \A_\ell(U_2) \to \A_{k+\ell}(U_1 \times U_2) \) by setting \[ \hat{\times}((p;\a), (q;\b)) := ((p,q); \iota_{1*}^q\a \wedge \iota_{2_*}^p\b), \] where \( (p;\a)\in {\cal A}_k(U_1) \) and \( (q; \b)\in {\cal A}_\ell(U_2) \) are \( k \)- and \(\ell\)-elements, respectively, and extend bilinearly. We call \( P \hat{\times} Q := \hat{\times}(P,Q)\) the \emph{Cartesian wedge product} of \( P \) and \( Q \).       We next show that Cartesian wedge product is continuous.  

\begin{prop} \label{prop:Cartesian} 
	Suppose \( P \in \A_k(U_1) \) and \( Q \in \A_{\ell}(U_2) \) are Dirac chains where \( U_1 \subseteq \R^n, U_2 \subseteq \R^m \) are open. Then \( P \hat{\times} Q \in \A_{k +\ell}(U_1 \times U_2) \) with 
	\[ \|P \hat{\times} Q\|_{B^{r+s, U_1 \times U_2}} \le \|P\|_{B^{r,U_1}}\|Q\|_{B^{s,U_2}}. \] 
\end{prop}

\begin{proof}  
 Choose \( \e > 0 \) and let \( \e' = \e/(\|P\|_{B^{r,U_1}} + \|Q\|_{B^{s,U_2}} +1) < 1 \). There exist decompositions \( P = \sum_{i=0}^{m_1} \D_{\s_i^{j_1(i)}}(p_i;\a_i) \) where \(  0\le j_1(i)\le r \) and \( Q = \sum_{j=0}^{m_2} \D_{\t_j^{j_2(j)}}(q_j;\b_j) \) where \(  0\le j_2(j)\le s \), \(  \D_{\s_i^{j_1(i)}}(p_i;\a_i) \) is inside \( U_1 \), \( \D_{\t_j^{j_2(j)}}(q_j;\b_j) \) is inside \( U_2 \) and such that \( \|P\|_{B^r,U_1} >  \sum_{i=0}^{m_1} \|\s_i^{j_1(i)}\|\|\a_i\| - \e  \)  and \( \|Q\|_{B^r,U_2} >  \sum_{j=0}^{m_2} \|\t_j^{j_2(j)}\|\|b_j\| - \e \).  

Thus \(  \D_{\s_i^{j_1(i)}}(p_i;\a_i) \hat{\times} \D_{\t_j^{j_2(j)}}(q_j;\b_j) \) is inside \( U_1 \times U_2 \) and \( P \hat{\times} Q = \sum_{i=0}^{m_1}\sum_{j=0}^{m_2}  \D_{\s_i^{j_1(i)}}(p_i;\a_i) \hat{\times} \D_{\t_j^{j_2(j)}}(q_j;\b_j) \).  Then 
\begin{align*} 
\|P\hat{\times} Q\|_{B^{r+s,U_1\times U_2}} &\le\sum_{i=0}^{m_1}\sum_{j=0}^{m_2} \|\D_{\s_i^{j_1(i)}}(p_i;\a_i) \hat{\times} \D_{\t_j^{j_2(j)}}(q_j;\b_j)\|_{B^{r+s,U_1\times U_2}}  \\&\le \sum_{i=0}^{m_1}\sum_{j=0}^{m_2}  \|\s_i^{j_1(i)}\|\|\t_j^{j_2(j)}\|\|\a_i\|\|\b_j\| \\&\le   (\|P\|_{B^{r,U_1}} +\e)(\|Q\|_{B^{s,U_2}} +\e)
\end{align*} 
Since this holds for all \( \e > 0 \), the result follows.      
\end{proof}  

Let \( J \in \hB_k^r (U_1) \) and \( K \in \hB_\ell^s(U_2) \). Choose Dirac chains \( P_i \to J \) converging in \( \hB_k^r (U_1) \), and \( Q_i \to K \) converging in \( \hB_\ell^s(U_2) \). Proposition \ref{prop:Cartesian} implies that \( \{P_i \hat{\times} Q_i\} \) is Cauchy.

\begin{defn}
	Define \[ J \hat{\times} K:= \lim_{i \to \i} P_i \hat{\times} Q_i. \]	
\end{defn} 

\begin{thm}\label{thm:cart} 
	Cartesian wedge product \( \hat{\times}: \hB_k^r(U_1) \times
\hB_\ell^s(U_2) \to \hB_{k+\ell}^{r+s}(U_1 \times U_2) \) is associative,
bilinear and continuous for all open sets \( U_1 \subseteq \R^n, U_2 \subseteq \R^m \) and
satisfies
\begin{enumerate} 
		\item \( \|J\hat{\times} K\|_{B^{r+s, U_1 \times U_2}} \le \|J\|_{B^{r,U_1}} \|K\|_{B^{s,U_2}} \)  for \( 0 \le r,s < i \);
		\item
\begin{align*} 
\p(J \hat{\times} K) =\begin{cases} ( \partial J) \hat{\times} K + (-1)^k J \hat{\times} ( \partial K), &k > 0, \ell >0 \\ ( \partial J) \hat{\times} K, &k > 0, \ell = 0\\ J \hat{\times} ( \partial K), &k = 0, \ell > 0 \end{cases} 
\end{align*} 
		\item \( J \hat{\times} K = 0 \) implies \( J = 0 \) or \( K = 0 \);  
		\item \( (p; \s \otimes \a) \hat{\times} (q; \t \otimes \b) = ((p,q); \s
\circ \t \otimes \iota_{1*}\a \wedge \iota_{2*}\b) \) where \( \s \otimes \a \) and \( \t \otimes \b \) are elements of the Koszul complex \( \oplus_{s=0}^\i \oplus_{k=0}^n S^s \otimes \L_k \) (Remark \ref{notation});
		\item  \( (\pi_{1*}\o \wedge \pi_{2*}\eta)(  J \hat{\times}   K) = \o(J)\eta(K) \)  for \( \o \in \B_k^r(U_1), \eta \in \B_\ell^s(U_2) \);
		\item \( \supp(J \hat{\times} K) = \supp(J) \times \supp(K) \).
\end{enumerate}
\end{thm}

\begin{proof} 
(a):  This is a consequence of Proposition \ref{prop:Cartesian}.

(b):  This follows from   Theorem \ref{thm:bod}:
 \( \p ((p,q); \a \wedge \b) = (\p ((p,q);\a)) \cdot ((p,q); \b) +
(-1)^k ((p,q);\a) \cdot (\p ((p,q);\b))  =(\p (p;\a)) \hat{\times} (q; \b) +
(-1)^k (p;\a) \hat{\times} (\p (q;\b)) \), for all \( (p;\a) \in \A_k(U_1), \text{ and } (p;\b)\in {\cal P}_\ell(U_2) \). The boundary relations for Dirac chains follow by linearity. We know  \(\hat{\times} \) is continuous by (a) and \( \p \) is continuous by Theorem \ref{thm:bod}, and therefore the relations extend to \( J \in\hB_k^r(U_1), K \in
\hB_\ell^s(U_2) \), or \( J \in  
\hB_k(U_1) \) and \( K \in \hB_\ell(U_2) \).

 (c): Suppose \( J \hat{\times} K = 0 \), \( J \ne 0 \) and \( K \ne 0 \). Let \( \{e_i\} \) be an orthonormal basis of \( \R^{n+m} \), respecting the Cartesian wedge product. Suppose \( e_I \) is a \( k \)-vector in \( \L_k(\R^n) \) and \( e_L \) is an \( \ell \)-vector in \( \L_\ell(\R^m) \). Then
\begin{equation*} \cint_Q \left( \cint_P f de_I
\right) g de_L = \cint_{P \hat{\times} Q} h de_I de_L 
\end{equation*} 
where \( h(x,y) = f(x) g(y) \), \( P \in \A_k(U_1), Q \in \A_\ell(U_2) \). The proof follows easily by writing \( P = \sum_{i=1}^r (p_i; \a_i) \) and \( Q = \sum_{j=1}^s(q_i; \b_i) \) and expanding. By continuity of Cartesian wedge product \( \hat{\times} \) and the integral, we deduce
\begin{equation*} \cint_K \left( \cint_J f de_I \right) g de_L = \cint_{J \hat{\times} K}
h de_I de_L 
\end{equation*} 
where \( h(x,y) = f(x) g(y) \), \( J \in \hB_k(U_1), K \in \hB_\ell(U_2) \). Since \( J\ne 0 \) and \( K \ne 0 \), there exist \( f \in \B_0^(U_1), g \in \B_0(U_2) \) such that \( \cint_J fde_I \ne 0 \), \( \cint_K gde_L \ne 0 \). Therefore, \( \cint_K \left( \cint_J f de_I \right) g de_L = \cint_J f de_I \cint_K g de_L \ne 0 \), which implies \( \cint_{J \hat{\times} K} h de_I de_L \ne 0 \), contradicting the assumption that \( J \hat{\times} K = 0 \).

 (d): We first show that \( (p; \s \otimes \a) \hat{\times} (q; \b) = ((p;q); \s \otimes \a \wedge \b) \). The proof is by induction on the order \( j \) of \( \s \). This holds if \( j = 0 \), by definition of \( \hat{\times} \). Assume it holds for order \( j-1 \). Suppose that \( \s \) has order \( j \). Let \( \s = u \circ \s' \). By Proposition \ref{prop:Cartesian} 
\begin{align*} (p; \s \otimes \a) \hat{\times} (q; \b) &= (\lim_{t \to 0} (p +tu; \s' \circ \a/t) - (p; \s' \circ\a/t)) \hat{\times} (q;\b) \\&= \lim_{t \to 0} (p +tu; \s' \circ \a/t)\hat{\times} (q;\b) - (p; \s' \circ \a/t) \hat{\times} (q;\b) \\&= \lim_{t \to 0} ((p+tu,q);\s' \circ \a \wedge \b/t) - ((p,q); \s' \circ \a \wedge \b/t) \\&= P_u((p,q); \s' \circ \a \wedge \b) = ((p,q); \s' \circ u \otimes \a \wedge \b) = ((p,q); \s \circ \a\wedge \b). \end{align*} 
In a similar way, one can show \( (p; \s \otimes \a) \hat{\times} (q; \t \otimes \b) = ((p;q); \s \circ \t \otimes \a \wedge \b) \) using induction on the order \( j \) of \( \s \).

(e): This follows from the definition of \( \hat{\times} \) for simple elements: 
\begin{align*} (\pi^{1*}\o \wedge \pi^{2*}\eta)( (p;\a) \hat{\times} (q;\b)) &= (\pi^{1*}\o \wedge \pi^{2*}\eta)((p,q); \iota_{1*}\a \wedge \iota_{2_*}\b) \\&=\o(p;\a)\eta(q;\b).
\end{align*} 
This extends to Dirac chains \( A \hat{\times} B \) by first fixing \( A \) and extending linearly in the second variable, and then fixing \( B \) and extending linearly in the first variable.    Continuity follows by Proposition \ref{prop:Cartesian}.

(f): This follows from (e), the definition of support, and the definition of \( \hat{\times} \).
\end{proof}   

Cartesian wedge product extends to a continuous bilinear map \( \hat{\times}: \hB_k(U_1) \times \hB_\ell(U_2) \to \hB_{k+\ell}(U_1 \times U_2)\)  and the relations (b)-(f) continue to hold. 

\begin{cor}\label{cor:carte}
If \( \{e_1, \dots, e_n\} \) is a basis of \( \R^n \) and \( \{f_1, \dots, f_m\} \) is a basis of \( \R^m \), then \( \p_{e_i}(J \hat{\times} K) =  (\p_{e_i} J )\hat{\times} K \) and \( \p_{f_j}(J \hat{\times} K) = (-1)^{\dim K} J \hat{\times} \p_{f_j} K \) for all 	\( J \in\hB_k^r(U_1) \mbox{ and } K \in
	\hB_\ell^s(U_2) \), or \( J \in  
	\hB_k(U_1) \) and \( K \in \hB_\ell(U_2) \).   
\end{cor}

\begin{proof} 
This follows from Theorem \ref{thm:cart} since \( \p_{e_i} K = 0 \) and \( \p_{f_j}J = 0 \).
\end{proof}

\begin{example}\label{carte} 
Recall that if \( A \) is affine \( k \)-cell in \( U_1 \), then \( A \) is represented by an element \( \widetilde{A} \in \hB_k(U_1) \). If \( A \) and \( B \) are affine \( k \)- and \( \ell \)-cells in \( \R^n \) and \( \R^m \), respectively, then the classical Cartesian product \( A \times B \) is an affine \( (k+\ell) \)-cell in \( \R^{n+m} \). The chain \( \widetilde{A \times B} \) representing \( A \times B \) satisfies \( \widetilde{A \times B} = \widetilde{A} \hat{\times} \widetilde{B} \).
\end{example}

\begin{remarks} \mbox{}
	\begin{itemize} 
			\item The boundary relations hold if we replace boundary \( \p \) with the directional boundary \( \p_v = P_v E_v^\dagger \) for \( v \in \R^{n+m} \).
			\item  According to Fleming (\cite{fleming}, \S 6), ``It is not possible to give a satisfactory definition of the cartesian product \( A \times B \) of two arbitrary flat chains.''  Fleming defines   ``Cartesian product'' on polyhedral chains, and this coincides with our Cartesian wedge product according to Example \ref{carte}. 
	\end{itemize} 
\end{remarks} 
  
A refinement of Cartesian wedge product is important for applications of this theory to Plateau's problem \cite{plateau10}.   
\begin{defn} \label{def:masschain}
		Let \( J \in \hB_k^r(U) \). Define \[ M(J) := \inf\{ \liminf \|A_i\|_{B^0}: A_i \to J \mbox{ in the } B^r \mbox{ norm }\}. \]
	\end{defn}   An important example is \( J = \widetilde{(a,b)} \in \hB_1^1 \). It is not hard to prove that \( M(J) = b-a \).     
	 
   Recall the translation operator \( T_v \) and \( \D_v := T_v - Id \).  
  \begin{lem}\label{lem:vmass}
 		    \( \|\D_v J\|_{B^r} \le \|v\|M(J) \)  for all \( J \in \hB_k^r(U) \) and \( r \ge 1 \).    
 		\end{lem} 
 	
 		\begin{proof} This follows since \( \|\D_v A\|_{B^r} \le \|v\|\|A\|_{B^{r-1}} \le \|v\|\|A\|_{B^0} \) for Dirac chains \( A \) and by the definition of \( M(J) \).  
 		\end{proof} 

We can refine the estimate if the mass of one of the chains is finite. 		  
   	
	\begin{prop}\label{prop:productfinitemass}
	Suppose \( J \in \hB_k^1(U_1) \) satisfies \( M(J) < \i \) and  \( K \in \hB_\ell^r(U_2) \). Then  \[ \|J\hat{\times} K\|_{B^{r, U_1 \times U_2}} \le M(J) \|K\|_{B^{r,U_2}}. \]
	\end{prop} 
	
	\begin{proof} 
		Since \( (p;\a)  \hat{\times} (q; \b) = ((p,q); \a \wedge \b) \) it follows that \(   (p;\a) \hat{\times} K  = (p;1) \hat{\times} E_\a K \in \hB_{\ell+1}^1(U_1 \times U_2) \).  Hence 
		\[ \| (p;\a) \hat{\times} K\|_{B^{r, U_1 \times U_2}} = \|(p;1) \hat{\times} E_\a K\|_{B^{r, U_1 \times U_2}} = \|E_\a K\|_{B^{r, U_2}} \le \|\a\|\|K\|_{B^{r, U_2}}. \]

		 Using  Lemma \ref{lem:vmass}	
				\begin{align}\label{eq:deltas} \| J \hat{\times} \D_u(p; \a) \|_{B^{r, U_1 \times U_2}}  = \|\D_u  E_\a J)\|_{B^{r,U_2}}  \le \|\a\|\|\D_u   J)\|_{B^{r,U_2}} \le  \|\a\|\|u\|M(J).  
				\end{align}   
		 Let \( A \) be a Dirac \( k \)-chain. Given \( \e > 0 \) we can use Definition \ref{def:norms} to write 
			 \begin{align} \label{eq:b1}
			        \|A\|_{B^r} > \sum_{i=0}^r \sum_{i=0}^m \|\s_i^{j(i)}\|\|\a_i\|- \e.
			 \end{align} where \(  A=\sum_{i=0}^r \D_{\s_i^{j(i)}}(p_i;\a_i) \).
				Using \eqref{eq:deltas} and  \eqref{eq:b1}  we deduce 
						\begin{align}\label{eq:hattimes}
						  \|J \hat{\times} A\|_{B^{r, U_1 \times U_2}} &\le  \sum_{i=0}^r \| J \hat{\times} \D_{s_i^{j(i)}}(p_i;\a_i) \|_{B^{r, U_1 \times U_2}}   \\&\le   M(J) \sum_{i=0}^r \|\s_i^{j(i)}\| \|\a_i\|  \notag\\&< M(J)( \|A\|_{B^r} + \e).   \notag
			 \end{align}		

				 Suppose \( A_i \to K \) are Dirac chains converging in the \( B^{r,U_1} \) norm.  We show \( \{ J \hat{\times} A_i\} \) is Cauchy in the \( B^{r,U_1 \times U_2} \) norm:  Using \eqref{eq:hattimes} and since \( \{A_i\} \) is Cauchy we have 
				\[
					\|J \hat{\times} A_i -  J \hat{\times} A_j\|_{B^{r, U_1 \times U_2}}  = \|J \hat{\times} (A_i-A_j )\|_{B^{r, U_1 \times U_2}}  \le  M(J)\|A_i -A_j\|_{B^{r,U_1}}.
				\]
				Since \( J \hat{\times} K := \lim_{i \to \i}  A_i \hat{\times} K \), it follows from \eqref{eq:hattimes} that 
		\begin{align*}
		    \|J \hat{\times} K  \|_{B^{r,U_1 \times U_2  }} &= \lim_{i \to \i}\| A_i \hat{\times} K\|_{B^{r,U_1 \times U_2}} \\&\le M(J) \lim_{i \to \i}\|A_i\|_{B^{r,U_2}} \\&=  M(J)\|K\|_{B^{r,U_2}}, 
		\end{align*} 
		  as desired. 
	\end{proof}
	This result is used in \cite{plateau10} to define the cone operator on \( K \in \hB_{\ell}^r  \) by setting \( J = \widetilde{(a,b)} \).   	  
	
\section{Fundamental theorems of calculus for chains in a flow} 
\label{sec:fundamental_theorems_of_calculus_for_chains_in_a_flow}

\subsection{Evolving chains} 
\label{sub:evolving_chains}

Let \( J \in \hB_k^r(U) \) have compact support in \( U \) and \( V \in {\cal V}^r(U)\) where \( U \) is open in \( \R^n \) where \( 0 \le r \le \i \). Let \( V_t \) be the time \( t \) map of the flow of \( V \) and \( J_t:= V_{t*}J \). For each \( p \in U \), the image \( V_t(p) \) is well-defined in \( U \) for sufficiently small \( t \). Since \( J \) has compact support, there exists \( t_0> 0 \) such that pushforward \( V_{t*} \) is defined on \( J \) for all \(0 \le t < t_0 \). Let \( \theta: U \times [0,t_0) \to U \) be given by \( \theta(p,t) := V_t(p) \). Then  \( \theta \in {\cal M}^{r+1}(U \times [0,t_0), U) \). By Theorem \ref{thm:opensets}, Corollary \ref{cor:pull}, Theorem \ref{thm:retX}, and Theorem \ref{thm:cart} we deduce  
\begin{equation}\label{FTCPlus}  
[J_t]_a^b := \theta_* E_{e_{n+1}}^\dagger ( J \hat{\times} \widetilde{(a,b)} ), 
\end{equation}  
for \(  [a,b] \subset [0,t_0) \), is a well-defined element of \( \hB_k^r(U \times (a,b)) \) where \( \widetilde{(a,b)} \) is the chain representing \( (a,b) \), and \( e_{n+1} \in \R^{n+1} \) is unit. Thus \( [J_t]_a^b \) is a single \( k \)-chain, the \emph{\textbf{flowing chain
 of \( J \)}},  which is distinct from the collection of \( k \)-chains \( \{J_t\}_{a \le t \le b} \), called an \emph{\textbf{evolving chain}} in the flow of the vector field \( V \).    

\begin{figure}[ht]
  	\centering
  		\includegraphics[height=1.75in]{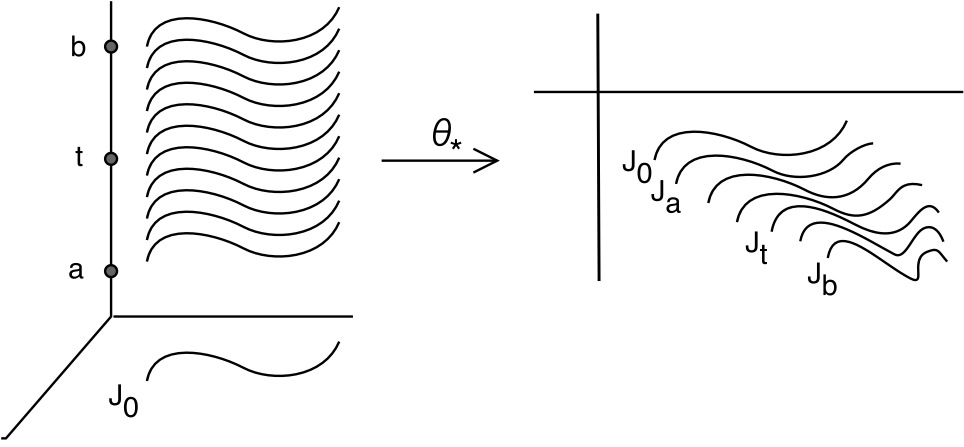}
  	\caption{A flowing chain}
  	\label{fig:MovingChains}
\end{figure}
      
\begin{examples} \mbox{}
		\begin{enumerate}
			\item Let \( J = (p; 1) \) where \( p \in U \) and \( X \) a Lipschitz vector field.  Then \( [J_t]_0^{t_0}   \) represents the parametrized path of \( p \) along the integral curve of \( X \) through \( p \) and is an element of \( \hB_0^1(U) \) while  \( [(p; X(p))_t]_0^{t_0}  \) represents the one-dimensional version of the curve and is an element of \( \hB_1^1(U) \).  
			\item Let \( A  \in \A_0^s(U) \).  For \( s = 0 \), then \( [A_t]_0^{t_0}  \) is the sum of representatives of finitely many parametrized paths starting at the points in the support of \( A \). Of course, \( t_0 \) must be chosen so that each path is contained in \( U \).  For \( s = 1 \), we obtain the path of the support of a dipole \( (p; u \otimes 1) \) as it moves in \( U \).   
			\item  Pushforward  modifies the mass of a \( k \)-element according to \( ( F_* (p;\a) = (F(p); F_{p *} \a) \).  However,   \( \perp F_* \perp (p; \a) = (F(p); \a) \) for all \( n \)-vectors \( \a \).     
			  The latter is useful for modeling incompressible fluids. One can also consider \( A \in \A_k^s(U) \) for any \( 0 \le k \le n \) and \( s \ge 0 \) and model particles, including dipoles of any order.  Fluids which are partially compressible can be modeled by using a linear combination of the two approaches  \(  c_1\, (\perp F_* \perp) + c_2\, (F_*) \) .   For \( 0 < k < n \), then \( \perp F_* \perp \) modifies a simple $k$-element \( (p;\a) \), i.e., changes its mass and $k$-direction,  according to the infinitesimal change of mass and direction of the simple \( (n-k) \)-element   normal to \(  (p;\a) \) while \( F_*   \) alters mass and direction according to how \( F_* \) acts on \( (p;\a) \) itself.  
 			
			 We can take linear combinations to model a flowing material which is partly compressible:   \( c \theta_* E_{e_{n+1}}^\dagger ( J \hat{\times} \widetilde{(a,b)}  + (1-c) \perp \theta_* \perp_t E_{e_{n+1}}^\dagger ( J \hat{\times} \widetilde{(a,b)} \) where \( \perp_t \) is \( \perp \) restricted to the slice of \( \R^{n+1} \) at time \( t \).
			
			\item Let \( \widetilde{M} \) represent a submanifold \( M \) of \( U \).  Then \( [\widetilde{M}_t]_0^{t_0} \) is the differential chain representing the evolution of \( M \), and takes into account infinitesimal distortions induced by the flow of \( X \).     
		\end{enumerate}
\end{examples}

It is not difficult to define evolving chains in space-time and the author is developing extensions of Reynolds' transport and its applications from this viewpoint in \cite{reynolds}. 

\begin{lem}\label{lem:movingchain}  
	The following relations hold:
\begin{enumerate}
 	\item \( \theta_*(J \hat{\times} (t;1)) = V_{t*} J \);
 	\item  \( P_V \theta_* = \theta_* P_{e_{n+1}} \);
 	\item  \( \p[J_t]_a^b = [\p J_t]_a^b \);
\end{enumerate}
\end{lem}

\begin{proof}\mbox{} 
	(a):  It suffices to prove this for \( J = (p;\a) \), a simple \( k \)-element with \( p \in U \).  But \( \theta_*((p;\a) \hat{\times} (t;1))  = \theta_*((p,t); \iota_{1*}\a) = V_{t*} (p;\a) \) since \( \theta(p,t) = V_t(p) \). 
	
	(b): By Theorem \ref{thm:mF} (c), \( \theta_* P_{e_{n+1}} = P_{\theta e_{n+1}} \theta_* = P_V \theta_* \) 

(c):  

\begin{align*}
   \p[J_t]_a^b &= \p \theta_* E_{e_{n+1}}^\dagger(J \hat{\times} \widetilde{(a,b)}) && \text{by Definition \ref{FTCPlus} }  \\&
 =     \theta_*\p E_{e_{n+1}}^\dagger(J \hat{\times} \widetilde{(a,b)})  && \text{by Proposition \ref{prop:whitneybd}}  \\&
    =   \theta_* E_{e_{n+1}}^\dagger \p (J \hat{\times} \widetilde{(a,b)})  && \text{by Theorem \ref{thm:Gpush}(b)}   
 \\& =   \theta_* E_{e_{n+1}}^\dagger (\p J \hat{\times} \widetilde{(a,b)} - J \hat{\times} ((b;1) -  (a;1)) ) && \text{by Theorem \ref{thm:cart}(b)}  
 \\& = [(\p J)_t]_a^b  &&\text{by \eqref{FTCPlus} applied to }   \p J  \mbox{ and  }\\&\quad &&     \text{since } E_{e_{n+1}}^\dagger( J \hat{\times} ((b;1) -  (a;1))) = 0. 
\end{align*}  
\end{proof} 

\subsection{New fundamental theorems} 
\label{sub:fundamental_theorem}

Given \( (a,b) \subset \R \), \( t\in \R \), and a differential chain \( J \), let \( J_a, J_b \) and \( J_t \) be as in  \S\ref{sub:evolving_chains}. 

\begin{thm}[Fundamental theorems for chains in a flow]\label{thm:Lieder} 
Suppose \( J \in \hB_k^r(U) \) is a differential chain with compact support in \( U \), \( V \in {\cal V}^r(U)\) is a vector field, and \( \o \in \B_k^{r+1}(U) \) is a differential form where \( U \) is open in \( \R^n \). Then 
\[ \cint_{J_b} \o - \cint_{J_a} \o = \cint_{[J_t]_a^b} L_V \o  =  \int_a^b\left(\cint_{J_t} L_V\o \right) dt. \]
\end{thm}
 
\begin{proof} 
 Since Dirac chains are dense, and both \( \p_v \) and Cartesian wedge product are continuous for all \( v \in \R^n \), we use Corollary \ref{cor:carte} to conclude \[ \p_{e_1}( \widetilde{(a,b)} \hat{\times} J) = (b;1) \hat{\times} J - (a;1) \hat{\times} J. \]   
  Finally we change coordinates to consider a local flow chart of \( V \): By Theorem \ref{thm:mF} we know \( \p_{\theta e_1} \theta_* = \theta_* \p_{e_1} \). Using Theorem \ref{thm:retXint} we deduce 
\begin{align*} 
	\cint_{[J_t]_a^b} L_V \o &= \cint_{P_V [J_t]_a^b} \o &&\text{by Theorem \ref{thm:predfirst}} \\& = \cint_{P_V\theta_* E_{e_{n+1}}^\dagger ( J \hat{\times} \widetilde{(a,b)})} \o &&\text{by Definition \ref{FTCPlus}} \\& = \cint_{\theta_* P_{e_{n+1}}E_{e_{n+1}}^\dagger (J \hat{\times} \widetilde{(a,b)})} \o &&\text{by Theorem \ref{thm:mF}} \\& = \cint_{ \theta_* \p_{e_{n+1}}(J \hat{\times} \widetilde{(a,b)}) } \o &&\text{by the definition of directional boundary  \ref{dirbound}} \\& = \cint_{\theta_*(J \hat{\times} (b;1) )- \theta_*(J \hat{\times} (a;1) ) } \o &&\text{since } \p(\widetilde{(a,b)}) = (b;1)-(a;1) \\& = \cint_{V_{b*}J} \o -\cint_{V_{a*}J} \o && \quad. 
\end{align*} 

We next establish the second equality:

Let \( f(t) = \cint_{J_t} \o \). It suffices to show that \( f'(t) = \cint_{P_V J_t} \o \). The result will then follow from the fundamental theorem of calculus. Since \( P_V \) is continuous, \( P_V J_t = \lim_{h \to 0} \frac{J_{t+h} - J_t}{h} \) for each \( t \in [a,b] \). Hence
\begin{align*}
f'(t) = \lim_{h \to 0} \frac{f(t+h)-f(t)}{h} = \lim_{h \to 0} \cint_{(J_{t+h}-J_t)/h} \o  = \cint_{P_V J_t}   \o
\end{align*} 
as we hoped. Then \( \int_a^b\left(\cint_{P_V J_t} \o \right) dt = \cint_{J_b} \o - \cint_{J_a} \o \). The result follows by Theorem \ref{thm:preV}.
\end{proof}

  If \( J_a \) is a \( 0 \)-element, then this is the fundamental theorem of calculus for integral curves. 

\begin{thm}[Stokes' theorem for flowing chains]\label{thm:st} 
	Let \( \o \in \B_{k-1}^{r+1}(U) \) be a differential form, \( J \in \hB_k^r(U) \) a differential chain, and \( V \in {\cal V}^r(U) \) a vector field. Then 
	\[ \cint_{[J_t]_a^b} d L_V \o = \cint_{\p J_b} \o - \cint_{\p J_a} \o. \] 
\end{thm} 

\begin{proof} 
This follows directly from Stokes' Theorem \ref{thm:stokes},  Theorem \ref{thm:Lieder}, and Lemma  \ref{lem:movingchain} (c).
\end{proof}

\begin{cor}\label{cor:exactL} 
	Let \( \o \in \B_{k-1}^{r+1}(U) \) be a differential form, \( J \in \hB_k^r(U) \) a differential chain, and \( V \in {\cal V}^r(U) \) a vector field. If \( L_V \o \) is closed, then \( \cint_{\p J_b} \o = \cint_{\p J_a} \o \).
\end{cor} 

\subsection{Differentiation of the integral} 
\label{sub:differentiation_of_the_integral}

\begin{defn}
	We say that a one-parameter family of differential $k$-chains \( J_t \in \hB_k^r(U), r \ge 0, \) or \( J_t \in \hB_k(U)  \), is \emph{\textbf{differentiable at time \( t \)}} if \( \frac{\p}{\p t} J_t  = \lim_{h \to 0} (J_{t+h} - J_t) /h  \) is well-defined in \(\hB_k^r(U) \), respectively \( \hB_k(U) \).   Similarly, a one-parameter family of differential $k$-forms \( \o_t \in  \B_k^r(U), r \ge 0 \), or \( \o_t \in \B_k(U) \) is \emph{differentiable at time \( t \)} if \( \frac{\p}{\p t} \o_t  = \lim_{h \to 0} (\o_{t+h} - o_t) /h  \) is well-defined in \( \B_k^r(U) \), respectively \(\B_k(U) \).    
	
\end{defn}

\begin{example}
	Suppose \( f_1:[0,1] \to \R \) is a Lebesgue integrable function and  \( f_t:[0,1] \to \R \) is a monotone increasing sequence of step functions converging to \( f \).  Let \( J_t \) be the differential $1$-chain representing the graph of \( f_t, 0 \le t < 1 \).  For convenience, set \( f_t = f_1 \) for \( 1 \le t \le 2 \).  It follows that any subsequence   \( J_{t_j} \) with increasing \( t_j \le 1 \)  is Cauchy in  \( \hB_1^1(\R^2) \).    Let \( J_1 = \lim_{j \to \i} J_{t_j} \). While    \( \supp(J_1) \) is the graph of \( f_1 \),   \( \supp(\p J_1) \) is the set of discontinuity points in the graph of \( f_1 \), together with the endpoints.   Furthermore \( \frac{\p}{\p t} J_s \) exists in \( \hB_1^2(\R^2) \) and is a $1$-chainlet of dipole order $1$ for each \( 0 < s < 2 \).   
 \end{example}

Numerous other examples are provided using flows.   If \( J_t \) is an \emph{evolving differential $k$-chain}, that is, the pushforward of a $k$-chain \( J \) via the time \( t \) map of a smooth vector field \( V \),  flow,  it follows directly from the definitions that \( J_t \) is differentiable at time \( t \) for each \( t \) of definition,  and 
\begin{equation}\label{ptJ}
	 \frac{\p}{\p t} J_t = P_V J_t.
\end{equation}

 Our next result extends the classical Leibniz Integral Rule to families of differential chains which are differentiable in time.  
\begin{thm}[Generalized Leibniz Integral Rule] \label{thm:Leibniz}   
If \( J_t \in \hB_k^r(U) \) and \( \o_t \in \B_k^r(U) \) are differentiable in time, then
 \[ \frac{\p}{\p t} \cint_{J_t} \o_t =  \cint_{J_t} \frac{\p}{\p t} \o_t + \cint_{\frac{\p}{\p t} J_t} \o_t. \]
\end{thm} 

\begin{proof} 
\begin{align*}
  \frac{\p}{\p t} \cint_{J_t} \o_t = \lim_{h \to 0}  \cint_{\frac{J_{t+h} - J_{t}}{h}} \o_t + \cint_{J_t} \frac{\o_{t+h} - \o_t}{h} =  \cint_{\frac{\p}{\p t} J_t} \o_t + \cint_{J_t} \frac{\p}{\p t} \o_t.
\end{align*} 
\end{proof} 

If \( J_t \) is constant with respect to time, we get \(  \frac{\p}{\p t} \cint_{J_t} \o_t =   \cint_{J_t} \frac{\p}{\p t} \o_t   \), as expected.  This result holds for all families of of differential chains and differential forms that are differentiable in time.  For chains evolving under a flow, we obtain powerful generalizations of Differentiation of the Integral and the Reynolds' Transport theorem of continuum mechanics since  \( J_0 \) can be any differential chain and \( V \) any sufficiently smooth vector field.  
 
   For evolving chains, we immediately deduce from Equation \eqref{ptJ}, Theorem \ref{thm:Leibniz} and the duality \( L_V \o = \o P_V \) (see \S\ref{sub:prederivative}):

 \begin{thm}[Differentiating the Integral]\label{cor:differintegral}  
	Suppose \( \{J_t\} \in \B_k^r(U) \) is an evolving differential $k$-chain under the flow of a vector field \( V \in {\cal V}^{r+1}(U) \), and \( \{\o_t\} \in \B_k^{r+1}(U) \) is a family of smooth differential \( k \)-form differentiable in time.   Then
\[
	\frac{\p}{\p t} \cint_{J_t} \o_t = \cint_{J_t}\left( \frac{\p}{\p t} \o_t + L_V \o_t\right).
\]  
\end{thm}

The classical method of differentiating the integral has been extremely useful in engineering and physics\footnote{Feynman wrote in his autobiography \cite{joking}, ``That book [Advanced Calculus, by Wood.] also showed how to differentiate parameters under the integral sign —- it's a certain operation. It turns out that's not taught very much in the universities; they don't emphasize it. But I caught on how to use that method, and I used that one damn tool again and again. So because I was self-taught using that book, I had peculiar methods of doing integrals.
The result was, when guys at MIT or Princeton had trouble doing a certain integral, it was because they couldn't do it with the standard methods they had learned in school. If it was contour integration, they would have found it; if it was a simple series expansion, they would have found it. Then I come along and try differentiating under the integral sign, and often it worked. So I got a great reputation for doing integrals, only because my box of tools was different from everybody else's, and they had tried all their tools on it before giving the problem to me.''}.  
From  the duality  \( d\o = \o \p \) (Theorem \ref{thm:stokes}) and  Cartan's Magic Formula \( L_V = d i_V + i_V d \)  we  immediately deduce:
\begin{cor}\label{cor:lei} 
	If \( \{\o_t\} \) is a family of smooth differential \( k \)-form differentiable in time, then  
		\[
			\frac{\p}{\p t} \cint_{J_t} \o_t = \cint_{J_t} \frac{\p}{\p t} \o_t +  \cint_{\p J_t} i_V \o_t.
		\] If \( \{J_t\} \) is an evolving cycle, then 
		\[
			\frac{\p}{\p t} \cint_{J_t} \o_t =     \cint_{J_t} \frac{\p}{\p t} \o_t +  \cint_{  J_t} i_V d \o_t.
		\]	
\end{cor}  

\begin{cor}[Generalization of the Reynolds' Transport Theorem to nonsmooth domains, evolving in a flow]\label{cor:Reynolds} 		
	Suppose \( J_t \in \B_k^r(U) \) is an evolving differential chain under the flow of a vector field \( V \in {\cal V}^{r+1}(U) \), and \( \o_t \in \B_k^{r+1}(U) \) is a smooth differential form which is differentiable in time.  Then  
	\[
		\frac{\p}{\p t} \left[ \cint_{J_t} \o_t \right] = \cint_{J_t} \frac{\p}{\p t} \o_t + \cint_{\p J_t}i_X \o_t + \cint_{E_X J_t} d \o_t = \cint_{J_t} \frac{\p}{\p t} \o_t + \cint_{P_X J_t} \o_t. 
	\] 
	\end{cor}
  
The classical version, a cornerstone of mechanics and fluid dynamics, is a straightforward consequence.  Our result applies to all domains of integration represented by differential chains.   
 
\section{Differential chains in manifolds} 
\label{sec:differential_chains_in_manifolds}

The Levi-Civita connection and metric can be used to extend this theory to open subsets of Riemannian manifolds. (See \cite{thesis}, as well.) We sketch some of the main ideas here.  Let  \( {\cal R}_{U,M} \) be the category whose objects are pairs \( (U, M) \) of open subsets \( U \) of  Riemannian manifolds \( M \), and morphisms are smooth maps.  Let \( TVS \) be the category whose objects are locally convex topological vector spaces, and morphisms are continuous linear maps between them.  Then  \( \hB \) is a functor from  \( {\cal R}_{U,M} \) to \( TVS \). We define \( \hB(U) =  \hB(U, M) \) much as we did for \( M = \R^n \), but with a few changes:   Norms of tangent vectors and masses of \( k \)-vectors are defined using the metric.  The connection can be used to define \( B^r \) norms of vector fields.  Dirac \( k \)-chains are defined as formal sums \( \sum (p_i; \a_i) \) where \( p_i \in U \), \( \a_i \in \L_k(T_{(p_i)}(M)) \).  The vector space of Dirac \( k \)-chains is \( \A_k(U, M) \).  Difference chains are defined using pushforward along the flow of   locally defined, unit vector fields.      From here, we can define the \( B^r \) norms on \( \A_k(U, M) \), and the \( B^r \) norms on forms, and thus the topological vector spaces  \( \hB_k(U, M) \) and \( \B_k(U,M) \).  Useful tools include naturality of the operators    \S\ref{ssub:naturality_of_the_operators}, partitions of unity \S\ref{ssub:partitions_of_unity}, and cosheaves.  
\newpage
\section{Notation}               
	\label{sub:notation_review_for_differential_chains}
	\begin{itemize}  
		\item \( (p;\a) \) \( k \)-element (Definition \ref{diracchains});
		\item \( \|A\|_{B^0} \) mass of a Dirac chain (Definition \ref{def:mass});
		\item \( \A_k(U) \) Dirac chains (Definition \ref{diracchains});
		\item \( T_u \) translation operator (Definition \ref{def:morenotation});
		\item \( \D_u(J) = T_u(J) - J \) difference chain (Definition \ref{def:morenotation}); 
		\item \( \D_{\s^j}(p; \a) \) \( j \)-difference \( k \)-chain (Definition \ref{def:morenotation});
		\item \( \|A\|_{B^r} \) \( B^r \) norm (Definition \ref{def:norms});  
		\item \( \hB_k^r(U) \) differential chains in \( U \) \ref{sub:the};
		\item \( u_k^{r,s} \) linking maps (Lemma \ref{lem:varinj}); 
		\item \( H_k(U) \) hull of linking maps  (Definition \ref{hk}) ;
		\item \( \cint_J \o \) evaluation \S\ref{ssub:integration}; 
		\item \( [S,T] = ST - TS \) and \( \{S, T\} = ST + TS \) where \( S \) and \( T \) are operators. 
		\item \( m_f \) multiplication by a function (Definition \ref{def:mf});  
		\item \( F_* \) pushforward (Definition \ref{push}); 
		\item \( \supp(J) \) support of a chain (Definition\ref{support});
		\item \( E_X \) extrusion operator (Definition \ref{extrusionX});
		\item \( E_X^\dagger \) retraction operator (Definition \ref{retractionX}); 
		\item \( P_X \) prederivative operator (Definitions \ref{prederivative} and \ref{prederivativeX});
		\item \( \p \) boundary operator (Definition \ref{eq:bd}); 
		\item \( \perp \)  perpendicular complement (Definition \ref{perp}); 
		\item \( \lozenge = \perp \p \perp \) coboundary operator (Definition \ref{coboundary});
		\item \(  \square = \lozenge \p + \p \lozenge \) geometric Laplace operator (Definition \ref{laplace});
		\item \( J \hat{\times} K \) Cartesian wedge product (Definition \ref{Cartesian2});
		\item \( M(J) \) mass of a differential chain (Definition \ref{def:masschain});  
		\item \( \{J_t\}_a^b \) evolving chain   \( \S\ref{sub:evolving_chains} \);
		\item \( [J_t]_a^b \) flowing chain  (Equation \eqref{FTCPlus}).
		
	\end{itemize}

\addcontentsline{toc}{section}{References} 
\bibliography{Jennybib.bib, bibliography.bib}{}
\bibliographystyle{amsalpha}

\bibliographystyle{amsalpha} 

\end{document}